\documentclass[12pt]{amsart}
\usepackage{amsfonts, amssymb, latexsym, hyperref, xcolor, dynkin-diagrams, tikz,tikz-cd}
\usepackage{ytableau}
\usepackage{comment,todonotes}

\newtheorem{theorem}{Theorem}[section]
\newtheorem{proposition}[theorem]{Proposition}
\newtheorem{lemma}[theorem]{Lemma}

\newtheorem{claim}[theorem]{Claim}
\newtheorem*{claim*}{Claim}
\newtheorem{corollary}[theorem]{Corollary}
\newtheorem{Main Conjecture}[theorem]{Main Conjecture}

\theoremstyle{definition}
\newtheorem{definition}[theorem]{Definition}
\theoremstyle{remark}

\newtheorem{example}[theorem]{Example}

\newtheorem{remark}[theorem]{Remark}
\theoremstyle{plain}

\newcommand{\Red}{{\mathrm{Red}}}

\newcommand{\C}{\mathcal{C}}

\newcommand{\cellsize}{12}
\newlength{\cellsz} 
\newsavebox{\cell}
\sbox{\cell}{\begin{picture}(\cellsize,\cellsize)
\put(0,0){\line(1,0){\cellsize}}
\put(0,0){\line(0,1){\cellsize}}
\put(\cellsize,0){\line(0,1){\cellsize}}
\put(0,\cellsize){\line(1,0){\cellsize}}
\end{picture}}
\newcommand\cellify[1]{\def\thearg{#1}\def\nothing{}%
\ifx\thearg\nothing
\vrule width0pt height\cellsz depth0pt\else
\hbox to 0pt{\usebox{\cell} \hss}\fi%
\vbox to \cellsz{
\vss
\hbox to \cellsz{\hss$#1$\hss}
\vss}}
\newcommand\tableau[1]{\vtop{\let\\\cr
\baselineskip -16000pt \lineskiplimit 16000pt \lineskip 0pt
\ialign{&\cellify{##}\cr#1\crcr}}}

\def\stackrel#1#2{\mathrel{\mathop{#2}\limits^{#1}}}
\newcommand{\excise}[1]{}

\begin{document}
\pagestyle{plain}
\title{Classification of Levi-spherical Schubert varieties}
\author{Yibo Gao}
\address{Department of Mathematics, Massachusetts Institute of Technology, Cambridge, MA 02139, USA}
\email{gaoyibo@mit.edu}
\author{Reuven Hodges}
\author{Alexander Yong}
\address{Department of Mathematics, U.~Illinois at Urbana-Champaign, Urbana, IL 61801, USA} 
\email{rhodges@illinois.edu, ayong@illinois.edu}
\date{April 20, 2021}
\maketitle

\begin{abstract}
A Schubert variety in the complete flag manifold $GL_n/B$ is \emph{Levi-spherical} if the action of a Borel subgroup in a Levi subgroup of a standard
parabolic has a dense orbit. We give a combinatorial classification of these Schubert varieties. This establishes a conjecture
of the latter two authors, and a new formulation in terms of standard Coxeter elements. Our proof uses the theory of \emph{key polynomials} (type $A$ Demazure module characters).
\end{abstract}

\section{Introduction}\label{sec:1}

This is a sequel to \cite{Hodges.Yong1}, which studies Schubert varieties that 
are spherical for the action of a Levi subgroup. That paper defines the notion of a \emph{spherical element}
of a Coxeter group. For symmetric groups, 
it was conjectured \cite[Conjecture~3.2]{Hodges.Yong1} that these elements characterize Levi-spherical
Schubert varieties in the complete flag manifold $GL_n/B$. We prove this conjecture using a new, concise formulation of spherical elements in terms of standard Coxeter elements. This provides the first (proved)
combinatorial characterizations of said geometric property for these Schubert varieties.

\subsection{Main idea}
Let $G=GL_n$. Its Weyl group $W\cong {\mathfrak S}_n$ consists of permutations of $[n]:=\{1,2,\ldots,n\}$. Thus $W$ is generated, as a Coxeter group, by the simple transpositions
$S=\{s_i=(i \ i+1): 1\leq i\leq n-1\}$. 
The set of \emph{left descents} is
\[J(w)=\{j\in [n-1]: w^{-1}(j)>w^{-1}(j+1)\}\] 
($j\in J(w)$ if $j+1$ appears to the left of $j$ in $w$'s one-line notation).

Let $\ell(w)$ denote the \emph{Coxeter length} of $w$. For $w\in {\mathfrak S}_n$, 
\[\ell(w)=\#\{1\leq i<j\leq n: w(i)>w(j)\}\]
counts \emph{inversions} of $w$. 

A \emph{parabolic subgroup} $W_I$ of $W$ is the subgroup generated by a subset $I\subset S$. A \emph{standard
Coxeter element} $c\in W_I$ is any product of the elements of $I$ listed in some order. Let $w_0(I)$ be the longest element of $W_I$.

\begin{definition}\label{def:spherical-standard-coxeter}
Let $w\in W$ and fix $I\subseteq J(w)$. Then $w$ is $I$-\emph{spherical} if $w_0(I)w$ is a standard Coxeter element for some parabolic subgroup $W_I$ of $W$.
\end{definition}

Let $\Red(w)$ be the set of \emph{reduced expressions}
$w=s_{i_1}\cdots s_{i_{\ell(w)}}$. Let  $D:=[n-1]-I=\{d_1<d_2<\ldots<d_k\}; \ d_0:=0, d_{k+1}:=n$. 
In \cite{Hodges.Yong1}, another definition of $I$-spherical was given.

\begin{definition}[{\cite[Definition~3.1]{Hodges.Yong1}}]\label{def:original-spherical}
Let $w\in {\mathfrak S}_n$ and $I\subseteq J(w)$. Then $w$ is \emph{$I$-spherical} 
if 
$R=s_{i_1}s_{i_2}\cdots s_{i_{\ell(w)}}\in \Red(w)$
 exists such that
\begin{itemize}
\item[(S.1)] $s_{d_i}$ appears at most once in $R$; and
\item[(S.2)] $\#\{m:d_{t-1}<i_m< d_t\} < {d_{t}-d_{t-1}+1\choose 2}$   for $1\leq t\leq k+1$.
\end{itemize}
\end{definition}

Our starting point is the following result:

\begin{theorem}
\label{thm:equivalent}
Definitions~\ref{def:spherical-standard-coxeter}
and~\ref{def:original-spherical} are equivalent.
\end{theorem}

\subsection{Schubert geometry and the main result}
Let ${\rm Flags}({\mathbb C}^n)$ be the variety of complete flags 
$\langle 0\rangle \subset F_1 \subset F_2 \subset \cdots \subset
F_{n-1}\subset {\mathbb C}^n$, where $F_i$ is a subspace of dimension $i$. 
The group $GL_n$ of invertible $n\times n$ matrices over ${\mathbb C}$ acts transitively on ${\rm Flags}({\mathbb C}^n)$ by change of basis. The \emph{standard
flag} is defined by $F_i={\rm span}(\vec e_1, \vec e_2,\ldots, \vec e_i)$ where $\vec e_i$ is the $i$-th
standard basis vector. The stabilizer of this flag is $B\subset GL_n$, the Borel subgroup
of upper triangular invertible matrices. Hence ${\rm Flags}({\mathbb C}^n)\cong GL_n/B$. $B$ acts on $GL_n/B$ with finitely many orbits; these are the \emph{Schubert cells} $X_w^{\circ}=BwB/B\cong
{\mathbb C}^{\ell(w)}$
indexed by $w\in {\mathfrak S}_n$ (viewed as a permutation matrix). Their closures $X_w:=\overline{X_{w}^{\circ}}$
are the \emph{Schubert varieties}; these are of interest in algebraic geometry and
representation theory. A standard reference
is \cite{Fulton}. 

For $I\subseteq J(w)$, let $L_I\subseteq GL_n$ be the Levi subgroup of invertible block diagonal matrices
\[L_I\cong GL_{d_1-d_0}\times GL_{d_2-d_1}\times \cdots \times GL_{d_k-d_{k-1}} \times GL_{d_{k+1}-d_k}.\]
As explained in, \emph{e.g.}, \cite[Section~1.2]{Hodges.Yong1}, $L_I$ acts on $X_w$.  
\begin{definition}
$X_w$ is \emph{$L_I$-spherical} if $X_w$ has a dense orbit of a Borel subgroup of $L_I$. If in addition, $I=J(w)$, $X_w$ is \emph{maximally spherical}.
\end{definition}

Our main result is a classification of $L_I$-spherical Schubert varieties:

\begin{theorem}[{\cite[Conjecture~3.2]{Hodges.Yong1}}]
\label{thm:main}
Let $w\in {\mathfrak S}_n$ and $I\subseteq J(w)$. $X_w\subseteq GL_n/B$ is $L_I$-spherical if and only if $w$ is $I$-spherical. 
\end{theorem}

Theorems~\ref{thm:equivalent} and \ref{thm:main} were used in  C.~Gaetz's \cite{Gaetz}, which proves \cite[Conjecture~3.8]{Hodges.Yong1}. Consequently, this gives a pattern avoidance criterion for maximally
spherical Schubert varieties \cite[Theorem~1.4, Corollary~1.5]{Gaetz}. We refer to \cite{Hodges.Yong1} for further references and background.

\subsection{Organization}
In Section~\ref{sec:2} we prove Theorem~\ref{thm:equivalent}; in the process, we establish a root-system uniform result (Proposition~\ref{lem:witness-start-w0}) that shows Definition~\ref{def:spherical-standard-coxeter} and Definition~\ref{def:spherical-original} from \cite{Hodges.Yong1} (a generalization of Definition~\ref{def:original-spherical}) are, in some sense,
``close''. Section~\ref{sec:3} introduces
some notation and terminology about symmetric groups, Bruhat order, and a certain
poset ${\mathcal S}_{I,\gamma}$ that we define.
In Section~\ref{sec:4} we recall notions about
\emph{key polynomials}, split-symmetry, and multiplicity-freeness from \cite{Hodges.Yong1}. This connects the Coxeter
combinatorics to the geometry. Thereby, Theorem~\ref{thm:main} is equivalent to
Theorem~\ref{thm:maingoal}, the form we prove.
In Section~\ref{sec:5} we introduce a subposet ${\mathcal P}_{c\lambda, \gamma}$ of ${\mathcal S}_{I,\gamma}$ whose
main feature is the ``Diamond property'' (Theorem~\ref{thm:diamond}). Assuming this property, we
also prove the ``$\Rightarrow$'' direction of Theorem~\ref{thm:maingoal}. 
The central observation is that ${\mathcal P}_{c\lambda, \gamma}$ is poset isomorphic to an interval in the Bruhat order of a Young subgroup (Proposition~\ref{prop:interval}). This permits us to reduce ``$\Rightarrow$'' to basics about the M\"obius function of Bruhat order \cite{Deodhar}. 
 Theorem~\ref{thm:diamond} is proved in Section~\ref{sec:6}. Finally, the ``$\Leftarrow$'' direction of
Theorem~\ref{thm:maingoal} is established in Section~\ref{sec:7}.

\section{Proof of Theorem~\ref{thm:equivalent}}\label{sec:2}

We first derive some results valid for any finite crystallographic root system $\Phi$. Let the
positive roots be $\Phi^+$, with simple roots $\Delta=\{\alpha_1,\ldots,\alpha_r\}$. Let $W$ be its finite Weyl group with corresponding simple generators $S=\{s_1,s_2,\ldots,s_r\}$, where we have fixed a bijection of $[r]:=\{1,2,\ldots,r\}$ with the nodes of the Dynkin 
diagram ${\mathcal G}$.  Let ${\Red}(w)$ be the set of the \emph{reduced expressions} $w=s_{i_1}\cdots s_{i_k}$, where $k=\ell(w)$ is the \emph{Coxeter length} of $w$.  The \emph{left descents} of $w$ are 
\[J(w)=\{j \in [r]: \ell(s_j w)<\ell(w)\}.\]

For $I\in 2^{[r]}$, let ${\mathcal G}_I$ be the induced subdiagram of ${\mathcal G}$.  Write
\begin{equation}
\label{eqn:thedecompabc}
{\mathcal G}_I=\bigcup_{z=1}^m {\mathcal C}^{(z)}
\end{equation}
as its decomposition into connected components. Let $w_0^{(z)}$ be the longest element of the parabolic subgroup
$W_{I^{(z)}}$ generated by $I^{(z)}=\{s_j: j\in {\mathcal C}^{(z)}\}$. 
The generalization of Definition~\ref{def:original-spherical} to general-type was given as follows:

\begin{definition}\label{def:spherical-original}
Let $w\in W$ and fix $I\subset J(w)$. Then $w$ is $I$-\emph{spherical} if there exists $R=s_{i_1}\cdots s_{i_{\ell(w)}}\in\Red(w)$ such that
\begin{itemize}
    \item $\#\{t\:|\: i_t=j\}\leq1$ for all $j\in[r]-I$, and
    \item $\#\{t\:|\: i_t\in\C^{(z)}\}\leq\ell(w_0^{(z)})+\#\text{vertices}(\C^{(z)})$ for $1\leq z\leq m$.
\end{itemize}
Such an $R$ is called an $I$-\emph{witness}.
\end{definition}

Definition~\ref{def:spherical-standard-coxeter} makes sense in the general context as well. However, that notion differs from 
Definition~\ref{def:spherical-original} in type $D_4$ and $F_4$ (this reduces confidence in
the general-type classification conjecture for Levi-spherical Schubert varieties \cite[Conjecture~1.9]{Hodges.Yong1}). 
We plan to study this further in future work.
%
%
%
%
%
%
%
%
%
%
%
%
%
%
%
%
%
%
%
%
%
%
%
%
%

We now develop some preliminary results.

\begin{lemma}\label{lem:independent-si-component}
Let $w\in W$ and fix $I\subset J(w)$. Let 
$R=s_{i_1}\cdots s_{i_{\ell(w)}}$ and $R'=s_{j_1}\cdots s_{j_{\ell(w)}} \in \Red(w)$
be such that each $s_t$, $t\in[r]-I$, appears at most once in $R$, and at most once in $R'$. Then for each $1\leq z\leq m$, 
\[\#\{t\:|\: i_t\in\mathcal{C}^{(z)}\}=\#\{t\:|\: j_t\in\mathcal{C}^{(z)}\}.\]
\end{lemma}
\begin{proof}
We may assume without loss of generality that $\Phi$ is irreducible. Furthermore,
we may assume without loss of generality that each $s_i\in S$ is used in any (equivalently, all) $R''\in \Red(w)$, since otherwise we work individually on the root systems associated to each
irreducible component of $\Delta\setminus \{\alpha_i\}$.

 We induct on $m\geq 1$. In the base case $m=1$, then 
\[\#\{t\:|\: i_t\in\mathcal{C}^{(1)}\}=\ell(w)-(r-\# I)\] 
is independent of any choice of $R''$, so we are done.

For the induction step, consider a fixed $\mathcal{C}\in\{\mathcal{C}^{(1)},\ldots,\mathcal{C}^{(m)}\}$. Fix some $t_0\in[r]-I$ such that not all of $\mathcal{C}^{(1)},\ldots,\mathcal{C}^{(m)}$ lie in the same connected component of (the Dynkin diagram of) $S\setminus\{t_0\}$. Such $t_0$ can be chosen because $m\geq2$ and that the Dynkin diagram for $W$ is a tree. Let $J_1,J_2,\ldots,J_p$ be the connected components of $S\setminus\{t_0\}$ and assume $\mathcal{C}\subset J_1$.

Note that generators in different $J_i$'s commute with each other. For the reduced word $R$, we can regroup it as $w_{J_1}\cdots w_{J_p}s_{t_0}u_{J_1}\cdots u_{J_p}$ where $w_{J_i},u_{J_i}\in W_{J_i}$, the parabolic subgroup generated by $J_i$. We can rearrange it as 
\[
w=(w_{J_2}\cdots w_{J_p})(w_{J_1}s_{t_0}u_{J_1})(u_{J_2}\cdots u_{J_p}).
\]
Similarly, for $R'$ we obtain 
\begin{align*}
w & = w_{J_1}'\cdots w_{J_p}'s_{t_0}u_{J_1}'\cdots u_{J_p}'\\
\ & = (w_{J_2}'\cdots w_{J_p}')(w_{J_1}'s_{t_0}u_{J_1}')(u_{J_2}'\cdots u_{J_p}').
\end{align*}
Since $w_{J_1}s_{t_0}u_{J_1}$ does not contain any simple generators associated to $K=J_2\cup\cdots\cup J_p$, it is the unique minimal double coset representative of $W_KwW_K$. This implies that $w_{J_1}s_{t_0}u_{J_1}=w_{J_1}'s_{t_0}u_{J_1}'$, where we obtained the same Weyl group element from different reduced decompositions.

Now apply the induction hypothesis by replacing $S$ by $J_1\cup\{t_0\}$, $I$ by $I\cap J_1\cup\{t_0\}$, $w$ by the minimal length coset representative of $W_KwW_K$, $R$ (and $R'$) by the subword of $R$ (and $R'$) that equals $w_{J_1}s_{t_0}u_{J_1}=w_{J_1}'s_{t_0}u_{J_1}'$, and leaving $\mathcal{C}$ unchanged. 
\end{proof}

For each $\alpha\in\Phi^+$, define its \emph{support} to be
\[{\mathrm{Supp}}(\alpha)=\{\alpha_i\in\Delta\:|\: \alpha-\alpha_i \text{\ is a nonnegative linear combination of $\Delta$}\}.\]
Also, for each positive root $\alpha=\sum_{i=1}^r c_i\alpha_i$, written as a nonnegative linear combination of $\Delta$, define its \emph{height} to be $\mathrm{ht}(\alpha)=\sum_{i=1}^r c_i$. The next folklore result is well-known,
but we do not know a precise reference with proof. We include one here:

\begin{lemma}\label{lem:support-connected}
For each $\alpha\in\Phi^+$, $\mathrm{Supp}(\alpha)$ is a connected subgraph in the Dynkin diagram.
\end{lemma}
\begin{proof}
We use induction on $\mathrm{ht}(\alpha)$. The base case $\mathrm{ht}(\alpha)=1$, \emph{i.e.}, $\alpha\in\Delta$, is clear. 

In the induction step, for each $\alpha\in\Phi^+\setminus\Delta$, there exists $i\in[r]$ such that $\alpha':=s_i\alpha=\alpha-k\alpha_i\in\Phi^+$ for some positive integer $k$. We know that $\mathrm{ht}(\alpha')<\mathrm{ht}(\alpha)$ so $\mathrm{Supp}(\alpha')$ is connected by induction hypothesis. At the same time, $\mathrm{Supp}(\alpha)=\mathrm{Supp}(\alpha')\cup\{\alpha_i\}$. If $\alpha_i\in\mathrm{Supp}(\alpha')$, then $\mathrm{Supp}(\alpha)=\mathrm{Supp}(\alpha')$ is connected. Thus, we assume $\alpha_i\notin\mathrm{Supp}(\alpha')$. Let $\langle -,-\rangle$ denote the standard inner product on the ambient vector space containing our root system. We have
\[\alpha=s_i\alpha'=\alpha'-\frac{2\langle \alpha',\alpha_i\rangle}{\langle\alpha_i,\alpha_i\rangle}\alpha_i\neq\alpha'.\]
As $\langle \alpha',\alpha_i\rangle\neq0$, there exists some $\alpha_j\in\mathrm{Supp}(\alpha')$ such that $\langle\alpha_j,\alpha_i\rangle\neq0$, meaning that the node $j$ is connected to the node $i$ in the Dynkin diagram. Therefore, $\mathrm{Supp}(\alpha)=\mathrm{Supp}(\alpha')\cup\{\alpha_i\}$ is connected.
\end{proof}

\begin{lemma}\label{lem:root-path}
Suppose that we have an equality of reduced words 
$s_{i_1}s_{i_2}\cdots s_{i_{k-1}}=s_{i_2}s_{i_3}\cdots s_{i_k}$. 
Then $\#\{t\:|\: i_t=j\}\geq2$ for all $j$ on the path (excluding $i_1$ and $i_k$) between ${i_1}$ and ${i_k}$ in ${\mathcal G}$.
\end{lemma}
\begin{proof}
Let $w=s_{i_1}s_{i_2}\cdots s_{i_{k-1}}=s_{i_2}s_{i_3}\cdots s_{i_k}$. As $s_{i_2}\cdots s_{i_k}$ is reduced, $\alpha_{i_k}$ is a right inversion of $w$, where $\alpha_{i_k}$ is the simple root corresponding to $s_{i_k}$, \emph{i.e.}, $\alpha_{i_k}\in\Phi^{+}$ and $w\alpha_{i_k}\in\Phi^{-}$. Let $-\beta=w\alpha_{i_k}$ so $\beta\in\Phi^+$. We have that 
\[\beta=-s_{i_2}\cdots s_{i_k}\alpha_{i_k}=-s_{i_2}\cdots s_{i_{k-1}}(-\alpha_{i_k})=s_{i_2}\cdots s_{i_{k-1}}\alpha_{i_k}.\] 
This means that $s_{i_1}\beta=w\alpha_{i_k}=-\beta$ so $\beta=\alpha_{i_1}$.

Note that since $s_{i_j}\cdots s_{i_k}$ is reduced and has $\alpha_{i_k}$ as its right descent, we know 
\[s_{i_j}\cdots s_{i_{k-1}}s_{i_k}\alpha_{i_k}\in\Phi^-, \ s_{i_j}\cdots s_{i_{k-1}}\alpha_{i_k}\in\Phi^+.\]

Consider the sequence of positive roots 
\[\alpha_{i_k},s_{i_{k-1}}\alpha_{i_k},\ldots,s_{i_2}\cdots s_{i_{k-1}}\alpha_{i_k}=\alpha_{i_1}.\]
By definition, $s_{\alpha}(x)=x-\frac{2\langle x,\alpha\rangle}{\langle \alpha,\alpha\rangle}\alpha$.
Hence the symmetric difference
\[{\mathrm{Supp}}(s_{i_t}\cdots s_{i_{k-1}}\alpha_{i_k})\ \Delta \ {\mathrm{Supp}}(s_{i_{t+1}}\cdots s_{i_{k-1}}\alpha_{i_k})\subseteq \{\alpha_{i_t}\}, \text{ \  for $t=k-1,\ldots,2$.}\]
Recall that for each $\alpha\in\Phi^+$, its support $\mathrm{Supp}(\alpha)$ is connected in the Dynkin diagram (Lemma~\ref{lem:support-connected}). Fix any $j$ on the path between $i_1$ and $i_k$ in the Dynkin diagram. As a result, there exists some $p$ such that
$\alpha_{j}\in {\mathrm{Supp}}(s_{i_p}\cdots s_{i_{k-1}}\alpha_{i_k})$. Thus, there must be some $s_j$ among $s_{i_p},\ldots,s_{i_{k-1}}$ so that a positive multiple of $\alpha_j$ can be added from $\alpha_{i_k}$, and there must be some $s_j$ among $s_{i_2},\ldots,s_{i_{p-1}}$ so that a positive multiple of $\alpha_j$ can be subtracted to obtain $\alpha_{i_1}$.
\end{proof}

We use this textbook result:
\begin{proposition}[{Deletion property \cite[Proposition~1.4.7]{Bjorner.Brenti}}]\label{lem:deletion-property}
Let $w=s_{i_1}\cdots s_{i_{\ell}}$ be a reduced word. Then for a left descent $s_{i_0}$ of $w$, i.e. $\ell(s_{i_0}w)=\ell(w)-1$, we have another reduced word $w=s_{i_0}s_{i_1}\cdots \widehat{s_{i_j}}\cdots s_{i_{\ell}}$, where $\widehat{s_{i_j}}$ means the deletion of $s_{i_j}$.
\end{proposition}

The culmination of the above root-system uniform arguments is this next proposition, which says that
Definition~\ref{def:spherical-original} is, in general, ``close'' to Definition~\ref{def:spherical-standard-coxeter}.

\begin{proposition}\label{lem:witness-start-w0}
If $w\in W$ is $I$-spherical (in the sense of Definition~\ref{def:spherical-original}), then there exists an $I$-witness $R$ of $w$ of the form
$R=R' R''$ where $R'\in \Red(w_0(I))$ and $R''\in \Red(w_0(I)w)$.
\end{proposition}
\begin{proof}
Let $R^{(0)}=s_{i_1}\cdots s_{i_{\ell}}$ be an $I$-witness of $w$. Pick any 
$R'=s_{k_1}\cdots s_{k_{\ell'}}\in \Red(w_0(I))$.
We gradually modify $R^{(0)}$, so that at each step it remains an $I$-witness, until it is of the desired form. For each $j=\ell',\ldots,1$, add $s_{k_j}$ to the start of $R$. By the deletion property (Proposition~\ref{lem:deletion-property}), some $s_{i_{j'}}$ is deleted resulting in $R^{(1)}\in \Red(w)$. By Lemma~\ref{lem:root-path}, $k_j$ and $i_{j'}$ must be in the same $\mathcal{C}^{(z)}$ since otherwise, some $s_i$ with $i\notin I$ on the path from $k_j$ to $i_{j'}$ in the Dynkin diagram is used at least twice in $R^{(0)}$, contradicting that $R^{(0)}$ is an $I$-witness. Thus, in $R^{(1)}$, $\#\{t\:|\: i_t\in\mathcal{C}^{(z)}\}$ remains unchanged for each $z$.
Repeating this, $k_{\ell'}$ many times, we obtain an $I$-witness $R^{(k_{\ell'})}=R'R''$, as claimed.
\end{proof}

Henceforth, we assume that $W={\mathfrak S}_n$. Recall that $w\in {\mathfrak S}_n$ \emph{contains the pattern} $u\in {\mathfrak S}_k$ if there exists $i_1<i_2<\ldots<i_k$ such that $w(i_1), w(i_2),\ldots,w(i_k)$ is in the same relative order as $u(1),u(2),\ldots,u(k)$. Furthermore $w$ \emph{avoids} $u$ if no such indices exist.

We need the following proposition relating pattern avoidance and standard Coxeter elements. A more general statement for finite Weyl groups can be found in \cite{GH20}.

\begin{proposition}[\cite{Tenner}]\label{prop:Sn-coxeter-pattern}
A permutation $w\in\mathfrak{S}_n$ is a product of distinct generators, \emph{i.e.,} a standard Coxeter element in some parabolic subgroup, if and only if $w$ avoids 321 and 3412.
\end{proposition}

\noindent
\emph{Conclusion of the proof of Theorem~\ref{thm:equivalent}:} If $w\in {\mathfrak S}_n$ satisfies the Definition~\ref{def:spherical-standard-coxeter} then it clearly satisfies Definition~\ref{def:original-spherical}.

Conversely, suppose $w\in {\mathfrak S}_n$ satisfies Definition~\ref{def:original-spherical}. We now show that it satisfies Definition~\ref{def:spherical-standard-coxeter}. Recall
$D=[n]-I=\{d_1 < d_2< \ldots < d_k\}; d_0=0, d_{k+1}=n$.
Let 
\[A_i=\{d_{i-1}+1,\ldots,d_i\} \text{ \ for $i=1,\ldots,k+1$.}\]
Assume $w$ is $I$-spherical with some $I$-witness. By Proposition~\ref{lem:witness-start-w0} and Definition~\ref{def:original-spherical}, we can write $w=w_0(I)u$ such that there is a reduced word $R''=s_{i_1}\cdots s_{i_{\ell(u)}}$ of $u$ such that
\begin{itemize}
\item $s_{d_i}$ appears at most once in $R''$; and 
\item $\#\{m\:|\: d_{t-1}<i_m<d_t\}<{{d_t-d_{t-1}+1}\choose2}-{{d_t-d_{t-1}}\choose2}=d_t-d_{t-1}$ for $1\leq t\leq k+1$.
\end{itemize}

By Proposition~\ref{prop:Sn-coxeter-pattern}, it suffices to show that $u=w_0(I)\cdot w$ avoids $321$ and $3412$, or equivalently, $u^{-1}$ avoids $321$ and $3412$. 
Since Proposition~\ref{lem:witness-start-w0} implies $\ell(w)=\ell(w_0(I))+\ell(u)$,
$u=w_0(I)\cdot w$ does not have left descents in $I$. In other words, $u^{-1}$ is increasing on
the indices $A_i$ for $1\leq i\leq k+1$.

Think about $R''$ as successive multiplications of $u^{-1}$ on the right by simple transpositions of $R''$ (read right to left) until one reaches $\mathrm{id}$ (for example, if $u^{-1}=2413$, 
$R''=s_1 s_3 s_2$ represents $2\underline{41}3\to 21\underline{43}\to \underline{21}34\to 1234$). Since $s_{d_i}$ appears at most once in $R''$, we know $|\{u^{-1}(1),u^{-1}(2),\ldots,u^{-1}(d_i)\}\setminus [d_i]|\leq 1$. Moreover, if this cardinality is $1$, $s_{d_i}$ swaps $\max\{u^{-1}(1),\ldots,u^{-1}(d_i)\}$ at index $d_i$ with $\min\{u^{-1}(d_i+1),\ldots,u^{-1}(n)\}$ at index $d_{i}+1$.

First suppose $u^{-1}$ contains $3412$ at indices $k_1<k_2<k_3<k_4$.
Then any reduced expression of $u^{-1}$ contains at least two copies of $s_j$ for $k_2\leq j<k_3$. Since $u^{-1}(k_2)>u^{-1}(k_3)$, $k_2$ and $k_3$ lie in different $A_i$'s. This means that there exists some $k_2\leq j<k_3$ with $j\notin I$ such that $s_j$ is used at least twice in $R''$, a contradiction. 

If $u^{-1}$ contains $321$ at indices $k_1<k_2<k_3$ with $k_i\in A_{t_i}$, then $t_1<t_2<t_3$. We concentrate on the block $A_{t_2}$ and will show that simple transpositions in $A_{t_2}$ are used at least $d_{t_2}-d_{t_2-1}$ times in $R''$. A visualization of $u^{-1}$ is shown in Figure~\ref{fig:equivalent-definition-321}.
\begin{figure}[h!]
\begin{tikzpicture}
\node at (1,0.5) {$\bullet$};
\node at (2,1) {$\bullet$};
\node at (3,3.5) {$\bullet$};
\node at (4,1.5) {$\bullet$};
\node at (5,2) {$\bullet$};
\node at (6,3) {$\bullet$};
\node at (7,2.5) {$\bullet$};
\node at (8,4) {$\bullet$};
\draw[dotted](3.5,0)--(3.5,4.5);
\draw[dotted](6.5,0)--(6.5,4.5);
\node at (3,0) {$k_1$};
\node at (6,0) {$k_2$};
\node at (7,0) {$k_3$};
\node at (2,4.5) {$t_1$};
\node at (5,4.5) {$t_2$};
\node at (7.5,4.5) {$t_3$};
\node at (-0.5,4.5) {block};
\end{tikzpicture}
\caption{An example of a permutation $u^{-1}$ containing $321$ in the proof of Theorem~\ref{thm:equivalent}}
\label{fig:equivalent-definition-321}
\end{figure}
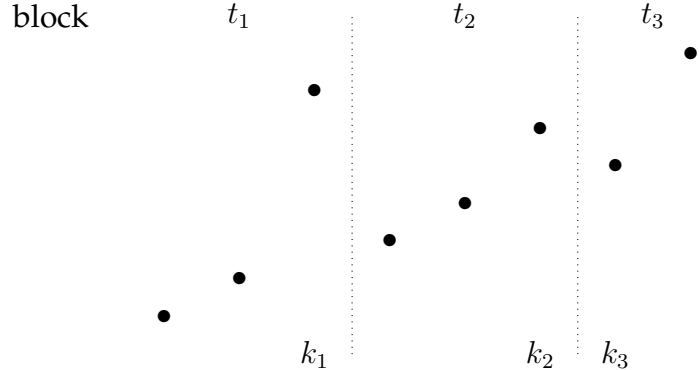

Recall that $s_{d_{t_2-1}}$ exchanges the maximum value in indices $A_1\cup\cdots\cup A_{t_2-1}$ with the minimum value in indices $A_{t_2}\cup\cdots\cup A_{k+1}$. Since $u^{-1}(k_2)>u^{-1}(k_3)$, the value $u^{-1}(k_2)$ is not the minimum among $u^{-1}(A_{t_2}\cup\cdots\cup A_{k+1})$ and thus cannot arrive left of index $d_{t_2-1}+1$ during this $s_{d_{t_2-1}}$ swap. Similarly, since $u^{-1}(k_2)<u^{-1}(k_1)$, the value $u^{-1}(k_2)$ cannot go to the right of index $d_{t_2}-1$. As a result, the value of $u^{-1}(k_2)$ occurs among $u^{-1}(A_{t_2})$ as we are using $R''$ to transform $u^{-1}$ into $\mathrm{id}$. 

In order to put $u^{-1}(k_1),u^{-1}(k_2),u^{-1}(k_3)$ into the correct order, both the values $u^{-1}(k_1)$ and $u^{-1}(k_3)$ must enter $A_{t_2}$ and exchange with $u^{-1}(k_2)$. In particular, all of the simple transpositions $s_j$, $j=d_{t_2-1}+1,\ldots,d_{t_2}-1$ must be used in order to exchange $u^{-1}(k_1)$ with $u^{-1}(k_3)$. Moreover, certain $s_j$ need to be applied twice: if $u^{-1}(k_1)$ switches with $u^{-1}(k_2)$ at transposition $s_j$ before $u^{-1}(k_2)$ switches with $u^{-1}(k_3)$, then $s_j$ must be used again; and if $u^{-1}(k_3)$ switches with $u^{-1}(k_2)$ first at $s_j$, then $s_j$ must be used again as well to eventually switch $u^{-1}(k_2)$ and $u^{-1}(k_1)$. Either way, in this case, the total number of times that $s_j$, $j=d_{t_2-1}+1,\ldots,d_{t_2}-1$, is used is at least $d_{t_2}-d_{t_2-1}$. \qed

\section{Bruhat order of Young subgroups and the poset ${\mathcal S}_{I,\gamma}$}\label{sec:3}

The symmetric group ${\mathfrak S}_n$ has the poset structure of
\emph{(strong) Bruhat order} $<_{\text{Bruhat}}$. It is convenient for us to use the ``upside down'' version. That is, the
covering relations are $u<_{\text{Bruhat}} u s_{ij}$ where $\ell(u)-1=\ell(us_{ij})$ and $s_{ij}=(i\ j )$ is a transposition.
Hence, under this choice of convention, the longest length permutation  
$w_0= n \ n-1 \ \ldots \ 3\ 2 \ 1$ is the unique minimum, and the identity permutation is the unique maximum.

A sequence of non-negative integers $\alpha=(\alpha_1,\alpha_2,\ldots,\alpha_n)$ is a \emph{weak composition}. Let ${\sf Comp}_n$ be the set of all such compositions. 
Let  ${\sf Par}_t$ be the set of partitions with at most $t$ nonzero-parts. A \emph{split-partition} is
\[(\lambda^1,\ldots,\lambda^k)\in {\sf Par}_D:={\sf Par}_{d_1-d_0}\times \cdots \times{\sf Par}_{d_{k+1}-d_k}.\]

Fix $\gamma\in {\sf Par}_D$, where $D=[n-1]-I$ (as in Section~\ref{sec:1}), which we will identify
(in the obvious way) with an element of ${\sf Comp}_n$.

\begin{definition}\label{def:sameblock}
$i,j\in [n]$ are \emph{in the same block} (with respect to $D=[n]-I$) if there exists $t\in [0,k]$ such that $d_{t}+1\leq i,j\leq d_{t+1}$.
\end{definition}

Let $\delta_t=(t,t-1,\ldots,3,2,1)$. Given $\gamma$, pick $\Delta:=\Delta_{\gamma}\in {\mathbb Z}_{\geq 0}^n$ to be any fixed but arbitrary strictly decreasing vector such that:
\begin{itemize} 
\item In the $i$-th block (of size $d_{i}-d_{i-1}$), the components of $\Delta$ are of the form $(f_i,f_i,\ldots, f_i)+\delta_{d_i -d_{i-1}}$ where $f_i$ is some positive integer depending on $i$.
\item $\gamma+\Delta$ is a vector with distinct components.
\end{itemize}
Let ${\widehat {\mathfrak S}}_n$ be permutations on the (distinct) entries of $\gamma+\Delta$. Clearly there is an isomorphism of Bruhat orders between
that of ${\mathfrak S}_n$ and ${\widehat {\mathfrak S}}_n$ that sends $w_0$ to $\Delta+\gamma$.
We will therefore mildly abuse notation and use $<_{\text{Bruhat}}$ for either order, as the context
will be clear. Let 
\[\Omega:({\mathfrak S}_n,<_{\text{Bruhat}})\to ({\widehat {\mathfrak S}}_n,
<_{\text{Bruhat}})\] 
be this poset isomorphism.

Now, let 
\[\tilde {\mathcal S}_{I,\gamma}= {\widehat {\mathfrak S}}_{d_1-d_0}\times {\widehat {\mathfrak S}_{d_2-d_1}}\times\cdots\times
{\widehat {\mathfrak S}}_{d_{k+1}-d_k}\] 
be the Young subgroup of ${\widehat {\mathfrak S}}_n$, where ${\widehat {\mathfrak S}}_{d_{i+1}-d_i}$ is the permutation group on the labels of $\Delta+\gamma$ in the $i$-th block. Thus, strong Bruhat order $<_{\text{Bruhat}}$ on ${\widehat {\mathfrak S}}_n$ restricts to  $\tilde{\mathcal S}_{I,\gamma}$. 

\begin{definition}
Given $\tilde\beta\in \tilde {\mathcal S}_{I,\gamma}$ (thought of as a vector in ${\mathbb Z}_{\geq 0}^n$), let 
\[\Phi(w)={\tilde\beta}-\Delta.\]
\end{definition}

 Let ${\mathcal S}_{I,\gamma}:={\text{Im}} \ \Phi \subset {\sf Comp}_n$. For $x,y\in {\mathcal S}_{I,\gamma}$ define $x<_{\text{Bruhat}} y$ if $\Phi^{-1}(x)<_{\text{Bruhat}} \Phi^{-1}(y)$.
 
\begin{proposition}\label{prop:posetsiso}
$({\mathcal S}_{I,\gamma},<_{\text{\emph{Bruhat}}})\cong (\tilde{\mathcal S}_{I,\gamma},<_{\text{\emph{Bruhat}}})
\cong ({\mathfrak S}_{d_1-d_0}\times\cdots \times {\mathfrak S}_{d_{k+1}-d_k},<_{\text{\emph{Bruhat}}})$. 
\end{proposition}
\begin{proof}
$\Phi$ is injective and hence a bijection onto its image. It is a poset map by construction. This 
proves the first isomorphism. The second isomorphism is induced from $\Omega$.
\end{proof}

\begin{definition}
If $\beta\!=\!(\beta_1,\ldots,\beta_n)\!\in\! {\sf Comp}_n$ and $i\!<\!j\in [n-1]$, define 
$t_{ij}\!:\!{\sf Comp}_n\!\to\! {\sf Comp}_n$ 
by 
\begin{equation}
\label{tijdef}
t_{ij}(\ldots,\beta_i,\ldots,\beta_j,\ldots)=
(\ldots,\beta_j-(j-i),\ldots,\beta_i+(j-i),\ldots).
\end{equation}
Also let $t_{i}:=t_{i\ i{+}1}$.
\end{definition}

The next lemma asserts that the role of $t_{ij}$'s in ${\mathcal S}_{I,\gamma}$ is the same as that of the $s_{ij}=(i\ j)$ in ${\mathfrak S}_n$. In particular, the $t_i$'s are analogous to the simple transpositions. 

\begin{lemma}\label{lemma:commutes}
For $i<j$ in the same block, this diagram commutes:
\begin{equation}
\begin{array}{ccc}
\tilde{\mathcal S}_{I,\gamma} & \stackrel{\Phi}{\longrightarrow} &  {\mathcal S}_{I,\gamma} \\
s_{ij} \downarrow & & \downarrow t_{ij} \\
\tilde{\mathcal S}_{I,\gamma}  & \stackrel{\Phi}{\longrightarrow} & {\mathcal S}_{I,\gamma}.
\end{array}
\end{equation}
\end{lemma}
\begin{proof}
Let $\tilde{\beta}\in\tilde{\mathcal S}_{I,\gamma}$. By definition of $\Delta$, there is some number $f$ such that $\Delta_k=f-k$ for $i\leq k\leq j$. We have
\begin{align*}
t_{ij}\Phi\tilde{\beta}=&t_{ij}(\ldots,\tilde{\beta}_i-f+i,\ldots,\tilde{\beta}_k-f+k,\ldots,\tilde{\beta}_j-f+j,\ldots)\\
=&(\ldots,\tilde{\beta}_j-f+j-(j-i),\ldots,\tilde{\beta}_k-f+k,\ldots,\tilde{\beta}_i-f+i+(j-i),\ldots)\\
=&(\ldots,\tilde{\beta}_j-f+i,\ldots,\tilde{\beta}_k-f+k,\ldots,\tilde{\beta}_i-f+j,\ldots)\\
=&\Phi(\ldots,\tilde{\beta}_j,\ldots,\tilde{\beta}_k,\ldots,\tilde{\beta}_i,\ldots)=\Phi s_{ij}\tilde{\beta}
\end{align*}
as desired.
\end{proof}

\begin{example}
Let $n=3$, $I=\{1,2\}$ with a single block, $\gamma=443$ and $\Delta=321$. Figure~\ref{fig:example-S3} shows the poset $\tilde{\mathcal S}_{I,\gamma}$ and ${\mathcal S}_{I,\gamma}$ with the actions of $s_{ij}$'s and $t_{ij}$'s respectively.
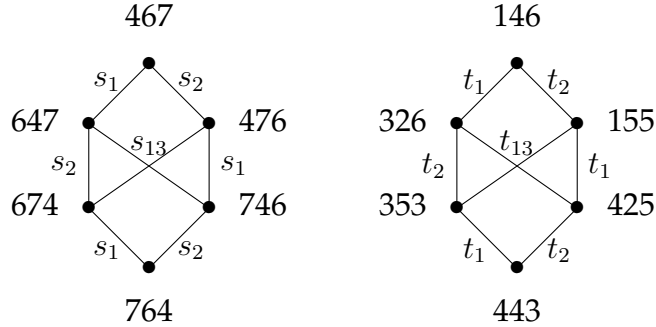
\begin{figure}[h!]
\begin{tikzpicture}[scale=0.8]
\node[label=below:764] at (0,0) {$\bullet$};
\node[label=left:674] at (-1,1) {$\bullet$};
\node[label=right:746] at (1,1) {$\bullet$};
\node[label=left:647] at (-1,2.4) {$\bullet$};
\node[label=right:476] at (1,2.4) {$\bullet$};
\node[label=above:467] at (0,3.4) {$\bullet$};
\draw(0,0)--(-1,1)--(-1,2.4)--(0,3.4)--(1,2.4)--(1,1)--(0,0);
\draw(-1,1)--(1,2.4);
\draw(1,1)--(-1,2.4);
\node at (-0.7,0.3) {$s_1$};
\node[left] at (-1,1.7) {$s_2$};
\node at (-0.7,3.1) {$s_1$};
\node at (0.7,0.3) {$s_2$};
\node[right] at (1,1.7) {$s_1$};
\node at (0.7,3.1) {$s_2$};
\node[above] at (0,1.7) {$s_{13}$};
\end{tikzpicture}
\qquad
\begin{tikzpicture}[scale=0.8]
\node[label=below:443] at (0,0) {$\bullet$};
\node[label=left:353] at (-1,1) {$\bullet$};
\node[label=right:425] at (1,1) {$\bullet$};
\node[label=left:326] at (-1,2.4) {$\bullet$};
\node[label=right:155] at (1,2.4) {$\bullet$};
\node[label=above:146] at (0,3.4) {$\bullet$};
\draw(0,0)--(-1,1)--(-1,2.4)--(0,3.4)--(1,2.4)--(1,1)--(0,0);
\draw(-1,1)--(1,2.4);
\draw(1,1)--(-1,2.4);
\node at (-0.7,0.3) {$t_1$};
\node[left] at (-1,1.7) {$t_2$};
\node at (-0.7,3.1) {$t_1$};
\node at (0.7,0.3) {$t_2$};
\node[right] at (1,1.7) {$t_1$};
\node at (0.7,3.1) {$t_2$};
\node[above] at (0,1.7) {$t_{13}$};
\end{tikzpicture}
\caption{Example of the poset $\tilde{\mathcal S}_{I,\gamma}$ (left) and ${\mathcal S}_{I,\gamma}$ (right)}
\label{fig:example-S3}
\end{figure}
\end{example}

\begin{remark}
Having formally defined $({\mathcal S}_{I,\gamma},<_{\text{Bruhat}})$ above, in the remainder of the paper, one can think of this poset as generated from $\gamma$ \emph{via} the action of $t_{ij}$'s, including just the~$t_i$'s.
\end{remark}

\begin{definition}\label{def:rank}
For $\beta\in{\mathcal S}_{I,\gamma}$, let $\theta(\beta)$ be the \emph{rank} of $\beta$, \emph{i.e.}, there exists a saturated chain 
\[\beta=\beta^{(\theta)}\gtrdot_{\text{Bruhat}}\beta^{(\theta-1)}\gtrdot_{\text{Bruhat}}\cdots\gtrdot_{\text{Bruhat}}\beta^{(0)}=\gamma\] 
of length $\theta=\theta(\beta)$ from $\beta$ to the minimum $\gamma$ in ${\mathcal S}_{I,\gamma}$. Also define the \emph{sign} of $\beta$ to be ${\sf sgn}(\beta):=(-1)^{\theta(\beta)}$.
\end{definition}

These facts follow immediately from the usual Bruhat orders and the isomorphism $\Phi$.

\begin{lemma}\label{lemma:posetBruhatorder}
For $\beta\in {\mathcal S}_{I,\gamma}$ and $i,j$ in the same block,
\begin{itemize}
\item[(i)] $\beta_i>\beta_j-(j-i)$ if and only if $\beta<_{\text{\emph{Bruhat}}} t_{ij}\beta$; in particular, $\beta_i-i\neq\beta_j-j$ for $i\neq j$;
\item[(ii)] ${\sf sgn}(t_{ij}\beta)=-{\sf sgn}(\beta)$;
\end{itemize}
\end{lemma}

\section{Polynomials and sphericality} \label{sec:4}

\subsection{Key polynomials} Let ${\sf Pol}:={\mathbb Z}[x_1,x_2,\ldots, x_n]$ be the 
polynomial ring in the indeterminates $x_1,x_2,\ldots, x_n$.
For $\alpha=(\alpha_1,\alpha_2,\ldots,\alpha_n) \in {\sf Comp}_n$, the \emph{key polynomial}
$\kappa_{\alpha}$ is defined as follows. If $\alpha$ is weakly decreasing, then $\kappa_{\alpha}:=\prod_i x_i^{\alpha_i}$. Otherwise, suppose $\alpha_i>\alpha_{i+1}$. Let
\[\pi_i:{\sf Pol}\to {\sf Pol}, \ \ f\mapsto \frac{x_i f(\ldots,x_i,x_{i+1},\ldots)-x_{i+1}f(\ldots,x_{i+1},x_i,\ldots)}{x_i-x_{i+1}},\]
and 
\[\kappa_{\alpha}=\pi_i(\kappa_{\widehat\alpha}) \text{\ where ${\widehat\alpha}:=(\alpha_1,\ldots,\alpha_{i+1},\alpha_i,\ldots)$.}\]

We need facts about the operators
$\pi_i$; our reference is \cite{Lascoux:polynomials}. The operators $\pi_i$ satisfy the relations 
\begin{align*}
\pi_i \pi_j & = \pi_j \pi_i \text{\ \ (for $|i-j|>1$)}\\ 
\pi_i \pi_{i+1} \pi_i & = \pi_{i+1}\pi_i \pi_{i+1}\\
\pi_i^2 & =\pi_i.
\end{align*} 

Recall that the \emph{Demazure product} on ${\mathfrak S}_n$ is defined by
\[w*s_i=\begin{cases}
ws_i & \text{if $\ell(ws_i)=\ell(w)+1$}\\
0 & \text{otherwise.}
\end{cases}.\]
This product is associative.
Then $R=(s_{i_1},\cdots, s_{i_{\ell}})$ is a \emph{Hecke word} of $w$ if $w=s_{i_1}*s_{i_2}*\cdots* s_{i_{\ell}}$.

For any $w\in {\mathfrak S}_n$ one unambiguously defines 
\[\pi_w:=\pi_{i_1}\pi_{i_2}\cdots\pi_{i_{\ell}},\]
where $R=(s_{i_1},\ldots, s_{i_{\ell}})$ is any Hecke word of $w$.

Now suppose $\lambda=(\lambda_1\geq \lambda_2\geq \ldots \geq \lambda_n)$ is a partition, and $w\in {\mathfrak S}_n$. Define
\[\kappa_{w\lambda}:=\kappa_{\lambda_{w^{-1}(1)},\ldots,\lambda_{w^{-1}(n)}}.\]
With this choice of convention, we have
\begin{equation}
\label{eqn:piwlambda}
    \kappa_{w\lambda}=\pi_w \kappa_{\lambda}.
\end{equation}

Below, we do not assume $\ell(w)=\ell(w_0(I))+\ell(c)$:
\begin{lemma}
If $w=w_0(I)c$ where $c$ is a standard Coxeter element, then $\kappa_{w\lambda}=\pi_{w_0(I)}\kappa_{c\lambda}$.
\end{lemma}
\begin{proof}
By two applications of (\ref{eqn:piwlambda}), and the definition of $\pi_w$
\[\kappa_{w\lambda}=\kappa_{w_0(I)c\lambda}=\pi_{w_0(I)c}(\kappa_{\lambda})=
\pi_{w_0(I)}\pi_c(\kappa_{\lambda})=\pi_{w_0(I)}\kappa_{c\lambda}.\qedhere\]
\end{proof}

For any $\alpha\in {\sf Comp}_n$, let
\[a_{\alpha_1+n-1,\alpha_2+n-2,\ldots,\alpha_n}:=\det(x_j^{\lambda_i+n-i})_{1\leq i,j\leq n}.\]
In particular, 
\[\Delta_n:=a_{n-1,n-2,\ldots,0}=\prod_{1\leq j<k\leq n}(x_j-x_k)\] 
is the Vandermonde
determinant. Define a \emph{generalized
Schur polynomial} $s_{\alpha}$ by 
\begin{equation}
    \label{eqn:schurdef}
s_{\alpha}(x_1,\ldots,x_n):=a_{\alpha_1+n-1,\alpha_2+n-2,\ldots,\alpha_n}/a_{n-1,n-2,\ldots,1,0}.
\end{equation}

This is well-known, and clear from (\ref{eqn:schurdef}) and the
row-swap property of determinants: 
\begin{lemma}
\label{lemma:maintrick}
$s_{t_i\alpha}(x_1,\ldots,x_n)=-s_{\alpha}(x_1,\ldots,x_n)$. Thus, if $\alpha_{i+1}=\alpha_i+1$ then
$s_{\alpha}(x_1,\ldots,x_n)=0$.
\end{lemma}

A result  we need  is a characterization of the monomials $x^{\beta}$ that appear
(with nonzero coefficient) in $\kappa_{\alpha}$. Graphically represent the weak composition $\alpha$ as a \emph{skyline} $D(\alpha)$ of boxes where column $i$ (from the left) is a tower of $\alpha_i$ boxes. For example, if $\alpha=(3,0,4,1,0,2)$ then the associated skyline is
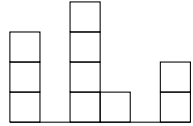
\[
\begin{tikzpicture}[scale=0.400000000000000]
\draw(0,0)--(1,0);
\draw(0,0)--(1,0)--(1,1)--(0,1)--(0,0);
\draw(0,1)--(1,1)--(1,2)--(0,2)--(0,1);
\draw(0,2)--(1,2)--(1,3)--(0,3)--(0,2);
\draw(1,0)--(2,0);
\draw(2,0)--(3,0);
\draw(2,0)--(3,0)--(3,1)--(2,1)--(2,0);
\draw(2,1)--(3,1)--(3,2)--(2,2)--(2,1);
\draw(2,2)--(3,2)--(3,3)--(2,3)--(2,2);
\draw(2,3)--(3,3)--(3,4)--(2,4)--(2,3);
\draw(3,0)--(4,0);
\draw(3,0)--(4,0)--(4,1)--(3,1)--(3,0);
\draw(4,0)--(5,0);
\draw(5,0)--(6,0);
\draw(5,0)--(6,0)--(6,1)--(5,1)--(5,0);
\draw(5,1)--(6,1)--(6,2)--(5,2)--(5,1);
\end{tikzpicture}
\]
Define ${\sf Tab}(\alpha)$ to be fillings of $D(\alpha)$ with ${\mathbb N}:=\{1,2,3,\ldots\}$ such that:
\begin{itemize}
\item no label appears twice in a row (\emph{row distinct}); and
\item the labels in column $i$ are at most $i$ (\emph{flagged}).
\end{itemize}
 The \emph{weight}
of $T\in {\sf Tab}(\alpha)$ is the vector ${\sf wt}(T)=(c_1,c_2,\ldots)$ where $c_i=\#\{i\in T\}$. The following result
is implicit in \cite{ARYpreprint, ARYFPSAC,ARYAdv} and explicit in \cite{Fan}.

\begin{theorem}\label{thm:nonzerorule}
$[x^{\beta}] \kappa_{\alpha}\neq 0$ if and only if there exists $T\in {\sf Tab}(\alpha)$ with content
$\beta$.
\end{theorem}
\begin{proof}
We explicate the argument alluded to in \cite{ARYpreprint, ARYFPSAC, ARYAdv}; we refer
to these papers for definitions. This argument
differs from the one in \cite{Fan}. In \cite{ARYAdv}, it is shown that a lattice point $\beta$ appears
in the \emph{Schubitope} associated to $D(\alpha)$ (rotated $90$-degrees clockwise) if and only
if there exists $T\in {\sf Tab}(\alpha)$ with content
$\beta$. In \cite{Fink}, it is proved that these lattice points correspond exactly to the monomials
of $\kappa_{\alpha}$.
\end{proof}

A consequence of Theorem~\ref{thm:nonzerorule} that we will use is
\begin{corollary}
\label{cor:nonzeroconsequence}
Let $\alpha,\beta\in {\sf Comp}_n$ and assume $[x^{\beta}]\kappa_{\alpha}>0$.
Suppose $i<j$ and $\beta_j-\beta_i=t\in {\mathbb Z}_{>0}$. For $1\leq s\leq t$,
let $\beta':=(\ldots,\beta_i+s,\ldots,\beta_j-s,\ldots)$. Then $[\beta']\kappa_{\alpha}>0$.
\end{corollary}
\begin{proof}
By Theorem~\ref{thm:nonzerorule} there exists $T\in {\sf Tab}(\alpha)$ of content $\beta$.
By definition, there are $\beta_j$ distinct rows where $T$ has a label $j$, and there are $\beta_i$
distinct rows where $T$ has a label $i$. Since $\beta_j-\beta_i=t$, there exist $s$ rows where $T$
contains a $j$ but not an $i$. Define $T'$ by replacing $j$ by $i$ in those $s$ rows. Since $i<j$,
we conclude $T'\in {\sf Tab}(\beta')$ and hence (by Theorem~\ref{thm:nonzerorule}), 
$[\beta']\kappa_{\alpha}>0$, as claimed.
\end{proof}

Given $\alpha$, define the set of \emph{Kohnert diagrams} ${\sf Koh}(\alpha)$ iteratively. To start $D(\alpha)\in {\sf Koh}(\alpha)$. If $D\in {\sf Koh}(\alpha)$, consider the
top-most box in any column. Let $D'$ be the result of moving that box left, in the same row,
to the rightmost location that is not occupied (if it exists); this operation is a \emph{Kohnert move}. Now include $D'\in {\sf Koh}(\alpha)$, as well. We emphasize that ${\sf Koh}(\alpha)$ is a finite set (rather than multiset), hence if a diagram $D$
is obtained by two different sequences of Kohnert moves starting from $D(\alpha)$, then $D$ only counts once in ${\sf Koh}(\alpha)$. 

Given $D\in {\sf Koh}(\alpha)$, let 
\[{\sf Kohwt}(D)=\prod_{i=1}^n x_i^{\#\text{boxes of $D$ in column $i$}}.\]

\begin{theorem}[Kohnert's rule \cite{Kohnert}]\label{thm:Kohnert}
$\kappa_{\alpha}=\sum_{D\in {\sf Koh}(\alpha)} {\sf Kohwt}(D)$.
\end{theorem}

Define \emph{dominance order} on $\alpha,\beta\in {\sf Comp}_n$ such that $|\alpha|:=\sum_{i=1}^n \alpha_i=\sum_{i=1}^n\beta_i :=|\beta|$ by $\alpha\leq_{\sf dom}\beta$ if for every $1\leq t\leq n$ we have $\sum_{i=1}^t \alpha_i\leq \sum_{i=1}^t \beta_i$.

\begin{corollary}
\label{cor:weightsdom}
Let $\alpha,\beta\in {\sf Comp}_n$ with $[x^{\beta}]\kappa_{\alpha}>0$. Then $\beta \geq_{\sf dom} \alpha$.
\end{corollary}

\subsection{Split-symmetry} 
We recall some notions from \cite[Section~4]{Hodges.Yong1}. Suppose 
\[d_0:=0<d_1<d_2<\ldots<d_k<d_{k+1}:=n\] 
and $D=\{d_1,\ldots,d_k\}$. Let $\Pi_D$ be the subring of ${\sf Pol}$ 
consisting of the polynomials
that are separately symmetric in $X_i:=\{x_{d_{i-1}+1},\ldots,x_{d_i}\}$ for
$1\leq i\leq k+1$. If $f\in \Pi_D$, $f$ is \emph{$D$-split-symmetric}.

The ring $\Pi_D$ has a basis of \emph{$D$-Schur polynomials} 
\[s_{\lambda^1,\ldots,\lambda^k}:=s_{\lambda^1}(X_1)s_{\lambda^2}(X_2)\cdots s_{\lambda^k}(X_k),\]
where 
\[(\lambda^1,\ldots,\lambda^k)\in {\sf Par}_D:={\sf Par}_{d_1-d_0}\times \cdots \times{\sf Par}_{d_{k+1}-d_k},\]
and ${\sf Par}_t$ is the set of partitions with at most $t$ nonzero-parts. See \cite[Definition~4.3, Corollary~4.4]{Hodges.Yong1}. Thus, for any
$f\in \Pi_D$ there is a unique expression
\[f=\sum_{(\lambda^1,\ldots,\lambda^k)\in {\sf Par}_D} c_{\lambda^1,\ldots,\lambda^k} s_{\lambda^1,\ldots,\lambda^k}.\]
If $c_{\lambda^1,\ldots,\lambda^k}\in \{0,1\}$ for all $(\lambda^1,\ldots,\lambda^k)\in {\sf Par}_D$,
$f$ is called \emph{$D$-multiplicity-free}.

This fact allows us to study Levi-sphericality using key polynomials:

\begin{theorem}[{\cite[Theorem~4.13]{Hodges.Yong1}}]
\label{thm:fundamentalRelationship}
Let $\lambda \in {\sf Par}_n$, and $w \in {\mathfrak S}_n$. Suppose $I\subseteq J(w)$ and $D=[n-1]-I$.
$X_w$ is $L_I$-spherical if and only if  $\kappa_{w \lambda}$ is $D$-multiplicity-free
for all $\lambda \in {\sf Par}_n$.
\end{theorem}

In view of Theorem~\ref{thm:fundamentalRelationship}, the following is equivalent to Theorem~\ref{thm:main}.

\begin{theorem}
\label{thm:maingoal}
Let $D=[n-1]-I$. $w$ is $I$-spherical if and only if
$\kappa_{w\lambda}$ is $D$-multiplicity-free for all $\lambda\in {\sf Par}_n$.
\end{theorem}

Our goal is therefore to prove Theorem~\ref{thm:maingoal}. To do this, we will use the lemma below.

\begin{lemma}
\label{lemma:thetrick}
Let $\beta\in {\sf Comp}_n$, then 
\[\pi_{w_0(I)}(x_1^{\beta_1}\cdots x_n^{\beta_n})\in \{0,
{\sf sgn}(\beta) s_{\alpha^1,\ldots,\alpha^k}\},\] 
where $(\alpha^1,\ldots,\alpha^k)\in {\sf Par}_D$.
\end{lemma}
\begin{proof}
First, consider the special case that $w_0(I)=w_0$. By \cite[Proposition~1.5.1]{Lascoux:polynomials},
\[\pi_{w_0}(f)=\frac{1}{\Delta_n}x^{\rho}\sum_{w\in {\mathfrak S}_n}(-1)^{\ell(w)}w(f).\]
Hence by (\ref{eqn:schurdef}), $\pi_{w_0}(x^{\beta})=s_{\beta}$. Rearrange $\beta$ to be weakly decreasing by application of the operators $t_1,t_2,\ldots$ and
swapping two adjacent entries where the left entry is strictly smaller than the other one. This can always be achieved unless during this process one arrives at a composition $\kappa$ where $\kappa_{i+1}=\kappa_i+1$.
In that case, Lemma~\ref{lemma:maintrick} asserts $s_{\beta}=0$.
Otherwise we arrive at $\alpha\in {\sf Par}_n$
and Lemma~\ref{lemma:maintrick} combined with Definition~\ref{def:rank} shows $s_{\beta}={\sf sgn}(\beta)s_{\alpha}$.

In the general case, $w_{0}(I)$ is by definition the long element
of the Young subgroup ${\mathfrak S}_{d_1-d_0}\times \cdots \times {\mathfrak S}_{d_{k+1}-d_k}$
of ${\mathfrak S}_n$. Hence $w_0(I)=w_0^{(1)} w_0^{(2)}\ldots, w_0^{(k+1)}$ where
$w_0^{(i)}$ is the long element of ${\mathfrak S}_{d_{i}-d_{i-1}}=$ \ the parabolic
subgroup of ${\mathfrak S}_n$ generated by $s_{d_{i-1}+1},s_{d_{i-1}+2},\ldots, s_{d_i-1}$.
Hence, it follows that 
\begin{equation}
\label{eqn:April7aaa}
\pi_{w_0(I)}=\pi_{w_0^{(1)}}\pi_{w_0^{(2)}}\cdots \pi_{w_0^{(k+1)}}.
\end{equation}
and the factors commute. Thus, the general case follows from (\ref{eqn:April7aaa}) and
the special case.
\end{proof}

\section{The subposet ${\mathcal P}_{u\lambda,\gamma}$ of ${\mathcal S}_{I,\gamma}$  and the Proof of Theorem~\ref{thm:maingoal} ($\Rightarrow$)}\label{sec:5}

\begin{lemma}\label{lemma:containsall}
${\mathcal S}_{I,\gamma}$ (as a set) contains all $\beta\in {\sf Comp}_n$ such that
$\pi_{w_0(I)}x^{\beta}=\pm s_{\gamma}$.
\end{lemma}
\begin{proof}
Suppose $\beta\in {\sf Comp}_n$ satisfies $\pi_{w_0(I)}x^{\beta}=\pm s_{\gamma}(\neq 0)$.
As in the proof of Lemma~\ref{lemma:thetrick} by successive applying the operators $t_1,t_2,\ldots$ ($i\in I$) to $\beta$, we either arrive at some $\gamma'\in {\sf Par}_D$ or
a $\kappa\in {\sf Comp}_n$ with $\kappa_{i+1}=\kappa_i+1$ where $i,i+1$ are in the same block.
In the latter case we conclude, by the (proof of) Lemma~\ref{lemma:thetrick} that $\pi_{w_0(I)}x^{\beta}=0$, a contradiction. Otherwise we find $\pm s_{\gamma}= s_{\gamma'}$, which can
only happen if $\gamma=\gamma'$. Thus, we have found a sequence of $t_i$'s connecting
$\beta$ to $\gamma$. The result then follows from Lemma~\ref{lemma:commutes} and the
definition of ${\mathcal S}_{I,\gamma}$.
\end{proof}

We need a subposet of ${\mathcal S}_{I,\gamma}$ attached to the following datum:
\begin{itemize}
\item $w=w_0(I)u\in {\mathfrak S}_n$ where $I\subset J(w)$ and $\ell(w)=\ell(w_0(I))+\ell(u)$.
\item $\alpha=u\lambda$ for some $\lambda\in {\sf Par}_n$.
\item $\gamma\in {\sf Par}_D$ where $D=[n]-I=\{d_1<d_2<\ldots<d_k\}$.
\end{itemize}

\begin{definition}
${\mathcal P}_{\alpha,\gamma}$ is the subposet of ${\mathcal S}_{I,\gamma}$ induced by those
$\beta\in {\mathcal S}_{I,\gamma}$ such that $[x^\beta]\kappa_{\alpha}\neq 0$.
\end{definition}

The next result holds for $u=c$, a standard Coxeter element for a parabolic subgroup.

\begin{theorem}[Diamond property]\label{thm:diamond}
Let $\beta \in {\mathcal P}_{c \lambda,{\bf \gamma}}$. Let $i < j$ in the same block and $p < q$ in the same block with $(i,j) \neq (p,q)$.  If both $t_{ij}\beta$ and $t_{pq}\beta$ are in $\mathcal{P}_{c\lambda,{\bf\gamma}}$ and cover $\beta$, 
then there exists $\beta'\in {\mathcal P}_{c \lambda,{\bf \gamma}}$ such that $t_{ij}\beta , t_{pq}\beta < \beta'$.
\end{theorem}

We defer the proof of Theorem~\ref{thm:diamond} until Section~\ref{sec:diamond}. We complete this section by 
using Theorem~\ref{thm:diamond} to prove the ``$\Rightarrow$'' direction of Theorem~\ref{thm:main}. 

The following result is immediate from the Diamond Property (Theorem~\ref{thm:diamond}) and Newman's diamond lemma \cite{Newman}. 
\begin{lemma}
\label{claim:Mar4aaa}
${\mathcal P}_{c\lambda,\gamma}$ has a unique maximum.
\end{lemma}

\begin{lemma}\label{claim:Mar4bbb}
Suppose $\beta\in {\mathcal P}_{\alpha,\gamma}$, $\beta_i<\beta_j-(j-i)$ for some $i<j$ in the same block.
Then $t_{ij}\beta\in {\mathcal P}_{\alpha,\gamma}$.
\end{lemma}
\begin{proof}
By Lemma~\ref{lemma:containsall}
${\mathcal S}_{I,\gamma}$ consists of all $\beta$ such that
$\pi_{w_0(I)}x^{\beta}=\pm s_{\gamma}$. Let $\beta':=t_{ij}\beta$. Thus,
$\beta_i'=\beta_j-(j-i)$, $\beta_j'=\beta_i+(j-i)$, and $\beta_k'=\beta_k$ if $k\neq i,j$.
The hypothesis that $\beta_i<\beta_j-(j-i)$ means $\beta_i<\beta_i'$ and $\beta_j'<\beta_j$
and $\beta_j'-\beta_i'=(j-i)\in {\mathbb Z}_{>0}$. 
Hence by Corollary~\ref{cor:nonzeroconsequence}, $[x^{\beta'}]\kappa_{\alpha}>0$. Therefore,
it follows that $\beta'=t_{ij}\beta\in {\mathcal P}_{\alpha,\gamma}$, as desired.
\end{proof}

\begin{lemma}
Let ${\mathfrak S}:={\mathfrak S}_{d_1-d_0}\times\cdots\times {\mathfrak S}_{d_{k+1}-d_k}$ be a Young subgroup of ${\mathfrak S}_n$. Suppose $[u,v]\subset {\mathfrak S}$ is an interval. Then
\begin{equation}
\label{eqn:deodhar}
\sum_{u\leq w\leq v} (-1)^{\ell(uw)}=\begin{cases}
1 & \text{if $u=v$}\\
0 & \text{otherwise}
\end{cases}
\end{equation}
\end{lemma}
\begin{proof}
For a (locally) finite poset $P$ let $\mu_P:P\times P\to {\mathbb R}$ be its M\"obius function.
This is defined recursively by $\mu_P(x,x)=1$ and $\mu_P(x,z)=-\sum_{x\leq_P z<_P y} \mu_P(x,z)$. When $P={\mathfrak S}={\mathfrak S}_n$, the lemma holds since $(-1)^{\ell(uw)}$ is the M\"obius function
for ${\mathfrak S}_n$ under Bruhat order~\cite{Deodhar}.

For the general case, recall  
\cite[Proposition~3.8.2]{EC1}, which states that if $P$ and $Q$ be locally finite posets, and  $P\times Q$ is their direct product, if $(s,t)\leq (s',t')$ in $P\times Q$ then the M\"obius functions of $P\times Q, P$, and $Q$ are related by 
\begin{equation}
\label{eqn:productMobius}
\mu_{P \times Q}((s, t), (s', t')) = \mu_P(s, s')\mu_Q(t, t').
\end{equation}
Elements of ${\mathfrak S}$ are uniquely factorizable as 
$w=p^{(1)} p^{(2)} \cdots p^{(k+1)}$ where $p^{(i)}$ is an element of
the parabolic subgroup ${\mathfrak S}_{d_i-d_{i-1}}$ of ${\mathfrak S}_n$ generated by $s_{d_{i-1}+1},s_{d_{i-1}+2},\ldots,
s_{d_i-1}$. Similarly, let $u=q^{(1)} q^{(2)} \cdots q^{(k+1)}$ be the factorization of $u\in {\mathfrak S}$,
and $u\leq_{\text{Bruhat}} w$. By iterating application of (\ref{eqn:productMobius}) $k$-many times,
\[\mu_{\mathfrak S}(u,w)=\prod_{i=1}^{k+1}\mu_{{\mathfrak S}_{d_i-d_{i-1}}}(q^{(i)},p^{(i)})
=(-1)^{\sum_{i=1}^{k+1}\ell(q^{(i)} p^{(i)})}=(-1)^{\ell(wu)},\]
and the result follows.
\end{proof}

\begin{proposition}
\label{prop:interval}
$({\mathcal P}_{c\lambda,\gamma},<_{\text{\emph{Bruhat}}})$ is isomorphic (as posets) to an interval
in $({\mathfrak S}_{d_1-d_0}\times\cdots \times {\mathfrak S}_{d_{k+1}-d_k},<_{\text{\emph{Bruhat}}})$.
\end{proposition}

Assuming the proof of Theorem~\ref{thm:diamond} (given in the next section), we are
ready to present:

\noindent
\emph{Proof of Proposition~\ref{prop:interval} and Theorem~\ref{thm:maingoal} ($\Rightarrow$):}
Let 
\[\Gamma:({\mathcal S}_{I,\gamma},<_{\text{Bruhat}})
\rightarrow ({\mathfrak S}_{d_1-d_0}\times\cdots \times {\mathfrak S}_{d_{k+1}-d_k},<_{\text{Bruhat}})\]
denote the isomorphism of posets from Proposition~\ref{prop:posetsiso}.

Let $\beta_{\sf max}$ be the unique maximum of
${\mathcal P}_{c\lambda,\gamma}\subseteq {\mathcal S}_{I,\gamma}$, guaranteed to exist by Lemma~\ref{claim:Mar4aaa}.
The unique minimum is $\gamma$. It follows from Lemma~\ref{claim:Mar4bbb} that 
\[\Gamma({\mathcal P}_{c\lambda,\gamma})=[\Gamma(\gamma),\Gamma(\beta_{\sf max})]
\subseteq ({\mathfrak S}_{d_1-d_0}\times\cdots \times {\mathfrak S}_{d_{k+1}-d_k},<_{\text{Bruhat}}).\]
This is the assertion of Proposition~\ref{prop:interval}.

If ${\sf sgn}(\beta)$ is the sign associated to $\beta$, then this maps to $(-1)^{\ell(w_{\beta})}$, 
which agrees with the M\"obius function on ${\mathfrak S}$. Now apply (\ref{eqn:deodhar}) to conclude $s_{\gamma}$ appears in the $D$-split expansion of $\kappa_{w\lambda}=\pi_{w_0(I)}\kappa_{c\lambda}$
with coefficient zero or one, completing the proof of Theorem~\ref{thm:maingoal}.\qed

\begin{example}
Let $w=765432918$ and $\lambda=987654321$. Hence $J(w)=\{1,2,3,4,5,6,8\}$; let 
$I=\{2,3,4,5,6\}\subseteq J(w)$. Thus $w_0(I)=176543289$ and we can factor $w=w_0(I)c$ where $c$ is the
standard Coxeter element $c=234567918=s_8 s_1 s_2 s_3 s_4 s_5 s_6 s_7$. Now, $c^{-1}=812345697$ 
and $w^{-1}=865432197$. Therefore
$\alpha=c\lambda=298765413$, whereas $w\lambda=245678913$.

Since $D=[9]-I=\{1,7,8,9\}$, we have that $\kappa_{w\lambda}=\kappa_{245678913}\in \Pi_D$ 
is separately symmetric in the
sets of indeterminates $\{x_1\}, \{x_2,x_3,x_4,x_5,x_6,x_7\},\{x_8\},\{x_9\}$.

Since $c$ is a standard Coxeter element, by  \cite[Theorem 4.13(II)]{Hodges.Yong1}, we have that
$\kappa_{c\lambda}$ is $[n-1]$-multiplicity-free.
Consider the term $x^{928765422}$ appearing in $\kappa_{c\lambda}$.
Now
\[\pi_{w_0(I)}(x^{928765422})=s_{9, \underline{28}7654, 2, 2}
=-s_{9,7\underline{37}654,2,2}=s_{9,76\underline{46}54,2,2}
=-s_{9, 765554,2,2},\]
where we have underlined the swaps. 

The list of monomials $x^{\beta}$ of $\kappa_{c\lambda}$ such that $\pi_{w_0(I)}(x^{\beta})=\pm s_{9,765554,2,2}$, together with
the signs they contribute are:
\[[9, 7, 6, 5, 5, 5, 4, 2, 2] \ 1,
[9, 7, 4, 7, 5, 5, 4, 2, 2] \ -1,
[9, 7, 6, 4, 6, 5, 4, 2, 2] \ -1,\]
\[[9, 5, 8, 4, 6, 5, 4, 2, 2] \ 1,
[9, 7, 3, 7, 6, 5, 4, 2, 2] \ 1,
[9, 5, 8, 5, 5, 5, 4, 2, 2] \ -1,\]
\[[9, 2, 8, 7, 6, 5, 4, 2, 2] \ -1,
[9, 3, 8, 7, 5, 5, 4, 2, 2] \ 1.\]
These elements form a poset ${\mathcal P}_{c\lambda,\gamma=9,765554,2,2}$ shown in Figure~\ref{fig:976555422} isomorphic to an interval $[\mathrm{id},s_2s_3s_4]$ in Bruhat order, consistent with Proposition~\ref{prop:interval}.
\begin{figure}[h!]
\begin{tikzpicture}[scale=2.0]
\node[label=below:{9,765554,2,2}] at (0,0) {$\bullet$};
\node[label=left:{9,585554,2,2}] at (-1.5,1) {$\bullet$};
\node[label=below:{9,747554,2,2}] at (0,1) {$\bullet$};
\node[label=right:{9,764654,2,2}] at (1.5,1) {$\bullet$};
\node[label=left:{9,387554,2,2}] at (-1.5,2) {$\bullet$};
\node[label=above:{9,584654,2,2}] at (0,2) {$\bullet$};
\node[label=right:{9,737654,2,2}] at (1.5,2) {$\bullet$};
\node[label=above:{9,287654,2,2}] at (0,3) {$\bullet$};
\draw(0,0)--(-1.5,1);
\node at (-0.8,0.4) {$t_2$};
\draw(0,0)--(0,1);
\node at (0.1,0.4) {$t_3$};
\draw(0,0)--(1.5,1);
\node at (0.8,0.4) {$t_4$};
\draw(-1.5,1)--(-1.5,2);
\node at (-1.65,1.5) {$t_{24}$};
\draw(-1.5,1)--(0,2);
\draw(0,1)--(-1.5,2);
\draw(0,1)--(1.5,2);
\draw(1.5,1)--(0,2);
\draw(1.5,1)--(1.5,2);
\node at (1.6,1.5) {$t_3$};
\draw(-1.5,2)--(0,3);
\node at (-0.85,2.6) {$t_{25}$};
\draw(0,2)--(0,3);
\node at (0.15,2.6) {$t_{24}$};
\draw(1.5,2)--(0,3);
\node at (0.8,2.6) {$t_2$};
\end{tikzpicture}

\caption{The poset ${\mathcal P}_{c\lambda,\gamma}$ for $c=234567918$, $\lambda=987654321$, $\gamma=976555422$, $I=\{2,3,4,5,6\}$ with some edges labeled}
\label{fig:976555422}
\end{figure}
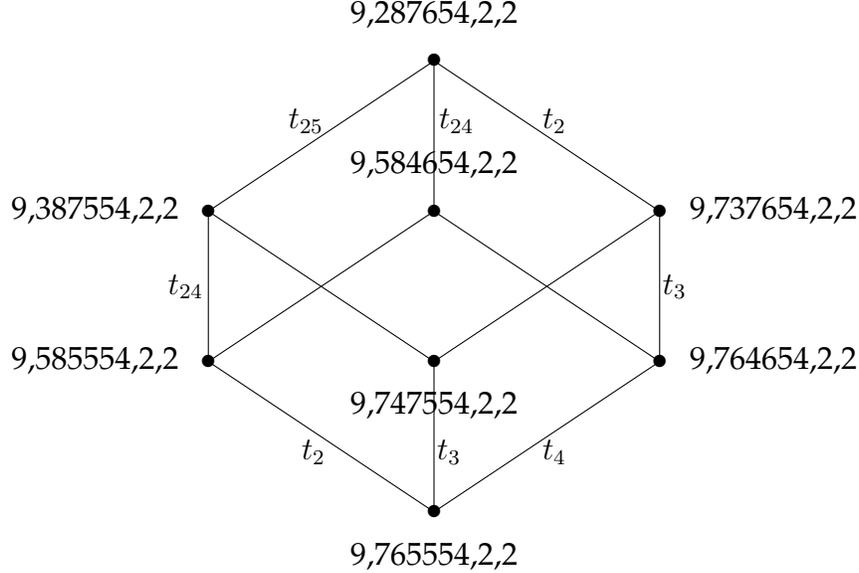

Indeed the coefficients sum to zero, in agreement with the above discussion about the M\"obius function.\qed
\end{example}

\section{Proof of the Diamond property (Theorem~\ref{thm:diamond})}\label{sec:diamond}
\label{sec:6}

Throughout this section we fix a decomposition $w=w_0(I)c$ where $c$ is a standard Coxeter element of some 
parabolic such that $\ell(w)=\ell(w_0(I))+\ell(c)$, and $\lambda\in {\sf Par}_n$.
 
\begin{lemma} 
\label{lem:clamblockdecrease}
Let $w = w_0(I) u \in {\mathfrak S}_n$ with $\ell(w)=\ell(w_0(I))+\ell(u)$. If $i \in I$, then $(u \lambda)_{i} \geq (u \lambda)_{i+1}$.
\end{lemma}
\begin{proof}
The length additivity of $w_0(I)$ and $u$ implies $J(u) \cap J(w_0(I)) = J(u) \cap I = \emptyset$. Thus $u^{-1}(i) < u^{-1}(i+1)$, and since $\lambda$ is a partition, $(u \lambda)_{i} = \lambda_{u^{-1}(i)} \geq \lambda_{u^{-1}(i+1)} = (u \lambda)_{i+1}$.
\end{proof}

We will use the following notion from \cite{Hodges.Yong1}:

\begin{definition}[Composition patterns]
Let 
${\sf Comp} := \bigcup_{n=1}^{\infty}  {\sf Comp}_n$.  
For $\alpha=(\alpha_1,\ldots,
\alpha_{\ell}), \beta=(\beta_1,\ldots, \beta_{k}) \in {\sf Comp}$, $\alpha$ \emph{contains the composition pattern} $\beta$ 
if there exists integers $j_1 < j_2 < \cdots < j_k$ that satisfy:
\begin{itemize}
\item $(\alpha_{j_{1}},\ldots,\alpha_{j_{k}})$ is order isomorphic to $\beta$ ($\alpha_{j_s} \leq \alpha_{j_t}$ if and only if $\beta_{s} \leq \beta_{t})$, 
\item $|\alpha_{j_s} - \alpha_{j_t}| \geq |\beta_{s} - \beta_{t}|$.
\end{itemize}
If $\alpha$ does not contain $\beta$, then $\alpha$ \emph{avoids} $\beta$. 
\end{definition}

\begin{lemma}
\label{lem:clampattern}
$c \lambda$ avoids $012$, $1032$, $0011$, $0021$, $1022$.
\end{lemma}
\begin{proof}
Since $c$ is a standard Coxeter element in a parabolic subgroup, $X_c \subseteq GL_n / B$ is a toric variety~\cite{Karupp}. Hence, by \cite[Theorem 4.13(II)]{Hodges.Yong1}, $\kappa_{c \lambda}$ is $[n-1]$-multiplicity-free for all $\lambda\in {\sf Par}_n$. In \cite{Hodges.Yong2}, it is shown that $\kappa_{\alpha}$ is $[n-1]$-multiplicity free if and only if $\alpha$ avoids $012$, $1032$, $0022$, $0021$, $1022$. Thus, since $\kappa_{c \lambda}$ is $[n-1]$-multiplicity-free, $c \lambda$ avoids $012$, $1032$, $0022$, $0021$, $1022$.

To seek a contradiction, suppose that $c \lambda$ contains the pattern $0011$. Let $j_1<j_2<j_3<j_4$ be the integers such that $(c \lambda)_{j_1}, (c \lambda)_{j_2}, (c \lambda)_{j_3}, (c \lambda)_{j_4}$ contains the composition pattern $0011$. Let $\tilde \lambda \in {\sf Par}_n$ be obtained from $\lambda$ by replacing all part lengths equal to $(c \lambda)_{j_3}$ by $(c \lambda)_{j_3} + 1$. Then $c \tilde\lambda$ contains the pattern $0022$. We conclude, \emph{via} \cite{Hodges.Yong2}, that $\kappa_{c \tilde\lambda}$ is not $[n-1]$-multiplicity-free. By \cite[Theorem 4.13(II)]{Hodges.Yong1}, this implies $X_c$ is not a toric variety, a contradiction. Thus $c \lambda$ must also avoid the pattern $0011$.
\end{proof}

\begin{definition}
\label{def:leftminrightmax}
Let ${\sf leftmin}_{\alpha}(i) = \min\{ \alpha_{j} : j \leq i \}$ and ${\sf rightmax}_{\alpha}(i) = \max\{ \alpha_{j} : j \geq i \}$. 
\end{definition}

\begin{lemma}
\label{lem:lminrmax}
Let $1 \leq i,j \leq n$ and $F\in {\sf Tab}(c\lambda)$. Then
\begin{itemize}
    \item[(i)] $({\sf wt}(F))_k \geq {\sf leftmin}_{c \lambda}(i)$ for $1 \leq k \leq i$.
    \item[(ii)] $({\sf wt}(F))_{k} \leq {\sf rightmax}_{c \lambda}(j)$ for $j \leq k \leq n$.
    \item[(iii)] If $i < j$ are in the same block and ${\sf leftmin}_{c \lambda}(i) = (c \lambda)_i$ and ${\sf rightmax}_{c \lambda}(j) = (c \lambda)_{j}$, then $({\sf wt}(F))_i = (c \lambda)_i$ and $({\sf wt}(F))_{j} = (c \lambda)_{j}$.
\end{itemize}
\end{lemma}
\begin{proof}
(i): By Definition \ref{def:leftminrightmax}, for $1 \leq k \leq i$, $(c \lambda)_k \geq {\sf leftmin}_{c \lambda}(i)$. By induction, and the definition of flagged fillings, $F(k, r) = k$ for $1 \leq k \leq i$ and $1 \leq r \leq {\sf leftmin}_{c \lambda}(i)$. Thus $({\sf wt}(F))_k \geq {\sf leftmin}_{c \lambda}(i)$ for $1 \leq k \leq i$.

(ii): Once again we apply Definition \ref{def:leftminrightmax}, concluding ${\sf rightmax}_{c \lambda}(k) \leq {\sf rightmax}_{c \lambda}(j)$ for $j \leq k \leq n$. By the definition of flagged fillings a value $k$ can only appear once in a fixed row, and only in columns greater than or equal to $k$. Hence, $({\sf wt}(F))_{k} \leq {\sf rightmax}_{c \lambda}(k) \leq {\sf rightmax}_{c \lambda}(j)$.

(iii): If $i, j$ are in the same block, then Lemma \ref{lem:clamblockdecrease}, applied inductively, implies $(c \lambda)_k \geq (c \lambda)_j$ for $i \leq k \leq j$. This, combined with ${\sf leftmin}_{c \lambda}(i) = (c \lambda)_i$, implies that ${\sf leftmin}_{c \lambda}(j) = (c \lambda)_j$. Applying (i) and (ii) to $j$ yields $({\sf wt}(F))_{j} \geq (c \lambda)_{j}$ and $({\sf wt}(F))_{j} \leq (c \lambda)_{j}$. Hence $({\sf wt}(F))_{j} = (c \lambda)_{j}$.

Additionally, $(c \lambda)_k \geq (c \lambda)_j$ for $i \leq k \leq j$ combined with ${\sf rightmax}_{c \lambda}(j) = (c \lambda)_{j}$ gives ${\sf rightmax}_{c \lambda}(i) = (c \lambda)_{i}$. Applying (i) and (ii) to $i$ again yields the desired equality.
\end{proof}

\begin{lemma}
\label{lem:rectanglesw}
Let $i\leq j$ with $(c \lambda)_k \geq (c \lambda)_{k+1}$ for $i \leq k < j$. Let $m$ be the maximum value such that $i \leq m \leq j$ and $(c \lambda)_m \geq {\sf leftmin}_{c \lambda}(i)$. Then
\[
| \left\{ (d,r) \in D(c \lambda) : d \leq m \right\} | = m \text{ for } 1 \leq r \leq {\sf leftmin}_{c \lambda}(i).
\]
This implies that for all $F \in  {\sf Tab}(c \lambda)$ , 
\[
F(d,r) = d \text{ for } 1 \leq r \leq {\sf leftmin}_{c \lambda}(i)\text{ and } 1\leq d \leq m.
\]
\end{lemma}
\begin{proof}
This first claim follows from the definition of ${\sf leftmin}_{c \lambda}(i)$. The latter then follows from inductively applying the flagged and row distinct properties of $F$.
\end{proof}

\begin{definition}
If $i < j$ with $(c \lambda)_k \geq (c \lambda)_{k+1}$ for $i \leq k < j$, ${\sf leftmin}_{c\lambda}(i) < (c\lambda)_i$, and 
${\sf rightmax}_{c\lambda}(j) > (c\lambda)_j$, then we say the pair $(i,j)$ is \emph{interweaved}. For such an $(i,j)$, define $${\sf center}_{c\lambda}(i,j) = \max\{ k : i \leq k \leq j \text{ and }(c\lambda)_k \geq {\sf rightmax}_{c\lambda}(j)\}.$$
\end{definition}

Notice ${\sf center}_{c\lambda}(i,j)\neq -\infty$ since $(c\lambda)_i\geq {\sf rightmax}_{c\lambda}(j)$ (otherwise, we have
${\sf leftmin}_{c\lambda}(i)<(c\lambda)_i<{\sf rightmax}_{c\lambda}(j)$ which says $c\lambda$ contains a $012$-pattern, contradicting Lemma~\ref{lem:clampattern}.

\begin{lemma}
\label{lem:formofclam}
Let $i < j$ with $(c \lambda)_k \geq (c \lambda)_{k+1}$ for $i \leq k < j$. Then
\begin{itemize}
\item[(i)] If ${\sf leftmin}_{c \lambda}(i) = (c \lambda)_i$ and ${\sf rightmax}_{c \lambda}(j) > (c \lambda)_{j}$, then \[
    | \left\{ (d,r) \in D(c \lambda) : d > i \right\} | \leq 1\text{ for }r > (c \lambda)_i,
\]
\item[(ii)] If ${\sf leftmin}_{c \lambda}(i) < (c \lambda)_i$\text{ and }${\sf rightmax}_{c \lambda}(j) = (c \lambda)_{j}$, then \[
    | \left\{ (d,r) \in D(c \lambda) : d \leq j \right\} | \geq j-1 \text{ for }{\sf leftmin}_{c \lambda}(i) < r \leq (c \lambda)_{j},
\]
\item[(iii)] If ${\sf leftmin}_{c \lambda}(i) < (c \lambda)_i$ and ${\sf rightmax}_{c \lambda}(j) > (c \lambda)_{j}$, then\[
	| \left\{ (d,r) \in D(c \lambda) : d \geq {\sf center}_{c \lambda}(i,j) \right\} | = 1\text{ for } {\sf leftmin}_{c \lambda}(i) < r \leq {\sf rightmax}_{c \lambda}(j),
\]
and
\[
| \left\{ (d,r) \in D(c \lambda) : d \leq {\sf center}_{c \lambda}(i,j) \right\} | = {\sf center}_{c \lambda}(i,j)-1\text{ for }{\sf leftmin}_{c \lambda}(i) < r \leq {\sf rightmax}_{c \lambda}(j).
\]
\end{itemize}
\end{lemma}
\begin{proof}
(i): Let $r > (c \lambda)_i$. If $j < d_1 < d_2$, then $c \lambda$ contains the pattern ($(c \lambda)_{i}$, $(c \lambda)_{j}, (c \lambda)_{d_1}$, $(c \lambda)_{d_2}$). Suppose that $(d_1, r),(d_2, r) \in D(c \lambda)$. This implies $(c \lambda)_{d_1}, (c \lambda)_{d_2} \geq (c \lambda)_{i}$.
This, combined with $(c \lambda)_{i} \geq (c \lambda)_{j}$, implies $((c \lambda)_{i}, (c \lambda)_{j}, (c \lambda)_{d_1}, (c \lambda)_{d_2})$ contains $012$, $1032$, $0021$, $0011$, or $1022$. This contradicts Lemma~\ref{lem:clampattern}. Thus 
\[| \left\{ (d, r) \in D(c \lambda) : d > j \right\} | \leq 1\text{ for }r > (c \lambda)_i.\] 
Further, since $r > (c \lambda)_i \geq (c \lambda)_k$ for $i \leq k \leq j$, 
\[
    | \left\{ (d, r) \in D(c \lambda) : d > i \right\} | \leq 1\text{ for }r > (c \lambda)_i.
\]

(ii): Let ${\sf leftmin}_{c \lambda}(i) < r \leq (c \lambda)_{j}$. If $d_1 < d_2 < i$, then $c \lambda$ contains the pattern $((c \lambda)_{d_1}$, $(c \lambda)_{d_2}$, $(c \lambda)_{i}$, $(c \lambda)_{j})$. Suppose that $(d_1, r),(d_2, r) \notin D(c \lambda)$. This implies $(c \lambda)_{d_1}, (c \lambda)_{d_2} \leq (c \lambda)_{j}$.
This, combined with $(c \lambda)_{i} \geq (c \lambda)_{j}$, implies $((c \lambda)_{d_1}, (c \lambda)_{d_2}, (c \lambda)_{i}, (c \lambda)_{j})$ contains $012$, $1032$, $0021$, $0011$, or $1022$. This contradicts Lemma~\ref{lem:clampattern}. Thus 
\[| \left\{ (d, r) \in D(c \lambda) : d \leq i \right\} | \geq i-1 \text{ for }{\sf leftmin}_{c \lambda}(i) < r \leq (c \lambda)_{j}.\] 
Further,  since $r \leq (c \lambda)_j \leq (c \lambda)_k$ for $i \leq k \leq j$,
\[
    | \left\{ (d, r) \in D(c \lambda) : d \leq j \right\} | \geq j-1 \text{ for }{\sf leftmin}_{c \lambda}(i) < r \leq (c \lambda)_{j}.
\]

(iii): Let $x$ be an integer such that $x < i$ and $(c \lambda)_x = {\sf leftmin}_{c \lambda}(i)$, and $y$ be an integer such that $y > j$ and $(c \lambda)_y = {\sf rightmax}_{c \lambda}(j)$.

Our claim holds vacuously if $(c \lambda)_x \geq (c \lambda)_y$. Hence, for the rest of the proof we assume $(c \lambda)_x < (c \lambda)_y$. Now $c \lambda$ contains the pattern $((c \lambda)_x, (c \lambda)_i, (c \lambda)_{j}, (c \lambda)_y)$ and by Lemma~\ref{lem:clampattern} this pattern avoids $012$. This, combined with $(c \lambda)_{j} < (c \lambda)_y$, implies
\begin{equation}
\label{eq:case4eq1}
(c \lambda)_x \geq (c \lambda)_{j}.
\end{equation}
It further implies, when combined with $(c \lambda)_{x} < (c \lambda)_i$, that
\begin{equation}
\label{eq:case4eq11}
(c \lambda)_i \geq (c \lambda)_{y}.
\end{equation}

Let $(c \lambda)_x < r \leq (c \lambda)_y$. Let ${\sf center}_{c \lambda}(i,j) < d_1 < d_2$. Suppose, to obtain a contradiction, that $(d_1, r), (d_1, r) \in D(c \lambda)$. Then 
\begin{equation}
\label{eq:case4eq2}
(c \lambda)_{d_1}, (c \lambda)_{d_2} > (c \lambda)_x.
\end{equation}
If $d_1 \leq j$, then the definition of ${\sf center}_{c \lambda}(i,j)$ implies $(c \lambda)_{d_1} < (c \lambda)_{y}$. This implies $c \lambda$ contains the pattern $((c \lambda)_x, (c \lambda)_{d_1}, (c \lambda)_y)$ which is a $012$ pattern. This contradicts Lemma~\ref{lem:clampattern}. Otherwise, if $j < d_1 < d_2$, then $c \lambda$ contains the pattern $((c \lambda)_x, (c \lambda)_{j}, (c \lambda)_{d_1}, (c \lambda)_{d_2})$. By \eqref{eq:case4eq1} and \eqref{eq:case4eq2}, this pattern contains $012$, $1032$, $0021$, $0011$, or $1022$. This contradicts Lemma~\ref{lem:clampattern}. Thus
\begin{equation}
\label{eq:case4claimCardBoxes}
    | \left\{ (d, r) \in D(c \lambda) : d \geq {\sf center}_{c \lambda}(i,j) \right\} | = 1\text{ for } {\sf leftmin}_{c \lambda}(i) < r \leq {\sf rightmax}_{c \lambda}(j).
\end{equation}

Let $(c \lambda)_x < r \leq (c \lambda)_y$. Let $d_1 < d_2 < {\sf center}_{c \lambda}(i,j)$. Suppose, to obtain a contradiction, that $(d_1, r), (d_2, r) \notin D(c \lambda)$. Thus 
\begin{equation}
\label{eq:case4eq22}
(c \lambda)_{d_1}, (c \lambda)_{d_2} < (c \lambda)_y.
\end{equation}
If $d_2 \geq i$, then $(c \lambda)_{d_2} \geq (c \lambda)_{j}$ and the definition of ${\sf center}_{c \lambda}(i,j)$ implies $(c \lambda)_{d_2} \geq (c \lambda)_y$. This contradicts \eqref{eq:case4eq22}. Otherwise, if $d_1 < d_2 < i$, then $c \lambda$ contains $((c \lambda)_{d_1}, (c \lambda)_{d_2}, (c \lambda)_{i}, (c \lambda)_{y})$. By \eqref{eq:case4eq11} and \eqref{eq:case4eq22}, this pattern contains $012$, $1032$, $0021$, $0011$, or $1022$. This contradicts Lemma~\ref{lem:clampattern}. We conclude $| \left\{ (d, r) \in D(c \lambda) : d \leq {\sf center}_{c \lambda}(i,j) \right\} | \geq {\sf center}_{c \lambda}(i,j)-1$. Since $(c \lambda)_x < r$, we can strengthen this inequality to
\[
    |\! \left\{ (d, r) \!\in\! D(c \lambda) \!:\! d \!\leq\! {\sf center}_{c \lambda}(i,j) \right\}\! |\! = \!{\sf center}_{c \lambda}(i,j)\!-\!1\text{ for }{\sf leftmin}_{c \lambda}(i) \!<\! r \!\leq\! {\sf rightmax}_{c \lambda}(j).
    \qedhere
\]\end{proof}

\begin{proposition}
\label{prop:goingup}
Let $\beta \in {\mathcal P}_{c \lambda,{\bf \gamma}}$, $i<j$ in the same block,
 and $\beta_i > \beta_{j} - (j-i)$. Then $t_{i,j} \beta \in {\mathcal P}_{c \lambda,{\bf \gamma}}$ if and only if 
\begin{enumerate}
\item ${\sf leftmin}_{c \lambda}(i) \leq \beta_{j} - (j-i)$;
\item ${\sf rightmax}_{c \lambda}(j) \geq \beta_i + (j-i)$; and
\item if $(i,j)$ is interweaved, then $$\qquad {\small \beta_1 + \cdots + \beta_{i-1} + (\beta_{j} - (j-i)) + \beta_{i+1} + \cdots + \beta_{{\sf center}_{c \lambda}(i,j)} \geq (c \lambda)_1 + \cdots + (c \lambda)_{{\sf center}_{c \lambda}(i,j)}.}$$
\end{enumerate}
\end{proposition}
\begin{proof}
\noindent $(\Rightarrow)$ We prove the contrapositive. That is, we assume that ${\sf leftmin}_{c \lambda}(i) > \beta_{j} - (j-i)$, ${\sf rightmax}_{c \lambda}(j) < \beta_i + (j-i)$, or $(i,j)$ is interweaved with $\beta_1 + \cdots + \beta_{i-1} + (\beta_{j} - (j-i)) + \beta_{i+1} + \cdots + \beta_{{\sf center}_{c \lambda}(i,j)} < (c \lambda)_1 + \cdots (c \lambda)_{{\sf center}_{c \lambda}(i,j)}$. Let $\tau = t_{i,j} \beta$ and suppose, to seek a contradiction, that $F \in {\sf Tab}(c \lambda)$ with $\tau = {\sf wt}(F)$.

\noindent \textit{Case ${\sf leftmin}_{c \lambda}(i) > \beta_{j} - (j-i)$:}  By the case hypothesis, ${\sf leftmin}_{c \lambda}(i) > \tau_{i} = ({\sf wt}(F))_{i}$. This contradicts Lemma~\ref{lem:lminrmax}(i).

\noindent \textit{Case ${\sf rightmax}_{c \lambda}(i) < \beta_j + (j-i)$:} By the case hypothesis, ${\sf rightmax}_{c \lambda}(i) < \tau_{j} = ({\sf wt}(F))_{j}$. This contradicts Lemma~\ref{lem:lminrmax}(ii).

\noindent \textit{Case $(i,j)$ is interweaved with $\beta_1 + \cdots + \beta_{i-1} + (\beta_{j} - (j-i)) + \beta_{i+1} + \cdots + \beta_{{\sf center}_{c \lambda}(i,j)} < (c \lambda)_1 + \cdots +(c \lambda)_{{\sf center}_{c \lambda}(i,j)}$:} The case hypothesis implies that  $c \lambda \nleq_{{\sf dom}} \tau$. This contradicts Corollary~\ref{cor:weightsdom}.

\noindent $(\Leftarrow)$ Since $[x^\beta]\kappa_{c \lambda}\neq 0$, we know there exists an $F \in {\sf Tab}(c \lambda)$ with ${\sf wt}(F)=\beta$. There are four cases to consider.

\noindent \textit{Case ${\sf leftmin}_{c \lambda}(i) = (c \lambda)_i$\text{ and }${\sf rightmax}_{c \lambda}(j) = (c \lambda)_{j}$}: By Lemma~\ref{lem:lminrmax}(iii), $\beta_i = (c \lambda)_i$ and  $\beta_{j} = (c \lambda)_{j}$. Thus
\begin{equation}
(c \lambda)_i = {\sf leftmin}_{c \lambda}(i) \leq \beta_{j} - (j-i) = (c \lambda)_{j} - (j-i),
\end{equation}
where the first equality is the case hypothesis, the inequality is the proposition hypothesis. Thus $j > i$ implies that $(c \lambda)_i < (c \lambda)_j$. This is a contradiction of Lemma~\ref{lem:clamblockdecrease}, and hence this case cannot occur.

\noindent \textit{Case ${\sf leftmin}_{c \lambda}(i) = (c \lambda)_i$\text{ and }${\sf rightmax}_{c \lambda}(j) > (c \lambda)_{j}$}:  By Lemma \ref{lem:rectanglesw}, $F(d, r) = d$ for all $1 \leq d \leq i$ and $r \leq (c \lambda)_i$. Hence, there is an $i$ in every row $r \leq (c \lambda)_i$ of $F$, and 
\begin{equation}
\label{eq:ababab2021}
\beta_i \geq (c \lambda)_i.
\end{equation}

The flagged property of $F$, combined with Lemma \ref{lem:formofclam}(i), implies that either $i$ or $j$, but not both, are in row $r > (c \lambda)_i$ of $F$. By the definition of $F$ and \eqref{eq:ababab2021}, there are exactly $\beta_i - (c \lambda)_i \geq 0$ such rows containing only $i$, but not $j$. By the case and proposition hypotheses, 
\[\beta_i - (c \lambda)_i = \beta_i - {\sf leftmin}_{c \lambda}(i) \geq \beta_i - (\beta_{j} - (j-i)) > 0.\] Setting $v := \beta_i - (\beta_{j} - (j-i))$ we can choose $v$ rows $r_1,\ldots,r_v > (c \lambda)_i$ in $F$ that contain $i$ and not $j$. 

The filling $G$ is obtained from $F$ by changing the $i$ in rows $r_1,\ldots,r_v$ to a $j$. By construction, $G$ is row distinct. For $i \leq k \leq j$, the boxes $(k, r_1),\ldots(k,r_v) \notin D(c \lambda)$  since $r_1,\ldots,r_v > (c \lambda)_i \geq (c \lambda)_k$. Hence the flagged property of $F$ implies that the $i$ in these rows of $F$ must appear in a column strictly greater than $j$. Thus the $j$ in these rows of $G$ appears in a column greater than $j$, and $G$ is flagged.

Let $\tau = {\sf wt}(G)$. Then $\tau_i = \beta_i - v = \beta_i - (\beta_i - (\beta_{j} - (j-i))) = \beta_{j} - (j-i)$, $\tau_{j} = \beta_{j} + v = \beta_{j} + (\beta_i - (\beta_{j} - (j-i))) = \beta_{i} + (j-i)$. Otherwise, $\tau_k = \beta_k$ for $r \neq i, j$. Thus $\tau = t_{i,j} \beta$. We conclude that $t_{i,j} \beta$ is an exponent vector of $\kappa_{c \lambda}$.

\noindent \textit{Case ${\sf leftmin}_{c \lambda}(i) < (c \lambda)_i$\text{ and }${\sf rightmax}_{c \lambda}(j) = (c \lambda)_{j}$}: 
The row distinct and flagged properties of $F$, combined with Lemma~\ref{lem:rectanglesw} and Lemma~\ref{lem:formofclam}(ii), imply that at least one of $i$ or $j$ are in row $r$ of $F$ for $1 \leq r \leq (c \lambda)_{j}$. By the case and proposition hypotheses, $\beta_{j} < \beta_{i} + (j-i) \leq {\sf rightmax}_{c \lambda}(j) = (c \lambda)_{j}$.

 Hence, there are at least $(c \lambda)_{j} - \beta_{j}$ rows $r$, with $1 \leq r \leq (c \lambda)_{j}$, of $F$ that contain $i$ but not $j$. Setting $v := \beta_i + (j-i) - \beta_{j} \leq (c \lambda)_{j} - \beta_{j}$, we choose $v$ rows in $F$, $r_1,\ldots,r_v \leq (c \lambda)_{j}$, that contain $i$ but not $j$. By Lemma \ref{lem:formofclam}(ii) and the flagged property of $F$, for each $e \in \{r_1,\ldots,r_v\}$ there is exactly one $d_e \leq j$ such that $(d_e, e) \notin D(c \lambda)$. It follows, by the definition of $e$ and the flagged property of $F$, that the content of row $e$ in the first $j$ columns of $F$ is equal to $\{1,\ldots,j-1\}$. We use this fact to define the filling $G$.

The filling $G$ is obtained from $F$ via the following rule. Let $1 \leq e \leq \lambda_1$, then
\begin{itemize}
    \item[(i)] $e \notin \{r_1,\ldots,r_v\}$: The $e$-th row of $G$ equals the $e$-th row of $F$.
    \item[(ii)] $e \in \{r_1,\ldots,r_v\}$: The $e$-th row of $G$ is defined by filling each of the values in $[j] \setminus \{i\}$ in the minimal column possible. Explicitly, $G(d, e)=d$ for $d<d_e$, $G(d, e)=d-1$ for $d_e < d \leq i$, $G(d, e)=d$ for $i < r < j$. Then, set $G(j, e)=j$, and for any column greater than $j$ the entries in row $e$ of $F$ and $G$ coincide.
\end{itemize}

Clearly $G$ is row distinct; for $e \in \{r_1,\ldots,r_v\}$, the content of row $e$ of $G$ is equal to the content of row $e$ of $F$ with the unique $i$ replaced by $j$. It is equally easy to verify that each of (i)-(ii) leaves the respective column in $G$ satisfying the flagged constraint. Let $\tau = {\sf wt}(G)$. Then $\tau_i = \beta_i - v = \beta_i - (\beta_i + (j-i) - \beta_{j}) = \beta_{j} - (j-i)$, $\tau_{j} = \beta_{j} + v = \beta_{j} + (\beta_i + (j-i) - \beta_{j}) = \beta_{i} + (j-i)$. Otherwise, $\tau_k = \beta_k$ for $k \neq i, j$. Thus $\tau = t_{i,j} \beta$. We conclude that $t_{i,j} \beta$ is an exponent vector of $\kappa_{c \lambda}$.

\noindent \textit{Case ${\sf leftmin}_{c \lambda}(i) < (c \lambda)_i$\text{ and }${\sf rightmax}_{c \lambda}(j) > (c \lambda)_{j}$}: Let $x$ be an integer such that $x < i$ and $(c \lambda)_x = {\sf leftmin}_{c \lambda}(i)$, and $y$ be an integer such that $y > j$ and $(c \lambda)_y = {\sf rightmax}_{c \lambda}(j)$. Suppose, for sake of contradiction, that $(c \lambda)_x \geq (c \lambda)_y$. Then, by Lemma~\ref{lem:lminrmax}(i), $\beta_i \geq (c \lambda)_x \geq (c \lambda)_y = {\sf rightmax}_{c \lambda}(j)$. Thus, $\beta_i + (j-i) > {\sf rightmax}_{c \lambda}(j)$, which contradicts the hypothesis (2). Thus,
\begin{equation}
\label{eq:case4leftminlessthanrightmax}
(c \lambda)_x < (c \lambda)_{y}.
\end{equation}

Corollary~\ref{cor:weightsdom} implies $\beta_1 + \cdots + \beta_{{\sf center}_{c \lambda}(i,j)} \geq (c \lambda)_1 + \cdots + (c \lambda)_{{\sf center}_{c \lambda}(i,j)}$. Now
\begin{equation}
\label{eq:centerijoverflow}
\begin{split}
| \left\{ (d, r) \in D(c \lambda) : d > {{\sf center}_{c \lambda}(i,j)}\text{, }1 \leq r \leq (c \lambda)_y \text{, and }F(d, r) \leq {{\sf center}_{c \lambda}(i,j)} \right\}| = & \\
\beta_1 + \cdots + \beta_{{\sf center}_{c \lambda}(i,j)} - ((c \lambda)_1 + \cdots + (c \lambda)_{{\sf center}_{c \lambda}(i,j)}) & \\
\end{split}
\end{equation}
Then our hypothesis $\beta_1 + \cdots + \beta_{i-1} + (\beta_{j} - (j-i)) + \beta_{i+1} + \cdots + \beta_{{\sf center}_{c \lambda}(i,j)} \geq (c \lambda)_1 + \cdots + (c \lambda)_{{\sf center}_{c \lambda}(i,j)}$ is equivalent to $\beta_1 + \cdots + \beta_{{\sf center}_{c \lambda}(i,j)} - ((c \lambda)_1 + \cdots + (c \lambda)_{{\sf center}_{c \lambda}(i,j)}) \geq \beta_i - (\beta_{j} - (j-i))$. Applying this to \eqref{eq:centerijoverflow} yields
\begin{equation}
\label{eq:case4claimboxeslessei}
\begin{split}
    | \left\{ (d, r) \in D(c \lambda) : d > {{\sf center}_{c \lambda}(i,j)}\text{, }1 \leq r \leq (c \lambda)_y \text{, and }F(d, r) \leq {{\sf center}_{c \lambda}(i,j)} \right\}| \geq & \\
    \beta_i - (\beta_{j} - (j-i)). & \\
\end{split}
\end{equation}
We can further refine \eqref{eq:case4claimboxeslessei}. By the definition of ${\sf center}_{c \lambda}(i,j)$, \eqref{eq:case4leftminlessthanrightmax}, and Lemma~\ref{lem:clamblockdecrease}, $(c \lambda)_{d} \geq (c \lambda)_x$ for all $i \leq d \leq {\sf center}_{c \lambda}(i,j)$. By Lemma \ref{lem:rectanglesw}, $F(d, r) = d$ for all $d \leq {\sf center}_{c \lambda}(i,j)$ and $r \leq (c \lambda)_x$. Thus, the row distinct property of $F$ transforms \eqref{eq:case4claimboxeslessei} into
\begin{multline}
\label{eq:case4claimboxeslesseirefined}
    | \left\{ (d, r) \in D(c \lambda) : d > {{\sf center}_{c \lambda}(i,j)}\text{, }(c \lambda)_x < r \leq (c \lambda)_y \text{, and }F(d, r) \leq {\sf center}_{c \lambda}(i,j) \right\} |\\  \geq \beta_i - (\beta_{j} - (j-i)).
\end{multline}
By Lemma \ref{lem:formofclam}(iii), the rows $(c\lambda)_x< r\leq c\lambda\leq (c\lambda)_y$ have ${\sf center}_{c \lambda}(i,j)$ boxes in $D(c\lambda)$. By \eqref{eq:case4claimboxeslesseirefined}, we can pick $v := \beta_i - (\beta_{j} - (j-i))$ of these rows, where the ${\sf center}_{c \lambda}(i,j)$ many boxes of $D(c\lambda)$ are filled using precisely the labels $1,2,\ldots,{\sf center}_{c \lambda}(i,j)$. By Lemma \ref{lem:formofclam}(iii), for each $e \in \{r_1,\ldots,r_v\}$ there is exactly one $d_e \leq {\sf center}_{c \lambda}(i,j) $ such that $(d_e, e) \notin D(c \lambda)$.

The filling $G$ is obtained from $F$ via the following rule. Let $1 \leq e \leq \lambda_1$.
\begin{itemize}
    \item[(i)] $e \notin \{r_1,\ldots,r_v\}$: The $e$-th row of $G$ equals the $e$-th row of $F$.
    \item[(ii)] The $e$-th row of $G$ is defined by filling each of the values in $[{\sf center}_{c \lambda}(i,j)-1] \setminus \{i\}$ in the minimal column possible. Explicitly, $G(d, e)= d$ for $d<d_e$, $G(d, e)=d-1$ for $d_e < d \leq i$, $G(d, e)=d$ for $i < d \leq {\sf center}_{c \lambda}(i,j)$. Then set the value of the unique box in a column greater than ${\sf center}_{c \lambda}(i,j)$ to be $j$.
\end{itemize}
Clearly $G$ is row distinct; for $e \in \{r_1,\ldots,r_v\}$, the content of row $e$ of $G$ is equal to the content of row $e$ of $F$ with the unique $i$ replaced by $j$. It is an easy check to verify that each of (i)-(ii) leaves the respective row in $G$ satisfying the flagged constraint. Let $\tau = {\sf wt}(G)$. Then $\tau_i = \beta_i - v = \beta_i - (\beta_i - (\beta_{j} - (j-i))) = \beta_{j} - (j-i)$, $\tau_{j} = \beta_{j} + v = \beta_{j} + (\beta_i - (\beta_{j} - (j-i))) = \beta_{i} + (j-i)$. Otherwise, $\tau_k = \beta_k$ for $k \neq i, j$. Thus $\tau = t_i \beta$. We conclude that $t_i \beta$ is an exponent vector of $\kappa_{c \lambda}$.
\end{proof}

\begin{lemma}
\label{lemma:overflowintorightmax} Let $i \in [n-1]$ and $\beta \in {\mathcal P}_{c \lambda,{\bf \gamma}}$. Then $$\max \{ {\sf rightmax}_{c \lambda}(i + 1) - {\sf leftmin}_{c \lambda}(i), 0 \} \geq \beta_1 + \cdots + \beta_i - ((c \lambda)_1 + \cdots + (c \lambda)_i).$$ 
\end{lemma}
\begin{proof}
Since $[x^\beta]\kappa_{c \lambda}\neq 0$, there exists an $F \in {\sf Tab}(c \lambda)$ with ${\sf wt}(F)=\beta$. Now,
\begin{equation}
\label{eq:overflowintorightmax:main}
| \left\{ (d, r) \in D(c \lambda) : d > i \text{ and } F(d, r) \leq i \right\}| = \beta_1 + \cdots + \beta_i - ((c \lambda)_1 + \cdots + (c \lambda)_i)
\end{equation}

We first prove this lemma for $i \in [n-1]$ with $(c \lambda)_i \geq (c \lambda)_{i+1}$.
Let $r \leq {\sf leftmin}_{c \lambda}(i)$. Lemma \ref{lem:rectanglesw} implies that $F(d_1, r) = d_1$ for $d_1 \leq i$ and $r \leq {\sf leftmin}_{c \lambda}(i)$. Since $F$ is row distinct this implies
\begin{equation}
\label{eq:overflowintorightmax:bottomleftrectangle}
| \left\{ (d, r) \in D(c \lambda) : d > i \text{ and } F(d, r) \leq i \right\}| = 0 \text{ for }1 \leq r \leq {\sf leftmin}_{c \lambda}(i).
\end{equation}

Suppose ${\sf leftmin}_{c \lambda}(i) \geq {\sf rightmax}_{c \lambda}(i+1)$. Then there exist no $(d, r) \in D(c \lambda)$ such that $d > i$ and $r > {\sf leftmin}_{c \lambda}(i)$. This, combined with \eqref{eq:overflowintorightmax:main} and \eqref{eq:overflowintorightmax:bottomleftrectangle}, implies $\beta_1 + \cdots + \beta_i - ((c \lambda)_1 + \cdots + (c \lambda)_i) = 0$. Thus our result trivially holds.

For the rest of the proof we assume ${\sf leftmin}_{c \lambda}(i) < {\sf rightmax}_{c \lambda}(i+1)$. 

\noindent \textit{Case ${\sf leftmin}_{c \lambda}(i) = (c \lambda)_i$\text{ and }${\sf rightmax}_{c \lambda}(i+1) = (c \lambda)_{i+1}$}: By our assumption $(c \lambda)_i \geq (c \lambda)_{i+1}$ and the case hypothesis, ${\sf leftmin}_{c \lambda}(i) = (c \lambda)_i \geq (c \lambda)_{i+1} = {\sf rightmax}_{c \lambda}(i+1)$. Thus, since we are assuming ${\sf leftmin}_{c \lambda}(i) < {\sf rightmax}_{c \lambda}(i+1)$, this case does not occur.

\noindent \textit{Case ${\sf leftmin}_{c \lambda}(i) = (c \lambda)_i$\text{ and }${\sf rightmax}_{c \lambda}(i+1) > (c \lambda)_{i+1}$}: We have that ${\sf leftmin}_{c \lambda}(i) = (c \lambda)_i$ paired with \eqref{eq:overflowintorightmax:bottomleftrectangle}, and combined with Lemma \ref{lem:formofclam}(i) implies $$| \left\{ (d, r) \in D(c \lambda) : d > i\text{ and } F(d, r) \leq i \right\}| \leq {\sf rightmax}_{c \lambda}(i + 1) - {\sf leftmin}_{c \lambda}(i).$$ Then \eqref{eq:overflowintorightmax:main} gives the required inequality.

\noindent \textit{Case ${\sf leftmin}_{c \lambda}(i) < (c \lambda)_i$\text{ and }${\sf rightmax}_{c \lambda}(i+1) = (c \lambda)_{i+1}$}: Lemma \ref{lem:formofclam}(ii) says $$| \left\{ (d,r) \in D(c \lambda) : d \leq i + 1 \right\} | \geq i\text{ for }{\sf leftmin}_{c \lambda}(i) < r \leq (c \lambda)_{i+1},$$ which implies
\begin{equation}
\label{eq:overflow:leftisalmostfull}
| \left\{ (d,r) \in D(c \lambda) : d \leq i \right\} | \geq i-1 \text{ for }{\sf leftmin}_{c \lambda}(i) < r \leq (c \lambda)_{i+1}.
\end{equation} 
Since ${\sf rightmax}_{c \lambda}(i+1) = (c \lambda)_{i+1}$, there exist no $(d, r) \in D(c \lambda)$ such that $d > i$ and $r > (c \lambda)_{i+1}$. This, combined with \eqref{eq:overflowintorightmax:bottomleftrectangle}, and the row distinct property of $F$ paired with \eqref{eq:overflow:leftisalmostfull}, implies that $$| \left\{ (d, r) \in D(c \lambda) : d > i \text{ and } F(d, r) \leq i \right\}| \leq {\sf rightmax}_{c \lambda}(i + 1) - {\sf leftmin}_{c \lambda}(i).$$ Applying \eqref{eq:overflowintorightmax:main} concludes the proof in this case.

\noindent \textit{Case ${\sf leftmin}_{c \lambda}(i) < (c \lambda)_k$\text{ and }${\sf rightmax}_{c \lambda}(i+1) > (c \lambda)_{i+1}$}: There exist no $(d, r) \in D(c \lambda)$ such that $d > i$ and $r > {\sf rightmax}_{c \lambda}(i+1)$. We apply Lemma \ref{lem:formofclam}(iii), noting that 
${\sf center}_{c \lambda}(i,i+1) = i$, and \eqref{eq:overflowintorightmax:bottomleftrectangle} to imply that $$| \left\{ (d, r) \in D(c \lambda) : d > i \text{ and } F(d, r) \leq i \right\}| \leq {\sf rightmax}_{c \lambda}(i + 1) - {\sf leftmin}_{c \lambda}(i).$$ Once again we conclude after applying \eqref{eq:overflowintorightmax:main}.

This completes the proof for $i$ such that $(c \lambda)_i \geq (c \lambda)_{i+1}$. Otherwise, $i \in [n-1]$ with $(c \lambda)_i < (c \lambda)_{i+1}$. If $i = 1$ or $i = n - 1$ the proof is straightforward. Otherwise, let $x < i < i+1 < y$. Then 
\begin{equation}
\label{eq:iplus1isrightmax}
(c \lambda)_{i+1} \geq (c \lambda)_{y}
\end{equation}
and $(c \lambda)_{x} \geq (c \lambda)_{i}$ by Lemma~\ref{lem:clampattern} ($012$-avoidance). If $(c \lambda)_x < (c \lambda)_y$, then $c \lambda$ contains the composition pattern $012$, $1032$, $0021$, $0011$, or $1022$. This contradicts Lemma~\ref{lem:clampattern}. Thus $(c \lambda)_x \geq (c \lambda)_y$ for all $x < i$. This implies ${\sf leftmin}_{c \lambda}(i-1) \geq (c \lambda)_y$ for all $i+1 < y$. We conclude  ${\sf leftmin}_{c \lambda}(i-1) \geq {\sf rightmax}_{c \lambda}(i+2)$. By Lemma~\ref{lem:rectanglesw} (the second displayed equation, where we have applied it to $i-1$) and the row distinct property of $F$, this implies
\begin{equation}
\label{eq:upperbd1}
| \left\{ (d, r) \in D(c \lambda) : d > i \text{, } 1 \leq r \leq {\sf rightmax}_{c \lambda}(i+2) \text{, and } F(d, r) \leq i - 1 \right\}| = 0.
\end{equation}
Then ${\sf leftmin}_{c \lambda}(i) = (c \lambda)_{i}$ by Lemma~\ref{lem:clampattern} ($012$-avoidance) and, combined with Lemma~\ref{lem:rectanglesw} applied to $i$, this implies
\begin{equation}
\label{eq:upperbd2}
| \left\{ (d, r) \in D(c \lambda) : d > i \text{, } 1 \leq r \leq {\sf leftmin}_{c \lambda}(i) \text{, and } F(d, r) \leq i \right\}| = 0.
\end{equation}
Now
\[
\begin{split}
| \left\{ (d, r) \in D(c \lambda) : d > i \text{ and } F(d, r) \leq i \right\}| & \leq (c \lambda)_{i+1} - {\sf leftmin}_{c \lambda}(i) \\  &  = {\sf rightmax}_{c \lambda}(i+1) - {\sf leftmin}_{c \lambda}(i). \\
\end{split}
\]
The inequality comes by studying the intervals $[1,{\sf leftmin}_{c\lambda}(i)]$, 
$({\sf leftmin}_{c\lambda}(i),{\sf rightmax}_{c\lambda}(i+2)]$, and $({\sf rightmax}_{c\lambda}(i+2),(c\lambda)_{i+1}]$. Respectively, we use \eqref{eq:upperbd2}, and \eqref{eq:upperbd1}  paired with the row distinct property of $F$, for the first two intervals. For the third interval, we use the fact that there is at most one column, namely $y=i+1$, such that $y>d$ and $(c\lambda)_y>{\sf rightmax}_{c\lambda}(i+2)$. The equality follows from \eqref{eq:iplus1isrightmax}.
\end{proof}

\begin{lemma}
\label{lem:notbothinterweaved}
Let $i<p<j<q$ be in the same block and $\beta \in {\mathcal P}_{c \lambda,{\bf \gamma}}$. If $(i,j)$ and $(p,q)$ are interweaved, $\beta<_{\text{\emph{Bruhat}}} t_{i,j}\beta$, and
$\beta<_{\text{\emph{Bruhat}}} t_{p,q}\beta$, then $t_{i,j} \beta \notin {\mathcal P}_{c \lambda,{\bf \gamma}}$ or $t_{p,q} \beta \notin {\mathcal P}_{c \lambda,{\bf \gamma}}$. 
\end{lemma}
\begin{proof}
If $(i,j)$ and $(p,q)$ are interweaved, then it is straightforward that ${\sf center}_{c \lambda}(i,j) = {\sf center}_{c \lambda}(p,q)$. Lemma~\ref{lem:clamblockdecrease} and the definition of ${\sf center}_{c \lambda}(i,j)$ implies
\[
(c \lambda)_k = {\sf rightmax}_{c \lambda}(k)\text{ for }i \leq k \leq {\sf center}_{c \lambda}(p,q),
\]
which in turn implies, via Lemma~\ref{lem:lminrmax}(ii), that
\begin{equation}
\label{eq:minusbeforecenter}
\beta_k - (c \lambda)_k \leq 0\text{ for }i \leq k \leq {\sf center}_{c \lambda}(p,q).
\end{equation}
In a similar fashion, the definition of ${\sf center}_{c \lambda}(i,j)$ and Lemma~\ref{lem:clampattern} ($012$-avoidance) implies $(c \lambda)_k = {\sf leftmin}_{c \lambda}(k)$ for ${\sf center}_{c \lambda}(p,q) < k \leq q$. Hence, Lemma~\ref{lem:lminrmax}(i) says
\begin{equation}
\label{eq:plusaftercenter}
\beta_k - (c \lambda)_k \geq 0\text{ for }{\sf center}_{c \lambda}(p,q) < k \leq q.
\end{equation}

Suppose that $t_{i,j} \beta, t_{p,q} \beta \in {\mathcal P}_{c \lambda,{\bf \gamma}}$. Let $C := \beta_1 + \cdots + \beta_{{\sf center}_{c \lambda}(i,j)} - ((c \lambda)_1 + \cdots + (c \lambda)_{{\sf center}_{c \lambda}(i,j)})$. Then, 
\begin{equation}
\label{eq:implicationsofoverflow}
\begin{split}
{\sf rightmax}_{c \lambda}(q) - (c \lambda)_q & = {\sf rightmax}_{c \lambda}(q+1) - {\sf leftmin}_{c \lambda}(q)\\
& \geq C + (\beta_{{\sf center}_{c \lambda}(i,j) + 1} + \cdots + \beta_{q}) - ((c \lambda)_{{\sf center}_{c \lambda}(i,j) + 1} + \cdots (c \lambda)_{q}) \\
& \geq C + (\beta_j - (c \lambda)_j) + (\beta_q - (c \lambda)_q), \\
\end{split}
\end{equation}
where the equality follows from the interweaving assumption combined with Lemma~\ref{lem:clampattern} ($012$-avoidance), the first inequality comes from
Lemma~\ref{lemma:overflowintorightmax} applied to $\beta$, and the final inequality
follows from \eqref{eq:plusaftercenter}.

Proposition~\ref{prop:goingup}(3) says
\begin{equation}
\label{eq:conseqijgoingupinterweave}
\beta_1 + \cdots + \beta_{i-1} + (\beta_{j} - (j-i)) + \beta_{i+1} + \cdots + \beta_{{\sf center}_{c \lambda}(i,j)} \geq (c \lambda)_1 + \cdots + (c \lambda)_{{\sf center}_{c \lambda}(i,j)},
\end{equation}
\begin{equation}
\label{eq:conseqpqgoingupinterweave}
\beta_1 + \cdots + \beta_{p-1} + (\beta_{q} - (q-p)) + \beta_{p+1} + \cdots + \beta_{{\sf center}_{c \lambda}(i,j)} \geq (c \lambda)_1 + \cdots + (c \lambda)_{{\sf center}_{c \lambda}(i,j)}. 
\end{equation}
Let $D := (\beta_{i+1} + \cdots + \beta_{{\sf center}_{c \lambda}(i,j)})-((c \lambda)_{i+1} + \cdots + (c \lambda)_{{\sf center}_{c \lambda}(i,j)})$. Reformulating \eqref{eq:conseqijgoingupinterweave} yields
\begin{equation}
\begin{split}
\label{eq:absolutemess1}
0 & \leq  (\beta_1 + \cdots + \beta_{i-1})-((c \lambda)_1 + \cdots + (c \lambda)_{i-1}) +  (\beta_j - (j-i) - (c \lambda)_i) + D \\
& \leq (c \lambda)_i -  {\sf leftmin}_{c \lambda}(i) +  (\beta_j - (j-i) - (c \lambda)_i) + D \\
& = (c \lambda)_i -  {\sf leftmin}_{c \lambda}(i) + ((c \lambda)_j + (\beta_j - (c \lambda)_j) - (j-i) - (c \lambda)_i) + D \\
& = ((c \lambda)_j - {\sf leftmin}_{c \lambda}(i)) + (\beta_j - (c \lambda)_j) - (j-i) + D \\
& \leq (\beta_j - (c \lambda)_j) - (j-i) + D, \\
\end{split}
\end{equation}
where the second inequality is via Lemma~\ref{lemma:overflowintorightmax} applied to $\beta$ (note
$(c\lambda)_i={\sf rightmax}_{c\lambda}(i)$ here), and the final inequality follows from Lemma~\ref{lem:clampattern} ($012$-avoidance).

Let $E := (\beta_{p+1} + \cdots + \beta_{{\sf center}_{c \lambda}(i,j)})-((c \lambda)_{p+1} + \cdots + (c \lambda)_{{\sf center}_{c \lambda}(i,j)})$. Reformulating \eqref{eq:conseqpqgoingupinterweave},
\begin{equation}
\begin{split}
\label{eq:absolutemess2}
0 & \leq  (\beta_1 + \cdots + \beta_{p-1})-((c \lambda)_1 + \cdots + (c \lambda)_{p-1}) +  (\beta_q - (q-p) - (c \lambda)_p) + E \\
& \leq (\beta_1 + \cdots + \beta_{i})-((c \lambda)_1 + \cdots + (c \lambda)_{i}) +  (\beta_q - (q-p) - (c \lambda)_p) + E \\
& \leq (\beta_1 + \cdots + \beta_{i})-((c \lambda)_1 + \cdots + (c \lambda)_{i}) +  (\beta_q - (q-p) - (c \lambda)_p) \\
& \leq (\beta_1 + \cdots + \beta_{i})-((c \lambda)_1 + \cdots + (c \lambda)_{i}) +  (\beta_q - (q-p) - {\sf rightmax}_{c \lambda}(q)) \\
& = (\beta_1\! +\! \cdots\! +\! \beta_{i})\!-\!((c \lambda)_1 \!+\! \cdots \!+\! (c \lambda)_{i}) \!+\! ((c \lambda)_q \!+\! (\beta_q \!-\! (c \lambda)_q) \!-\! (q-p) \!-\! {\sf rightmax}_{c \lambda}(q).
\end{split}
\end{equation}
where the second and third inequality are by \eqref{eq:minusbeforecenter}, the fourth inequality is by Lemma~\ref{lem:clampattern} ($012$-avoidance).

Adding \eqref{eq:absolutemess1} and \eqref{eq:absolutemess2} we have
\begin{equation}
0 \leq C + (\beta_j - (c \lambda)_j) + (\beta_q - (c \lambda)_q) - (j-i) - (q-p) + ((c \lambda)_q - {\sf rightmax}_{c \lambda}(q)) 
\end{equation}
which can be reformulated into
\begin{equation}
\begin{split}
{\sf rightmax}_{c \lambda}(q) - (c \lambda)_q & \leq C + (\beta_j - (c \lambda)_j) + (\beta_q - (c \lambda)_q) - (j-i) - (q-p) \\
& < C + (\beta_j - (c \lambda)_j) + (\beta_q - (c \lambda)_q)
\end{split}
\end{equation}
This, combined with \eqref{eq:implicationsofoverflow}, gives our desired contradiction. We conclude that $t_{i,j} \beta \notin {\mathcal P}_{c \lambda,{\bf \gamma}}$ or $t_{p,q} \beta \notin {\mathcal P}_{c \lambda,{\bf \gamma}}$.
\end{proof}

\noindent\emph{Conclusion of the proof of Theorem~\ref{thm:diamond}:}
Without loss of generality assume $(i,j) < (p,q)$ in lexicographic order. Both $\tau := t_{i,j}\beta$ and $\phi := t_{p,q}\beta$ cover $\beta$, thus
\begin{equation}
\label{eq:betaijineq}
\beta_i > \beta_j - (j-i) = \tau_i,
\end{equation}
\begin{equation}
\label{eq:betapqineq}
\beta_p > \beta_q - (q-p) = \phi_p.
\end{equation}
By Proposition~\ref{prop:goingup}, we have
\begin{equation}
\label{eq:betaijright}
\beta_i + (j-i) \leq {\sf rightmax}_{c \lambda}(j),
\end{equation}
\begin{equation}
\label{eq:betapqright}
\beta_p + (q-p) \leq {\sf rightmax}_{c \lambda}(q),
\end{equation}
\begin{equation}
\label{eq:betaijleft}
\beta_j - (j-i) \geq {\sf leftmin}_{c \lambda}(i),
\end{equation}
\begin{equation}
\label{eq:betapqleft}
\beta_q - (q-p) \geq {\sf leftmin}_{c \lambda}(p).
\end{equation}
Moreover, for the same reason, if $(i,j)$ is interweaved, then
\begin{equation}
\small
\label{eq:betaijcenter}
\beta_1 + \cdots + \beta_{i-1} + \beta_j - (j-i) + \beta_{i+1} + \cdots + \beta_{{\sf center}_{c \lambda}(i,j)} \geq (c \lambda)_1 + \cdots + (c \lambda)_{{\sf center}_{c \lambda}(i,j)}.
\end{equation}
If $(p,q)$ is interweaved, then
\begin{equation}
\small
\label{eq:betapqcenter}
\beta_1 + \cdots + \beta_{p-1} + \beta_q - (q-p) + \beta_{p+1} + \cdots + \beta_{{\sf center}_{c \lambda}(p,q)} \geq (c \lambda)_1 + \cdots + (c \lambda)_{{\sf center}_{c \lambda}(p,q)}.
\end{equation}
We now consider five cases depending on the overlap in the values $(i,j)$ and $(p,q)$. In what
follows, we will make repeated use of Lemma~\ref{lemma:posetBruhatorder}(i), which characterizes
the covering relation in $({\mathcal S}_{\gamma,I},<_{\text{Bruhat}})$.

\noindent \textit{Case 1.1 ($i$ and $p$ in the same block, $i=p$, $j < q$):} Suppose, for  contradiction, that $\beta_j - (j-i) = \beta_q - (q-p)$. Then, since $i = p$, this equality is equivalent to $\beta_j = \beta_q - (q - j)$. The contradicts Lemma~\ref{lemma:posetBruhatorder}(i), and hence $\beta_j - (j-i) \neq \beta_q - (q-p)$.

\noindent \textit{Subcase 1.1.1 $\beta_j - (j-i) > \beta_q - (q-p)$:} By the subcase hypothesis, $t_{i,j} \beta <_{\text{Bruhat}} t_{p,q} t_{i,j} \beta$. Then $t_{p,q} t_{i,j} \beta <_{\text{Bruhat}} t_{j,q} t_{p,q} t_{i,j} \beta$ by \eqref{eq:betaijineq}. Combining, we have $t_{i,j} \beta <_{\text{Bruhat}} t_{j,q} t_{p,q} t_{i,j} \beta = t_{p,q} \beta$. This contradicts the hypothesis that $t_{p,q} \beta$ covers $\beta$. Hence this subcase cannot occur.

\noindent \textit{Subcase 1.1.2 $\beta_j - (j-i) < \beta_q - (q-p)$:} We will show that $t_{i,j} \phi \in {\mathcal P}_{c \lambda,{\bf \gamma}}$. By the subcase hypothesis, the definition of $\phi$, and $i=p$,
\begin{equation}
\label{eq:prophypofirst:1.1.1}
\phi_i = \phi_p = \beta_q - (q-p) \geq \beta_j - (j-i-1) = \phi_j - (j-i-1).
\end{equation}
By \eqref{eq:betapqineq}, \eqref{eq:betaijright}, and $i=p$ we have
\begin{equation}
\label{eq:prophyposecond:1.1.1}
\phi_i + (j - i) = \beta_q - (q-p) + (j-i) < \beta_p + (j - i) = \beta_i + (j - i) \leq {\sf rightmax}_{c \lambda}(j).
\end{equation}
Since $\phi_j = \beta_j$, by \eqref{eq:betaijleft},
\begin{equation}
\label{eq:prophypothird:1.1.1}
\phi_j - (j-i) = \beta_j - (j-i) \geq {\sf leftmin}_{c \lambda}(i).
\end{equation}

Finally, $\phi_r = \beta_r$ for $r \neq p , q$. If $(i,j)$ is interweaved, then \eqref{eq:betaijcenter} and $i=p$ combined with the previous sentence implies
\begin{align}
\label{eq:prophypofourth:1.1.1}
\begin{split}
\small
\phi_1 \!+\! \cdots \!+\! \phi_{i-1} \!+\! \phi_j \!-\! (j-i) \!+\! \phi_{i+1} + \cdots + \phi_{{\sf center}_{c \lambda}(i,j)} &  = \beta_1 + \cdots + \beta_{i-1} + \beta_j - (j-i) \\
& \,\,\,\, + \beta_{i+1} + \cdots + \beta_{{\sf center}_{c \lambda}(i,j)} \\
& \geq (c \lambda)_1 + \cdots + (c \lambda)_{{\sf center}_{c \lambda}(i,j)}. \\
\end{split}
\end{align}
The hypotheses of Proposition~\ref{prop:goingup} are satisfied for for $t_{i,j} \phi$ by \eqref{eq:prophypofirst:1.1.1}, \eqref{eq:prophyposecond:1.1.1}, \eqref{eq:prophypothird:1.1.1}, and \eqref{eq:prophypofourth:1.1.1}. Hence, $t_{i,j} \phi \in {\mathcal P}_{c \lambda,{\bf \gamma}}$. 

By \eqref{eq:prophypofirst:1.1.1}, $\phi < t_{i,j} \phi$. By \eqref{eq:betapqineq}, $\beta_i + (j-i) = \beta_p + (j-i) > \beta_q - (q-p) + (j-i) = \beta_q - (q-j)$, and hence $\tau = t_{i,j} \beta <_{\text{Bruhat}} t_{j,q} t_{i,j} \beta = t_{i,j} \phi$.

\noindent \textit{Case 1.2 ($i$ and $p$ in the same block, $i < p$, $j = p$):} In this case,
\begin{equation}
\label{eq:tauineq1.2.1}
\tau_p = \beta_i + (j-i) > \beta_j \geq \beta_q - (q-p-1) = \tau_q - (q-p-1),
\end{equation}
\begin{equation}
\label{eq:phiineq1.2.1}
\phi_i = \beta_i \geq \beta_j - (j-i-1) \geq \beta_q - (q-p) - (j-i-1) = \phi_j - (j-i-1).
\end{equation}

Before breaking into subcases we first prove that $\phi <_{\text{Bruhat}} t_{i,j} \phi, t_{p,q} \tau$, and $\tau <_{\text{Bruhat}} t_{i,j} \phi, t_{p,q} \tau$. First, $\phi <_{\text{Bruhat}} t_{i,j} \phi$ and $\tau <_{\text{Bruhat}} t_{p,q} \tau$ follow from \eqref{eq:phiineq1.2.1} and \eqref{eq:tauineq1.2.1}. Then, \eqref{eq:betaijineq} implies 
\[\phi_i = \beta_i > \beta_j - (j-i) = \beta_j + (q-j) + (q-i) = \phi_q - (q-i),\] 
and hence $\phi <_{\text{Bruhat}} t_{i,q} \phi = t_{p,q} \tau$. Finally, by \eqref{eq:betapqineq}, 
\[\tau_i = \beta_j - (j-i) > \beta_q - (q-j) - (j-i) = \tau_q - (q-i),\] 
and thus $\tau <_{\text{Bruhat}} t_{i,q} \tau = t_{i,j} \phi$. Hence, in all the following subcases, it remains to show that at least one of $t_{i,j} \phi$ or $t_{p,q} \tau$ are in ${\mathcal P}_{c \lambda,{\bf \gamma}}$.

\noindent \textit{Subcase 1.2.1 ${\sf leftmin}_{c \lambda}(i) = (c \lambda)_i$\text{ and }${\sf rightmax}_{c \lambda}(j) > (c \lambda)_{j}$:} By Lemma~\ref{lem:clamblockdecrease}, the subcase hypothesis implies
\begin{equation}
\label{eq:leftminall1.2.1}
{\sf leftmin}_{c \lambda}(k) = (c \lambda)_k\text{ for }i\leq k \leq q.
\end{equation}
By Lemma~\ref{lem:lminrmax}(i) this implies
\begin{equation}
\label{eq:bclamdiff1.2.1}
\beta_k \geq (c \lambda)_k\text{ for }i\leq k \leq q.
\end{equation}
In this subcase, \eqref{eq:betaijleft} and \eqref{eq:betapqleft} become
\begin{align}
\label{eq:ijdiff1.2.1} \beta_j - (j-i) & \geq (c \lambda)_i = (c \lambda)_j + ((c \lambda)_i - (c \lambda)_j),\\
\label{eq:pqdiff1.2.1} \beta_q - (q-p) & \geq (c \lambda)_p = (c \lambda)_q + ((c \lambda)_p - (c \lambda)_q). 
\end{align}
Thus
\begin{align}
\label{eq:main1.2.1}
\begin{split}
{\sf rightmax}_{c \lambda}(q) - (c \lambda)_q & = {\sf rightmax}_{c \lambda}(q+1) - {\sf leftmin}_{c \lambda}(q) \\
& \geq \beta_1 + \cdots + \beta_q - ((c \lambda)_1 + \cdots + (c \lambda)_q) \\
& = \left(\sum_{t=1}^{i-1}\beta_t-(c\lambda)_t\right) +[\beta_i-(c\lambda)_i]
+\left(\sum_{t=i+1, t\neq p }^{q-1} \beta_t\!-\!(c\lambda)_t\right) \\
&  \ \ \ \ \ \ \ \  \!+\![\beta_p-(c\lambda)_p]+[\beta_q-(c\lambda)_q]
\\
& \geq [\beta_i-(c\lambda)_i]+[\beta_p-(c\lambda)_p]+[\beta_q-(c\lambda)_q]\\
& \geq [\beta_i \!-\! (c \lambda_i)] \!+\! [(j\!-\!i) \!+\! ((c \lambda)_i \!-\! (c \lambda)_j)] \!+\! [(q\!-\!p) \!+\! ((c \lambda)_p \!-\! (c \lambda)_q)] \\
& = \beta_i + (q-i) - (c \lambda)_q,
\end{split}
\end{align}
where the first equality follows from the subcase hypotheses and \eqref{eq:leftminall1.2.1}, the first inequality from Lemma~\ref{lemma:overflowintorightmax} with ${\sf rightmax}_{c \lambda}(q+1) - {\sf leftmin}_{c \lambda}(q) \geq 0$, the second inequality by Corollary~\ref{cor:weightsdom}
, and \eqref{eq:bclamdiff1.2.1}, and third inequality is by
\eqref{eq:ijdiff1.2.1}, \eqref{eq:pqdiff1.2.1}, and the final equality by $p=j$. Rewriting \eqref{eq:main1.2.1}, we arrive at ${\sf rightmax}_{c \lambda}(q) \geq \beta_i + (q-i) = \tau_p + (q-p)$.  Further, by \eqref{eq:betapqleft}, $\tau_q - (q-p) = \beta_q - (q-p) \geq {\sf leftmin}_{c \lambda}(p)$. The hypotheses of Proposition~\ref{prop:goingup} are satisfied for $t_{p,q} \tau$ by the preceding two sentences, the subcase hypothesis, and \eqref{eq:tauineq1.2.1}. Hence, $t_{p,q} \tau \in {\mathcal P}_{c \lambda,{\bf \gamma}}$ (notice $(p,q)$ cannot be interweaved since $j=p$ and
$c\lambda$ is $012$-avoiding by Lemma~\ref{lem:clampattern}).

\noindent \textit{Subcase 1.2.2 ${\sf leftmin}_{c \lambda}(i) < (c \lambda)_i$\text{ and }${\sf rightmax}_{c \lambda}(j) = (c \lambda)_{j}$:} By the subcase hypotheses,
\begin{equation}
\label{eq:rightmaxall1.2.2.1}
{\sf rightmax}_{c \lambda}(k) = (c \lambda)_k\text{ for }i\leq k \leq j,
\end{equation}
and hence by Lemma~\ref{lem:lminrmax}(ii)
\begin{equation}
\label{eq:rightmaxall1.2.2}
\beta_k \leq (c \lambda)_k\text{ for }i\leq k \leq j.
\end{equation}
In this subcase, \eqref{eq:betaijright} becomes
\begin{align}
\label{eq:ijdiff1.2.2} \beta_i + (j-i) & \leq (c \lambda)_j = (c \lambda)_i + ((c \lambda)_j - (c \lambda)_i). 
\end{align}
By Corollary~\ref{cor:weightsdom} applied to $\phi=t_{p,q}\beta$,
\begin{equation}
\label{eq:central1.2.2}
\beta_1 \!+\! \cdots \!+\! \beta_{i-1} - ((c \lambda)_1 + \cdots + (c \lambda)_{i-1}) \!\geq\! - ((\beta_{i} + \cdots + \beta_{j-1} + (\beta_{q} - (q-p)) - ((c \lambda)_{i} + \!\cdots\! + (c \lambda)_j)).
\end{equation}
We conclude
\begin{align}
\label{eq:main1.2.2}
\begin{split}
(c \lambda)_i -  {\sf leftmin}_{c \lambda}(i) & = {\sf rightmax}_{c \lambda}(i) - {\sf leftmin}_{c \lambda}(i-1) \\
& \geq \beta_1 + \cdots + \beta_{i-1} - ((c \lambda)_1 + \cdots + (c \lambda)_{i-1}) \\
& \geq - ((\beta_{i} + \cdots + \beta_{j-1} + (\beta_{q} - (q-p)) - ((c \lambda)_{i} + \cdots + (c \lambda)_j)) \\
& = -\left(\sum_{t=i+1}^{j-1}\beta_t-(c\lambda)_t\right)-[\beta_i-(c\lambda)_i] - [\beta_{q} - (q-p) - (c \lambda)_j]\\
& \geq -[\beta_i-(c\lambda)_i] - [\beta_{q} - (q-p) - (c \lambda)_j]\\
& \geq -[-(j-i) + ((c \lambda)_j - (c \lambda)_i)] - [(\beta_{q} - (q-p) - (c \lambda)_j] \\
& = (q-i) - (\beta_q - (c \lambda)_i);
\end{split}
\end{align}
the first equality follows by the subcase hypotheses, the first inequality from Lemma~\ref{lemma:overflowintorightmax} with ${\sf rightmax}_{c \lambda}(i) - {\sf leftmin}_{c \lambda}(i-1) \geq 0$, the second inequality by \eqref{eq:central1.2.2}, the third inequality by 
\eqref{eq:rightmaxall1.2.2}, 
the fourth by \eqref{eq:ijdiff1.2.2}, and the final equality by $p=j$. Now \eqref{eq:main1.2.2} is equivalent to 
\[{\sf leftmin}_{c \lambda}(i) \leq \beta_q - (q-i) = \phi_j - (j-i).\]  
Further, by \eqref{eq:betaijright}, 
\[\phi_i + (j-i) = \beta_i + (j-i) \geq {\sf rightmax}_{c \lambda}(j).\] 
The hypotheses of Proposition~\ref{prop:goingup} are satisfied for for $t_{i,j} \phi$ by the preceding two sentences, the subcase hypothesis, and \eqref{eq:phiineq1.2.1}. Hence, $t_{i,j} \phi \in {\mathcal P}_{c \lambda,{\bf \gamma}}$.

\noindent \textit{Subcase 1.2.3 ${\sf leftmin}_{c \lambda}(i) < (c \lambda)_i$\text{, }${\sf rightmax}_{c \lambda}(j) > (c \lambda)_{j}$:} In this subcase, ${\sf leftmin}_{c \lambda}(j) = (c \lambda)_j$, since ${\sf leftmin}_{c \lambda}(j) < (c \lambda)_j$ would imply that $c \lambda$ contains $012$. Thus, since $(i,j)$ is interweaved, ${\sf center}_{c \lambda}(i,j) < j$ and Lemma~\ref{lem:clamblockdecrease} implies 
\begin{equation}
\label{eq:leftmin1.2.3}
{\sf leftmin}_{c \lambda}(k) = (c \lambda)_k\text{ for }{\sf center}_{c \lambda}(i,j) < k \leq q. 
\end{equation}
Corollary~\ref{cor:weightsdom}, applied to $\beta$ and $\tau$, respectively, implies $$(\beta_1 + \cdots + \beta_{{\sf center}_{c \lambda}(i,j)}) - ((c \lambda)_1 + \cdots + (c \lambda)_{{\sf center}_{c \lambda}(i,j)}) \geq 0,$$ and  $$(\beta_1 + \cdots + \beta_{i-1} + \beta_j - (j-i) + \beta_{i+1} + \cdots + \beta_{{\sf center}_{c \lambda}(i,j)}) - ((c \lambda)_1 + \cdots + (c \lambda)_{{\sf center}_{c \lambda}(i,j)}) \geq 0.$$ These two inequalities, combined with $\beta_i > \beta_j - (j-i)$, yield
\begin{equation}
\label{eq:thistookmethreedaystorealize1.2.3}
(\beta_1 + \cdots + \beta_{{\sf center}_{c \lambda}(i,j)}) - ((c \lambda)_1 + \cdots + (c \lambda)_{{\sf center}_{c \lambda}(i,j)}) \geq \beta_i - (\beta_j - (j-i)).
\end{equation}
Thus
\begin{align}
\label{eq:main1.2.3}
\begin{split}
{\sf rightmax}_{c \lambda}(q) - (c \lambda)_q  & = {\sf rightmax}_{c \lambda}(q+1) - {\sf leftmin}_{c \lambda}(q) \\
& \geq \beta_1 + \cdots + \beta_{{\sf center}_{c \lambda}(i,j)} - ((c \lambda)_1 + \cdots + (c \lambda)_{{\sf center}_{c \lambda}(i,j)}) \\
& \,\,\,\, + (\beta_{{{\sf center}_{c \lambda}(i,j)+1}} - (c \lambda)_{{{\sf center}_{c \lambda}(i,j)+1}}) + \cdots + (\beta_{j-1} - (c \lambda)_{j-1}) \\
& \,\,\,\, + (\beta_q \!-\! (q-p) \!-\! (c \lambda)_j) \!+\! (\beta_{j+1} \!-\! (c \lambda)_{j+1}) \!+\! \cdots \!+\! (\beta_{q-1}\! -\! (c \lambda)_{q-1}) \\
& \,\,\,\, + (\beta_j + (q-p) - (c \lambda)_q) \\
& \geq \beta_1 + \cdots + \beta_{{\sf center}_{c \lambda}(i,j)} - ((c \lambda)_1 + \cdots + (c \lambda)_{{\sf center}_{c \lambda}(i,j)})\\
&\ \ \ \  + (\beta_j + (q-p) - (c \lambda)_q) \\
& \geq \beta_i - (\beta_j - (j-i)) + (\beta_j + (q-p) - (c \lambda)_q) \\
& = \beta_i + (q-i) - (c \lambda)_q,
\end{split}
\end{align}
where the first equality follows from the subcase hypotheses, the second inequality from Lemma~\ref{lemma:overflowintorightmax} with ${\sf rightmax}_{c \lambda}(q+1) - {\sf leftmin}_{c \lambda}(q) \geq 0$ applied to $\tau$, the third from \eqref{eq:leftmin1.2.3} and Lemma~\ref{lem:lminrmax}(i), the fourth by \eqref{eq:thistookmethreedaystorealize1.2.3}, and the final by $p=j$. Hence, \eqref{eq:main1.2.3} implies ${\sf rightmax}_{c \lambda}(q) \geq \beta_i + (q-i) = \tau_i + (q-p)$. By \eqref{eq:betapqleft}, ${\sf leftmin}_{c \lambda}(p) \leq \beta_q - (q-p) = \tau_q - (q-p)$. We conclude by Proposition~\ref{prop:goingup} applied to $t_{p,q} \tau$ that $t_{p,q} \tau \in {\mathcal P}_{c \lambda,{\bf \gamma}}$ (notice $(p,q)$ cannot be interweaved since $j=p$ and
$c\lambda$ is $012$-avoiding by Lemma~\ref{lem:clampattern}).

\noindent \textit{Case 1.3 ($i$ and $p$ in the same block, $i<p$, $j = q$):} Lemma~\ref{lemma:posetBruhatorder}(i) implies $\beta_i \neq \beta_p - (p - i)$.

\noindent \textit{Subcase 1.3.1 $\beta_i > \beta_p - (p - i)$:} It is easily checked that $t_{p,q} \beta <_{\text{Bruhat}} t_{i,p} t_{p,q} \beta <_{\text{Bruhat}} t_{p,q} t_{i,p} t_{p,q} \beta = t_{i,j} \beta$. Hence $t_{i,j} \beta$ is not a cover of $\beta$ and this subcase cannot occur.

\noindent \textit{Subcase 1.3.2 $\beta_i < \beta_p - (p - i)$:} By the subcase hypothesis, the definition of $\tau$, and $j=q$, 
\begin{equation}
\label{eq:prophypofirst:1.3.2}
\tau_p = \beta_p > \beta_i + (p - i) = \beta_i + (j-i) - (q-p) = \tau_j - (q - p) = \tau_q - (q - p).
\end{equation}

By \eqref{eq:betaijineq}, \eqref{eq:betapqleft}, and $j=q$ we have
\begin{equation}
\label{eq:prophyposecond:1.3.2}
\tau_q - (q - p) = \tau_j - (q-p) = \beta_i + (j-i) - (q-p) > \beta_j - (q - p) = \beta_q + (q - p) \geq {\sf leftmin}_{c \lambda}(q).
\end{equation}
Since $\tau_p = \beta_p$, by \eqref{eq:betapqright},
\begin{equation}
\label{eq:prophypothird:1.3.2}
\tau_p + (q-p) = \beta_p + (q-p) \leq {\sf rightmax}_{c \lambda}(q).
\end{equation}

Finally, $\tau_r = \beta_r$ for $r \neq i , j$. If $(p,q)$ is interweaved, then \eqref{eq:betapqcenter} and $j=q$ combined with the previous sentence implies
\begin{align}
\label{eq:prophypofourth:1.3.2}
\begin{split}
\small
\tau_1 + \cdots + \tau_{p-1} + \tau_q - (q - p) + & \\
 \tau_{p+1} + \cdots + \tau_{{\sf center}_{c \lambda}(p,q)} &  = \beta_1 + \cdots + \beta_{i-1} + \beta_j - (j-i) + \beta_{i+1} + \cdots \\
& \,\,\,\, + \beta_{p-1} \!+\! \beta_i \!+\! (j-i) - (q-p) \!+\! \beta_{p+1} \!+\! \cdots \!+ \!\beta_{{\sf center}_{c \lambda}(i,j)} \\
& = \beta_1 \!+\! \cdots \!+\! \beta_{p-1} \!+\! \beta_j - (q - p) \!+\! \beta_{p+1} \!+\! \cdots \!+\! \beta_{{\sf center}_{c \lambda}(i,j)} \\
& \geq (c \lambda)_1 + \cdots + (c \lambda)_{{\sf center}_{c \lambda}(i,j)}. \\
\end{split}
\end{align}
The hypotheses of Proposition~\ref{prop:goingup} are satisfied for $t_{pq} \tau$ by \eqref{eq:prophypofirst:1.3.2}, \eqref{eq:prophyposecond:1.3.2}, \eqref{eq:prophypothird:1.3.2}, and \eqref{eq:prophypofourth:1.3.2}. Hence, $t_{p, q} \tau \in {\mathcal P}_{c \lambda,{\bf \gamma}}$. 

We conclude by \eqref{eq:prophypofirst:1.1.1} that $\tau <_{\text{Bruhat}} t_{p,q} \tau$. By \eqref{eq:betaijineq}, $\phi_i = \beta_i > \beta_j - (j-i) = \beta_q - (q-p) - (p-i) = \phi_p - (p-i)$, and hence $\phi = t_{p,q} \beta <_{\text{Bruhat}} t_{i,p} t_{p,q} \beta = t_{p, q} \tau$.

\noindent \textit{Case 1.4 ($i<p<j<q$ are all disjoint):} In this case $\tau, \phi <_{\text{Bruhat}} t_{i, j} t_{p,q} \beta = t_{p,q} t_{i, j} \beta$. By Lemma~\ref{lem:notbothinterweaved}, at least one of $(i,j)$ or $(p,q)$ is not interweaved. If $(i,j)$ is not interweaved then it follows from applying
Proposition~\ref{prop:goingup} to $t_{i,j}\beta\in {\mathcal P}_{c\lambda,\gamma}$ 
 that $t_{i, j} t_{p,q} \beta$ satisfies the hypotheses of Proposition~\ref{prop:goingup}. Similarly, if $(p, q)$ is not interweaved, $t_{p, q} t_{i,j} \beta$ is shown to satisfy the hypotheses of Proposition~\ref{prop:goingup}.

\noindent \textit{Case 1.5 ($i<j<p<q$ are all disjoint):} Once again $\tau, \phi <_{\text{Bruhat}} t_{i, j} t_{p,q} \beta = t_{p,q} t_{i, j} \beta$. It is easy to check that $t_{i, j} t_{p,q} \beta$ satisfies the hypotheses of Proposition~\ref{prop:goingup}. \qed

\section{Proof of Theorem~\ref{thm:maingoal} ($\Leftarrow$)} \label{sec:7}
Let us restate the ``$\Leftarrow$'' direction of Theorem~\ref{thm:maingoal}:

\begin{proposition}
Let $w\in {\mathfrak S}_n$, $I\subset J(w)$ and $D=[n-1]-I$ where $w$ is not $I$-spherical. There exists $\lambda\in{\sf Par}_D$ such that $\kappa_{w\lambda}$ is not $D$-multiplicity-free.
\end{proposition}
\begin{proof}
Let $u=w_0(I)\cdot w$. Since $w$ is not $I$-spherical, by Definition~\ref{def:spherical-standard-coxeter}, $u$ is not a product of distinct generators. By Proposition~\ref{prop:Sn-coxeter-pattern}, $u$ contains $321$ or $3412$. We divide our analysis into cases based on the patterns contained in $u$. 
For $\mu\in {\sf Comp}_n$ write $\mu|_D=(\mu^1,\ldots,\mu^k)$ to denote the splitting of $\mu$ into blocks of sizes $d_1-d_0,\ldots,d_{k+1}-d_k=n-d_k$. Note that $\mu|_D\in{\sf Par}_D$ if it is weakly decreasing in each block.

\noindent
\textit{Case 1 ($u$ contains the pattern $321$):} Choose the partition $\lambda$ whose parts are in $\{2,1,0\}$ so that $u\lambda$ contains the values $0,1,2$ at indices $p'<q<r'$. Choose the pattern $012$ so that $r'-p'$ is minimized. Also choose the minimum $p\leq p'$ such that $u\lambda$ contains only $0$'s at indices $p,\ldots,p'$ and choose the maximum $r\geq r'$ such that $u\lambda$ contains only $2$'s at indices $r',\ldots,r$. An example of a skyline diagram of $u\lambda$ is shown in Figure~\ref{fig:skyline-012}.
\begin{figure}[h!]
\begin{tikzpicture}[scale=0.500000000000000]
\node at (0,0) {$\cdot$};
\node at (1,0) {$\cdot$};
\node at (2,0) {$\cdot$};
\node at (3,0) {$\cdot$};
\node at (4,0) {$\cdot$};
\node at (5,0) {$\cdot$};
\node at (6,0) {$\cdot$};
\node at (7,0) {$\cdot$};
\node at (8,0) {$\cdot$};
\node at (9,0) {$\cdot$};
\node at (10,0) {$\cdot$};
\node at (11,0) {$\cdot$};
\node at (12,0) {$\cdot$};
\node at (13,0) {$\cdot$};
\node at (14,0) {$\cdot$};
\node at (15,0) {$\cdot$};
\node at (16,0) {$\cdot$};
\draw(0,0)--(1,0)--(1,1)--(0,1)--(0,0);
\draw(1,0)--(2,0);
\draw(2,0)--(3,0);
\draw(3,0)--(4,0);
\draw(4,0)--(5,0)--(5,1)--(4,1)--(4,0);
\draw(4,1)--(5,1)--(5,2)--(4,2)--(4,1);
\draw(5,0)--(6,0)--(6,1)--(5,1)--(5,0);
\draw(5,1)--(6,1)--(6,2)--(5,2)--(5,1);
\draw(6,0)--(7,0)--(7,1)--(6,1)--(6,0);
\draw(7,0)--(8,0)--(8,1)--(7,1)--(7,0);
\draw(8,0)--(9,0)--(9,1)--(8,1)--(8,0);
\draw(9,0)--(10,0);
\draw(10,0)--(11,0);
\draw(11,0)--(12,0);
\draw(12,0)--(13,0)--(13,1)--(12,1)--(12,0);
\draw(12,1)--(13,1)--(13,2)--(12,2)--(12,1);
\draw(13,0)--(14,0)--(14,1)--(13,1)--(13,0);
\draw(13,1)--(14,1)--(14,2)--(13,2)--(13,1);
\draw(14,0)--(15,0)--(15,1)--(14,1)--(14,0);
\draw(14,1)--(15,1)--(15,2)--(14,2)--(14,1);
\draw(15,0)--(16,0)--(16,1)--(15,1)--(15,0);
\node at (0.500000000000000,-0.4) {$ $};
\node at (1.50000000000000,-0.4) {$p $};
\node at (2.50000000000000,-0.4) {$ $};
\node at (3.50000000000000,-0.4) {$p' $};
\node at (4.50000000000000,-0.4) {$ $};
\node at (5.50000000000000,-0.4) {$ $};
\node at (6.50000000000000,-0.4) {$ $};
\node at (7.50000000000000,-0.4) {$q $};
\node at (8.50000000000000,-0.4) {$ $};
\node at (9.50000000000000,-0.4) {$ $};
\node at (10.5000000000000,-0.4) {$ $};
\node at (11.5000000000000,-0.4) {$ $};
\node at (12.5000000000000,-0.4) {$r' $};
\node at (13.5000000000000,-0.4) {$ $};
\node at (14.5000000000000,-0.4) {$r $};
\node at (15.5000000000000,-0.4) {$ $};
\end{tikzpicture}
\caption{A skyline diagram for $u\lambda$ that contains $012$}
\label{fig:skyline-012}
\end{figure}
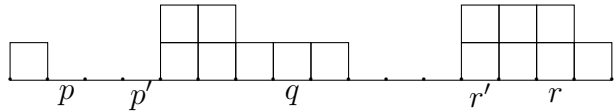
Here, $(u\lambda)_{p'}=0$, $(u\lambda)_{q}=1$, $(u\lambda)_{r'}=2$. In the interval $[p'+1,q]$,
$u\lambda$ can take on $1$'s or $2$'s, and all the $2$'s are left of the $1$'s by minimality of $r'-p'$. Similarly, in the interval $[q,r'-1]$,
$u\lambda$ takes on values $1$'s followed by $0$'s. Thus, in the interval $[p'+1,r'-1]$, say
$u\lambda$ takes on $k_2\geq0$ many $2$'s, then $k_1\geq1$ many $1$'s and then $k_0\geq0$ many $0$'s, and $(u\lambda)_{q}=1$.

Since $I\subset J(w)$, $D=[n-1]-I$, $w\lambda$ is weakly increasing in each block so $u\lambda$ is weakly decreasing in each block, \emph{i.e.}, $(u\lambda)|_D\in{\sf Par}_D$.  The argument that
follows only uses this property of $D$.

Consider the following composition
\[\gamma=(\gamma^1,\ldots,\gamma^k)=(u\lambda+\vec e_p- \vec e_r)|_D.\]
It is easily checked that if $(u\lambda)_i\geq (u\lambda)_{i+1}$, then $\gamma_i\geq\gamma_{i+1}$ by our choice of $p$ and $r$. Thus each $\gamma^i$ is indeed a partition, meaning that $\gamma\in{\sf Par}_D$.

Recall the poset ${\mathcal P}_{u\lambda,\gamma}$ (Section~\ref{sec:5}) contains all vectors $\beta$ such that the monomial $x^{\beta}$ appears in the expansion of $\kappa_{u\lambda}$ and $\pi_{w_0(I)}x^{\beta}=\pm s_{\gamma}$ (see Lemma~\ref{lemma:containsall}). By Lemma~\ref{claim:Mar4bbb}, ${\mathcal P}_{u\lambda,\gamma}$ is an order ideal  in ${\mathcal S}_{I,\gamma}$. Also each element $\beta$ can be generated from $\gamma$ via the moves $t_{ij}$'s. 

\begin{claim}\label{claim:height1}
${\mathcal P}_{u\lambda,\gamma}$ has height at most $1$. Moreover it has at most $k_1-1$ many
$\beta$ such that $\theta(\beta)=1$.
\end{claim}
\noindent
\emph{Proof of Claim~\ref{claim:height1}:}
Since all part sizes of $u\lambda$ belong in $\{0,1,2\}$, it is straightforward from
Lemma~\ref{lemma:posetBruhatorder}(i) that the only $t_{ij}$'s that increase the rank of $\beta$ are \[t_i:(\ldots,1,1,\ldots)\mapsto(\ldots,0,2,\ldots)\] for $i$ and $i+1$ in the same block. The number of nonzero values in the composition decreases by one when we apply such a move. Let $\#_{\neq0}\beta$ be the number of nonzero values in $\beta$. By Kohnert's rule (Theorem~\ref{thm:Kohnert}), $\#_{\neq0}\beta\geq\#_{\neq0}u\lambda$ for $[x^{\beta}]\kappa_{u\lambda}>0$. At the same time, $\#_{\neq0}\gamma=\#_{\neq0}u\lambda+1$, meaning that for all $\beta\in {\mathcal P}_{u\lambda,\gamma}$, $\beta$ can be obtained from $\gamma$ via at most one such move $t_i$. 

Next, let $\beta=t_i\gamma\in{\mathcal P}_{u\lambda,\gamma}$. Since $\beta\geq_{\sf dom}u\lambda$, by Corollary~\ref{cor:weightsdom}, we necessarily have $p'<i<r'$ so $i$ is one of $r'+k_2+1,\ldots,r'+k_2+k_1-1$ such that $i$ and $i+1$ are in the same block. Thus, there are at most $k_1-1$ choices for $i$.
\qed

\begin{claim}\label{claim:betaisone}
If $\beta\in {\mathcal P}_{u\lambda,\gamma}$ and $\theta(\beta)=1$ then
$[x^\beta]\kappa_{u\lambda}=1$.
\end{claim}
\noindent\emph{Proof of Claim~\ref{claim:betaisone}:}
 For each such $\beta=t_{i}\gamma$, there is exactly one corresponding Kohnert diagram, as we need to move the top box in column $r$ of $u\lambda$ to column $i+1$, and the single box in column $i$ of $u\lambda$ to column $p$. An example of such Kohnert diagrams corresponding to the example in Figure~\ref{fig:skyline-012} is shown in Figure~\ref{fig:skyline-012-negative}.\qed
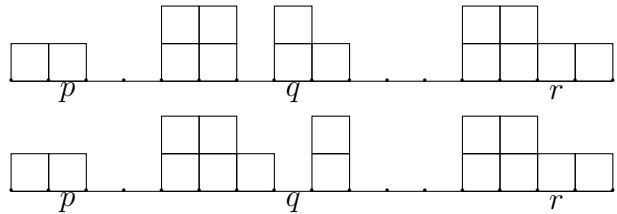
\begin{figure}[h!]
\begin{tikzpicture}[scale=0.500000000000000]
\node at (0,0) {$\cdot$};
\node at (1,0) {$\cdot$};
\node at (2,0) {$\cdot$};
\node at (3,0) {$\cdot$};
\node at (4,0) {$\cdot$};
\node at (5,0) {$\cdot$};
\node at (6,0) {$\cdot$};
\node at (7,0) {$\cdot$};
\node at (8,0) {$\cdot$};
\node at (9,0) {$\cdot$};
\node at (10,0) {$\cdot$};
\node at (11,0) {$\cdot$};
\node at (12,0) {$\cdot$};
\node at (13,0) {$\cdot$};
\node at (14,0) {$\cdot$};
\node at (15,0) {$\cdot$};
\node at (16,0) {$\cdot$};
\draw(0,0)--(1,0);
\draw(0,0)--(1,0)--(1,1)--(0,1)--(0,0);
\draw(1,0)--(2,0);
\draw(1,0)--(2,0)--(2,1)--(1,1)--(1,0);
\draw(2,0)--(3,0);
\draw(3,0)--(4,0);
\draw(4,0)--(5,0);
\draw(4,0)--(5,0)--(5,1)--(4,1)--(4,0);
\draw(4,1)--(5,1)--(5,2)--(4,2)--(4,1);
\draw(5,0)--(6,0);
\draw(5,0)--(6,0)--(6,1)--(5,1)--(5,0);
\draw(5,1)--(6,1)--(6,2)--(5,2)--(5,1);
\draw(6,0)--(7,0);
\draw(7,0)--(8,0);
\draw(7,0)--(8,0)--(8,1)--(7,1)--(7,0);
\draw(7,1)--(8,1)--(8,2)--(7,2)--(7,1);
\draw(8,0)--(9,0);
\draw(8,0)--(9,0)--(9,1)--(8,1)--(8,0);
\draw(9,0)--(10,0);
\draw(10,0)--(11,0);
\draw(11,0)--(12,0);
\draw(12,0)--(13,0);
\draw(12,0)--(13,0)--(13,1)--(12,1)--(12,0);
\draw(12,1)--(13,1)--(13,2)--(12,2)--(12,1);
\draw(13,0)--(14,0);
\draw(13,0)--(14,0)--(14,1)--(13,1)--(13,0);
\draw(13,1)--(14,1)--(14,2)--(13,2)--(13,1);
\draw(14,0)--(15,0);
\draw(14,0)--(15,0)--(15,1)--(14,1)--(14,0);
\draw(15,0)--(16,0);
\draw(15,0)--(16,0)--(16,1)--(15,1)--(15,0);
\node at (0.500000000000000,-0.3) {$ $};
\node at (1.50000000000000,-0.3) {$p $};
\node at (2.50000000000000,-0.3) {$ $};
\node at (3.50000000000000,-0.3) {$ $};
\node at (4.50000000000000,-0.3) {$ $};
\node at (5.50000000000000,-0.3) {$ $};
\node at (6.50000000000000,-0.3) {$ $};
\node at (7.50000000000000,-0.3) {$q $};
\node at (8.50000000000000,-0.3) {$ $};
\node at (9.50000000000000,-0.3) {$ $};
\node at (10.5000000000000,-0.3) {$ $};
\node at (11.5000000000000,-0.3) {$ $};
\node at (12.5000000000000,-0.3) {$ $};
\node at (13.5000000000000,-0.3) {$ $};
\node at (14.5000000000000,-0.3) {$r $};
\node at (15.5000000000000,-0.3) {$ $};
\end{tikzpicture}
\begin{tikzpicture}[scale=0.500000000000000]
\node at (0,0) {$\cdot$};
\node at (1,0) {$\cdot$};
\node at (2,0) {$\cdot$};
\node at (3,0) {$\cdot$};
\node at (4,0) {$\cdot$};
\node at (5,0) {$\cdot$};
\node at (6,0) {$\cdot$};
\node at (7,0) {$\cdot$};
\node at (8,0) {$\cdot$};
\node at (9,0) {$\cdot$};
\node at (10,0) {$\cdot$};
\node at (11,0) {$\cdot$};
\node at (12,0) {$\cdot$};
\node at (13,0) {$\cdot$};
\node at (14,0) {$\cdot$};
\node at (15,0) {$\cdot$};
\node at (16,0) {$\cdot$};
\draw(0,0)--(1,0);
\draw(0,0)--(1,0)--(1,1)--(0,1)--(0,0);
\draw(1,0)--(2,0);
\draw(1,0)--(2,0)--(2,1)--(1,1)--(1,0);
\draw(2,0)--(3,0);
\draw(3,0)--(4,0);
\draw(4,0)--(5,0);
\draw(4,0)--(5,0)--(5,1)--(4,1)--(4,0);
\draw(4,1)--(5,1)--(5,2)--(4,2)--(4,1);
\draw(5,0)--(6,0);
\draw(5,0)--(6,0)--(6,1)--(5,1)--(5,0);
\draw(5,1)--(6,1)--(6,2)--(5,2)--(5,1);
\draw(6,0)--(7,0);
\draw(6,0)--(7,0)--(7,1)--(6,1)--(6,0);
\draw(7,0)--(8,0);
\draw(8,0)--(9,0);
\draw(8,0)--(9,0)--(9,1)--(8,1)--(8,0);
\draw(8,1)--(9,1)--(9,2)--(8,2)--(8,1);
\draw(9,0)--(10,0);
\draw(10,0)--(11,0);
\draw(11,0)--(12,0);
\draw(12,0)--(13,0);
\draw(12,0)--(13,0)--(13,1)--(12,1)--(12,0);
\draw(12,1)--(13,1)--(13,2)--(12,2)--(12,1);
\draw(13,0)--(14,0);
\draw(13,0)--(14,0)--(14,1)--(13,1)--(13,0);
\draw(13,1)--(14,1)--(14,2)--(13,2)--(13,1);
\draw(14,0)--(15,0);
\draw(14,0)--(15,0)--(15,1)--(14,1)--(14,0);
\draw(15,0)--(16,0);
\draw(15,0)--(16,0)--(16,1)--(15,1)--(15,0);
\node at (0.500000000000000,-0.3) {$ $};
\node at (1.50000000000000,-0.3) {$p $};
\node at (2.50000000000000,-0.3) {$ $};
\node at (3.50000000000000,-0.3) {$ $};
\node at (4.50000000000000,-0.3) {$ $};
\node at (5.50000000000000,-0.3) {$ $};
\node at (6.50000000000000,-0.3) {$ $};
\node at (7.50000000000000,-0.3) {$q $};
\node at (8.50000000000000,-0.3) {$ $};
\node at (9.50000000000000,-0.3) {$ $};
\node at (10.5000000000000,-0.3) {$ $};
\node at (11.5000000000000,-0.3) {$ $};
\node at (12.5000000000000,-0.3) {$ $};
\node at (13.5000000000000,-0.3) {$ $};
\node at (14.5000000000000,-0.3) {$r $};
\node at (15.5000000000000,-0.3) {$ $};
\end{tikzpicture}
\caption{Kohnert diagrams with weight $x^{\beta}=x^{t_{ij}\gamma}$ where $\beta\in{\mathcal P}_{u\lambda,\gamma}$ }
\label{fig:skyline-012-negative}
\end{figure}

\begin{claim}\label{claim:overcount}
$[x^{\gamma}]\kappa_{u\lambda}=k_1+1$.
\end{claim}
\noindent
\emph{Proof of Claim~\ref{claim:overcount}:}
The $D\in {\sf Koh}(u\lambda)$ such that ${\sf Kohwt}(D)=\gamma$ are obtained by either
\begin{itemize}
\item moving the top box of column $r$ in $u\lambda$ moved to column $p$; or
\item moving the unique box in the column $z\in \{p'+k_2+1,\ldots,p'+k_2+k_1\}$ to column $p$ followed by moving the top box in column $r$ to column $z$.
\end{itemize}
These Kohnert diagrams corresponding to the example shown in Figure~\ref{fig:skyline-012} are shown in Figure~\ref{fig:skyline-012-positive}.\qed 

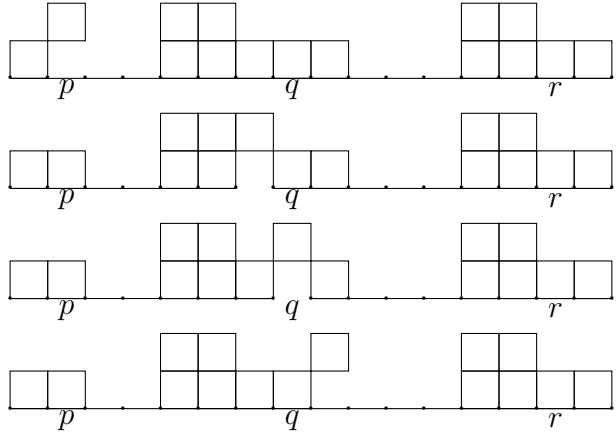
\begin{figure}[h!]
\begin{tikzpicture}[scale=0.500000000000000]
\node at (0,0) {$\cdot$};
\node at (1,0) {$\cdot$};
\node at (2,0) {$\cdot$};
\node at (3,0) {$\cdot$};
\node at (4,0) {$\cdot$};
\node at (5,0) {$\cdot$};
\node at (6,0) {$\cdot$};
\node at (7,0) {$\cdot$};
\node at (8,0) {$\cdot$};
\node at (9,0) {$\cdot$};
\node at (10,0) {$\cdot$};
\node at (11,0) {$\cdot$};
\node at (12,0) {$\cdot$};
\node at (13,0) {$\cdot$};
\node at (14,0) {$\cdot$};
\node at (15,0) {$\cdot$};
\node at (16,0) {$\cdot$};
\draw(0,0)--(1,0);
\draw(0,0)--(1,0)--(1,1)--(0,1)--(0,0);
\draw(1,0)--(2,0);
\draw(1,1)--(2,1)--(2,2)--(1,2)--(1,1);
\draw(2,0)--(3,0);
\draw(3,0)--(4,0);
\draw(4,0)--(5,0);
\draw(4,0)--(5,0)--(5,1)--(4,1)--(4,0);
\draw(4,1)--(5,1)--(5,2)--(4,2)--(4,1);
\draw(5,0)--(6,0);
\draw(5,0)--(6,0)--(6,1)--(5,1)--(5,0);
\draw(5,1)--(6,1)--(6,2)--(5,2)--(5,1);
\draw(6,0)--(7,0);
\draw(6,0)--(7,0)--(7,1)--(6,1)--(6,0);
\draw(7,0)--(8,0);
\draw(7,0)--(8,0)--(8,1)--(7,1)--(7,0);
\draw(8,0)--(9,0);
\draw(8,0)--(9,0)--(9,1)--(8,1)--(8,0);
\draw(9,0)--(10,0);
\draw(10,0)--(11,0);
\draw(11,0)--(12,0);
\draw(12,0)--(13,0);
\draw(12,0)--(13,0)--(13,1)--(12,1)--(12,0);
\draw(12,1)--(13,1)--(13,2)--(12,2)--(12,1);
\draw(13,0)--(14,0);
\draw(13,0)--(14,0)--(14,1)--(13,1)--(13,0);
\draw(13,1)--(14,1)--(14,2)--(13,2)--(13,1);
\draw(14,0)--(15,0);
\draw(14,0)--(15,0)--(15,1)--(14,1)--(14,0);
\draw(15,0)--(16,0);
\draw(15,0)--(16,0)--(16,1)--(15,1)--(15,0);
\node at (0.500000000000000,-0.3) {$ $};
\node at (1.50000000000000,-0.3) {$p $};
\node at (2.50000000000000,-0.3) {$ $};
\node at (3.50000000000000,-0.3) {$ $};
\node at (4.50000000000000,-0.3) {$ $};
\node at (5.50000000000000,-0.3) {$ $};
\node at (6.50000000000000,-0.3) {$ $};
\node at (7.50000000000000,-0.3) {$q $};
\node at (8.50000000000000,-0.3) {$ $};
\node at (9.50000000000000,-0.3) {$ $};
\node at (10.5000000000000,-0.3) {$ $};
\node at (11.5000000000000,-0.3) {$ $};
\node at (12.5000000000000,-0.3) {$ $};
\node at (13.5000000000000,-0.3) {$ $};
\node at (14.5000000000000,-0.3) {$r $};
\node at (15.5000000000000,-0.3) {$ $};
\end{tikzpicture}
\begin{tikzpicture}[scale=0.500000000000000]
\node at (0,0) {$\cdot$};
\node at (1,0) {$\cdot$};
\node at (2,0) {$\cdot$};
\node at (3,0) {$\cdot$};
\node at (4,0) {$\cdot$};
\node at (5,0) {$\cdot$};
\node at (6,0) {$\cdot$};
\node at (7,0) {$\cdot$};
\node at (8,0) {$\cdot$};
\node at (9,0) {$\cdot$};
\node at (10,0) {$\cdot$};
\node at (11,0) {$\cdot$};
\node at (12,0) {$\cdot$};
\node at (13,0) {$\cdot$};
\node at (14,0) {$\cdot$};
\node at (15,0) {$\cdot$};
\node at (16,0) {$\cdot$};
\draw(0,0)--(1,0);
\draw(0,0)--(1,0)--(1,1)--(0,1)--(0,0);
\draw(1,0)--(2,0);
\draw(1,0)--(2,0)--(2,1)--(1,1)--(1,0);
\draw(2,0)--(3,0);
\draw(3,0)--(4,0);
\draw(4,0)--(5,0);
\draw(4,0)--(5,0)--(5,1)--(4,1)--(4,0);
\draw(4,1)--(5,1)--(5,2)--(4,2)--(4,1);
\draw(5,0)--(6,0);
\draw(5,0)--(6,0)--(6,1)--(5,1)--(5,0);
\draw(5,1)--(6,1)--(6,2)--(5,2)--(5,1);

\draw(6,1)--(7,1)--(7,2)--(6,2)--(6,1);
\draw(7,0)--(8,0);
\draw(7,0)--(8,0)--(8,1)--(7,1)--(7,0);
\draw(8,0)--(9,0);
\draw(8,0)--(9,0)--(9,1)--(8,1)--(8,0);
\draw(9,0)--(10,0);
\draw(10,0)--(11,0);
\draw(11,0)--(12,0);
\draw(12,0)--(13,0);
\draw(12,0)--(13,0)--(13,1)--(12,1)--(12,0);
\draw(12,1)--(13,1)--(13,2)--(12,2)--(12,1);
\draw(13,0)--(14,0);
\draw(13,0)--(14,0)--(14,1)--(13,1)--(13,0);
\draw(13,1)--(14,1)--(14,2)--(13,2)--(13,1);
\draw(14,0)--(15,0);
\draw(14,0)--(15,0)--(15,1)--(14,1)--(14,0);
\draw(15,0)--(16,0);
\draw(15,0)--(16,0)--(16,1)--(15,1)--(15,0);
\node at (0.500000000000000,-0.3) {$ $};
\node at (1.50000000000000,-0.3) {$p $};
\node at (2.50000000000000,-0.3) {$ $};
\node at (3.50000000000000,-0.3) {$ $};
\node at (4.50000000000000,-0.3) {$ $};
\node at (5.50000000000000,-0.3) {$ $};
\node at (6.50000000000000,-0.3) {$ $};
\node at (7.50000000000000,-0.3) {$q $};
\node at (8.50000000000000,-0.3) {$ $};
\node at (9.50000000000000,-0.3) {$ $};
\node at (10.5000000000000,-0.3) {$ $};
\node at (11.5000000000000,-0.3) {$ $};
\node at (12.5000000000000,-0.3) {$ $};
\node at (13.5000000000000,-0.3) {$ $};
\node at (14.5000000000000,-0.3) {$r $};
\node at (15.5000000000000,-0.3) {$ $};
\end{tikzpicture}

\begin{tikzpicture}[scale=0.500000000000000]
\node at (0,0) {$\cdot$};
\node at (1,0) {$\cdot$};
\node at (2,0) {$\cdot$};
\node at (3,0) {$\cdot$};
\node at (4,0) {$\cdot$};
\node at (5,0) {$\cdot$};
\node at (6,0) {$\cdot$};
\node at (7,0) {$\cdot$};
\node at (8,0) {$\cdot$};
\node at (9,0) {$\cdot$};
\node at (10,0) {$\cdot$};
\node at (11,0) {$\cdot$};
\node at (12,0) {$\cdot$};
\node at (13,0) {$\cdot$};
\node at (14,0) {$\cdot$};
\node at (15,0) {$\cdot$};
\node at (16,0) {$\cdot$};
\draw(0,0)--(1,0);
\draw(0,0)--(1,0)--(1,1)--(0,1)--(0,0);
\draw(1,0)--(2,0);
\draw(1,0)--(2,0)--(2,1)--(1,1)--(1,0);
\draw(2,0)--(3,0);
\draw(3,0)--(4,0);
\draw(4,0)--(5,0);
\draw(4,0)--(5,0)--(5,1)--(4,1)--(4,0);
\draw(4,1)--(5,1)--(5,2)--(4,2)--(4,1);
\draw(5,0)--(6,0);
\draw(5,0)--(6,0)--(6,1)--(5,1)--(5,0);
\draw(5,1)--(6,1)--(6,2)--(5,2)--(5,1);
\draw(6,0)--(7,0);
\draw(6,0)--(7,0)--(7,1)--(6,1)--(6,0);

\draw(7,1)--(8,1)--(8,2)--(7,2)--(7,1);
\draw(8,0)--(9,0);
\draw(8,0)--(9,0)--(9,1)--(8,1)--(8,0);
\draw(9,0)--(10,0);
\draw(10,0)--(11,0);
\draw(11,0)--(12,0);
\draw(12,0)--(13,0);
\draw(12,0)--(13,0)--(13,1)--(12,1)--(12,0);
\draw(12,1)--(13,1)--(13,2)--(12,2)--(12,1);
\draw(13,0)--(14,0);
\draw(13,0)--(14,0)--(14,1)--(13,1)--(13,0);
\draw(13,1)--(14,1)--(14,2)--(13,2)--(13,1);
\draw(14,0)--(15,0);
\draw(14,0)--(15,0)--(15,1)--(14,1)--(14,0);
\draw(15,0)--(16,0);
\draw(15,0)--(16,0)--(16,1)--(15,1)--(15,0);
\node at (0.500000000000000,-0.3) {$ $};
\node at (1.50000000000000,-0.3) {$p $};
\node at (2.50000000000000,-0.3) {$ $};
\node at (3.50000000000000,-0.3) {$ $};
\node at (4.50000000000000,-0.3) {$ $};
\node at (5.50000000000000,-0.3) {$ $};
\node at (6.50000000000000,-0.3) {$ $};
\node at (7.50000000000000,-0.3) {$q $};
\node at (8.50000000000000,-0.3) {$ $};
\node at (9.50000000000000,-0.3) {$ $};
\node at (10.5000000000000,-0.3) {$ $};
\node at (11.5000000000000,-0.3) {$ $};
\node at (12.5000000000000,-0.3) {$ $};
\node at (13.5000000000000,-0.3) {$ $};
\node at (14.5000000000000,-0.3) {$r $};
\node at (15.5000000000000,-0.3) {$ $};
\end{tikzpicture}

\begin{tikzpicture}[scale=0.500000000000000]
\node at (0,0) {$\cdot$};
\node at (1,0) {$\cdot$};
\node at (2,0) {$\cdot$};
\node at (3,0) {$\cdot$};
\node at (4,0) {$\cdot$};
\node at (5,0) {$\cdot$};
\node at (6,0) {$\cdot$};
\node at (7,0) {$\cdot$};
\node at (8,0) {$\cdot$};
\node at (9,0) {$\cdot$};
\node at (10,0) {$\cdot$};
\node at (11,0) {$\cdot$};
\node at (12,0) {$\cdot$};
\node at (13,0) {$\cdot$};
\node at (14,0) {$\cdot$};
\node at (15,0) {$\cdot$};
\node at (16,0) {$\cdot$};
\draw(0,0)--(1,0);
\draw(0,0)--(1,0)--(1,1)--(0,1)--(0,0);
\draw(1,0)--(2,0);
\draw(1,0)--(2,0)--(2,1)--(1,1)--(1,0);
\draw(2,0)--(3,0);
\draw(3,0)--(4,0);
\draw(4,0)--(5,0);
\draw(4,0)--(5,0)--(5,1)--(4,1)--(4,0);
\draw(4,1)--(5,1)--(5,2)--(4,2)--(4,1);
\draw(5,0)--(6,0);
\draw(5,0)--(6,0)--(6,1)--(5,1)--(5,0);
\draw(5,1)--(6,1)--(6,2)--(5,2)--(5,1);
\draw(6,0)--(7,0);
\draw(6,0)--(7,0)--(7,1)--(6,1)--(6,0);
\draw(7,0)--(8,0);
\draw(7,0)--(8,0)--(8,1)--(7,1)--(7,0);
\draw(8,0)--(9,0);
\draw(8,1)--(9,1)--(9,2)--(8,2)--(8,1);
\draw(9,0)--(10,0);
\draw(10,0)--(11,0);
\draw(11,0)--(12,0);
\draw(12,0)--(13,0);
\draw(12,0)--(13,0)--(13,1)--(12,1)--(12,0);
\draw(12,1)--(13,1)--(13,2)--(12,2)--(12,1);
\draw(13,0)--(14,0);
\draw(13,0)--(14,0)--(14,1)--(13,1)--(13,0);
\draw(13,1)--(14,1)--(14,2)--(13,2)--(13,1);
\draw(14,0)--(15,0);
\draw(14,0)--(15,0)--(15,1)--(14,1)--(14,0);
\draw(15,0)--(16,0);
\draw(15,0)--(16,0)--(16,1)--(15,1)--(15,0);
\node at (0.500000000000000,-0.3) {$ $};
\node at (1.50000000000000,-0.3) {$p $};
\node at (2.50000000000000,-0.3) {$ $};
\node at (3.50000000000000,-0.3) {$ $};
\node at (4.50000000000000,-0.3) {$ $};
\node at (5.50000000000000,-0.3) {$ $};
\node at (6.50000000000000,-0.3) {$ $};
\node at (7.50000000000000,-0.3) {$q $};
\node at (8.50000000000000,-0.3) {$ $};
\node at (9.50000000000000,-0.3) {$ $};
\node at (10.5000000000000,-0.3) {$ $};
\node at (11.5000000000000,-0.3) {$ $};
\node at (12.5000000000000,-0.3) {$ $};
\node at (13.5000000000000,-0.3) {$ $};
\node at (14.5000000000000,-0.3) {$r $};
\node at (15.5000000000000,-0.3) {$ $};
\end{tikzpicture}
\caption{Kohnert diagrams with weight $x^{\gamma}$}
\label{fig:skyline-012-positive}
\end{figure}

Hence, by Claims~\ref{claim:height1},~\ref{claim:betaisone},~\ref{claim:overcount},
\[
[s_{\gamma}]\kappa_{w\lambda}=\sum_{\beta\in{\mathcal P}_{u\lambda,\gamma}}{\sf sgn}(\beta)[x^{\beta}]\kappa_{u\lambda}\geq (k_1+1)-(k_1-1)=2
\]
so $\kappa_{w\lambda}$ is not $D$-multiplicity-free.

\noindent
\textit{Case 2 ($u$ avoids the pattern $321$ but $u$ contains the pattern $3412$):} Pick $\lambda\in {\sf Par}_n$ to consist of values in $\{3,2,1,0\}$ so that $u\lambda$ contains the values $1,0,3,2$ at indices $p'<q'<r'<z'$ so that $z'-p'$ is minizied. Analogous to Case 1, choose the minimum $p\leq p'$ such that $u\lambda$ contains only $1$'s in the interval $[p,p']$ and choose the maximum $z\geq z'$ such that $u\lambda$ contains only $2$'s on $[z',z]$. Let $q>p$ be the minimum index such that $(u\lambda)_{q}=0$ and let $r<z$ be the maximum index such that $(u\lambda)_{r}=3$. Since $u$ avoids $321$, $u\lambda$ avoids $012$, and together with the minimality of $z'-p'$, we see that $(u\lambda)_{p'+1},\ldots,(u\lambda)_{z'-1}$ can only take on values in $\{0,3\}$. An example of a skyline diagram of $u\lambda$ is shown in Figure~\ref{fig:skyline-1032}.
\begin{figure}[h!]
\begin{tikzpicture}[scale=0.500000000000000]
\node at (0,0) {$\cdot$};
\node at (1,0) {$\cdot$};
\node at (2,0) {$\cdot$};
\node at (3,0) {$\cdot$};
\node at (4,0) {$\cdot$};
\node at (5,0) {$\cdot$};
\node at (6,0) {$\cdot$};
\node at (7,0) {$\cdot$};
\node at (8,0) {$\cdot$};
\node at (9,0) {$\cdot$};
\node at (10,0) {$\cdot$};
\node at (11,0) {$\cdot$};
\node at (12,0) {$\cdot$};
\node at (13,0) {$\cdot$};
\node at (14,0) {$\cdot$};
\node at (15,0) {$\cdot$};
\node at (16,0) {$\cdot$};
\node at (17,0) {$\cdot$};
\draw(0,0)--(1,0)--(1,1)--(0,1)--(0,0);
\draw(0,1)--(1,1)--(1,2)--(0,2)--(0,1);
\draw(0,2)--(1,2)--(1,3)--(0,3)--(0,2);
\draw(1,0)--(2,0)--(2,1)--(1,1)--(1,0);
\draw(2,0)--(3,0)--(3,1)--(2,1)--(2,0);
\draw(3,0)--(4,0);
\draw(4,0)--(5,0);
\draw(5,0)--(6,0)--(6,1)--(5,1)--(5,0);
\draw(5,1)--(6,1)--(6,2)--(5,2)--(5,1);
\draw(5,2)--(6,2)--(6,3)--(5,3)--(5,2);
\draw(6,0)--(7,0)--(7,1)--(6,1)--(6,0);
\draw(6,1)--(7,1)--(7,2)--(6,2)--(6,1);
\draw(6,2)--(7,2)--(7,3)--(6,3)--(6,2);
\draw(7,0)--(8,0);
\draw(8,0)--(9,0)--(9,1)--(8,1)--(8,0);
\draw(8,1)--(9,1)--(9,2)--(8,2)--(8,1);
\draw(8,2)--(9,2)--(9,3)--(8,3)--(8,2);
\draw(9,0)--(10,0);
\draw(10,0)--(11,0);
\draw(11,0)--(12,0)--(12,1)--(11,1)--(11,0);
\draw(11,1)--(12,1)--(12,2)--(11,2)--(11,1);
\draw(11,2)--(12,2)--(12,3)--(11,3)--(11,2);
\draw(12,0)--(13,0)--(13,1)--(12,1)--(12,0);
\draw(12,1)--(13,1)--(13,2)--(12,2)--(12,1);
\draw(12,2)--(13,2)--(13,3)--(12,3)--(12,2);
\draw(13,0)--(14,0)--(14,1)--(13,1)--(13,0);
\draw(13,1)--(14,1)--(14,2)--(13,2)--(13,1);
\draw(13,2)--(14,2)--(14,3)--(13,3)--(13,2);
\draw(14,0)--(15,0)--(15,1)--(14,1)--(14,0);
\draw(14,1)--(15,1)--(15,2)--(14,2)--(14,1);
\draw(15,0)--(16,0)--(16,1)--(15,1)--(15,0);
\draw(15,1)--(16,1)--(16,2)--(15,2)--(15,1);
\draw(16,0)--(17,0)--(17,1)--(16,1)--(16,0);
\node at (0.500000000000000,-0.4) {$ $};
\node at (1.50000000000000,-0.4) {$p $};
\node at (2.50000000000000,-0.4) {$p' $};
\node at (3.50000000000000,-0.4) {$q $};
\node at (4.50000000000000,-0.4) {$ $};
\node at (5.50000000000000,-0.4) {$ $};
\node at (6.50000000000000,-0.4) {$ $};
\node at (7.50000000000000,-0.4) {$q' $};
\node at (8.50000000000000,-0.4) {$ $};
\node at (9.50000000000000,-0.4) {$ $};
\node at (10.5000000000000,-0.4) {$ $};
\node at (11.5000000000000,-0.4) {$ $};
\node at (12.5000000000000,-0.4) {$r' $};
\node at (13.5000000000000,-0.4) {$r $};
\node at (14.5000000000000,-0.4) {$ $};
\node at (15.5000000000000,-0.4) {$z,z' $};
\node at (16.5000000000000,-0.4) {$ $};
\end{tikzpicture}
\caption{A skyline diagram for $u\lambda$ that contains $1032$ and avoids $012$ (possibly $z=z'$)}
\label{fig:skyline-1032}
\end{figure}
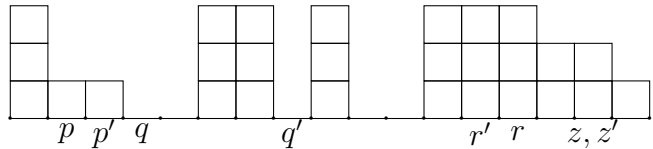

Similar to Case 1, let 
\[\gamma=(\gamma^1,\ldots,\gamma^k)=(u\lambda+\vec e_p+\vec e_q-\vec e_r-\vec e_z)|_D\in{\sf Par}_D.\] 

\begin{claim}
\label{claim:height0}
${\mathcal P}_{u\lambda,\gamma}=\{\gamma\}$.
\end{claim}
\noindent 
\emph{Proof of Claim~\ref{claim:height0}:} By Proposition~\ref{prop:posetsiso}, 
Lemma~\ref{lemma:commutes}, and Lemma~\ref{claim:Mar4bbb}, it suffices to show that there does not exist $i,i+1$ in the same block such that $\beta=t_i\gamma\in {\mathcal P}_{u\lambda,\gamma}$. If such a $t_i $ exists, then $[x^{\beta}]\kappa_{u\lambda}$ and so $\beta\geq_{\sf dom} u\lambda$, Corollary~\ref{cor:weightsdom}. Also we must have $p\leq i<z$ since $\gamma$ and $u\lambda$ only differ in that interval. Let 
$\beta_{\leq j}:=(\beta_1,\ldots,\beta_j)$ and recall that $\#_{\neq0}\beta$ is the number of nonzero entries in $\beta$. By Kohnert's rule, Theorem~\ref{thm:Kohnert}, for $\beta\in{\mathcal P}_{u\lambda,\gamma}$, $\#_{\neq0}\beta_{\leq j}\geq \#_{\neq0}(u\lambda)_{\leq j}$ for all $j$. Consider the following cases:
\begin{itemize}
\item $p=i<q$, $t_i:\gamma=(\ldots,2,1,\ldots)\mapsto(\ldots,0,3,\ldots)$, $\#_{\neq0}\beta_{\leq i}<\#_{\neq0}(u\lambda)_{\leq i}$;
\item $p<i<q$, $t_i:\gamma=(\ldots,1,1,\ldots)\mapsto(\ldots,0,2,\ldots)$, $\#_{\neq0}\beta_{\leq i}<\#_{\neq0}(u\lambda)_{\leq i}$;
\item $q\leq i<r$, $t_i:\gamma=(\ldots,1,0,\ldots)\mapsto(\ldots,-1,2,\ldots)$ or $(\ldots,3,\beta_{i+1},\ldots)\mapsto (\ldots,\beta_{i+1}-1,4,\ldots)$, with impossible part sizes;
\item $r\leq i<z$, $t_i:\gamma=(\ldots,2,2,\ldots)\mapsto(\ldots,1,3,\ldots)$ or $(\ldots,2,1,\ldots)\mapsto(\ldots,0,3,\ldots)$, where the newly generated part of size $3$ cannot be obtained by Kohnert's rule, Theorem~\ref{thm:Kohnert}, since $u\lambda$, $\gamma$ and $\beta$ only differ on the interval $[p,z]$, that is $\beta\not \in {\mathcal P}_{u\lambda,\gamma}$, a contradiction.
\end{itemize}
As a result, no such $t_i$ exists. \qed

\begin{claim}\label{claim:gammaistwo}
$[x^{\gamma}]\kappa_{u\lambda}=2$.
\end{claim}
\noindent 
\emph{Proof of Claim~\ref{claim:gammaistwo}:}
The $D\in {\sf Koh}(u\lambda)$ such that ${\sf Kohwt}(D)=\gamma$ are obtained from $u\lambda$ by 
\begin{itemize}
\item moving the top box of column $r$ to column $p$ and moving the top box of column $z$ to column $q$; or
\item moving the top box of column $r$ to column $q$ and moving the top box of column $z$ to column $p$;
\end{itemize}
as shown in Figure~\ref{fig:skyline-1032-positive}.\qed

Therefore, by Claim~\ref{claim:height0} and Claim~\ref{claim:gammaistwo}, $[s_{\lambda}]\kappa_{w\lambda}=[x^{\gamma}]\kappa_{u\lambda}=2$, as desired.
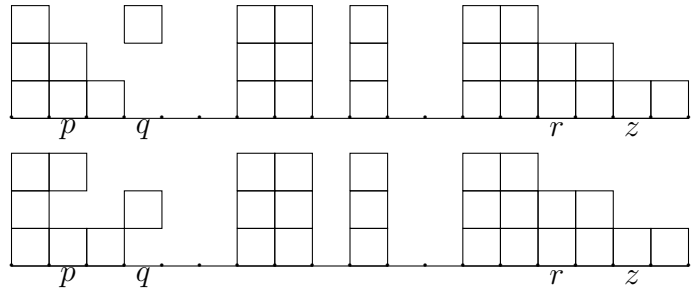
\begin{figure}[h!]
\begin{tikzpicture}[scale=0.500000000000000]
\node at (0,0) {$\cdot$};
\node at (1,0) {$\cdot$};
\node at (2,0) {$\cdot$};
\node at (3,0) {$\cdot$};
\node at (4,0) {$\cdot$};
\node at (5,0) {$\cdot$};
\node at (6,0) {$\cdot$};
\node at (7,0) {$\cdot$};
\node at (8,0) {$\cdot$};
\node at (9,0) {$\cdot$};
\node at (10,0) {$\cdot$};
\node at (11,0) {$\cdot$};
\node at (12,0) {$\cdot$};
\node at (13,0) {$\cdot$};
\node at (14,0) {$\cdot$};
\node at (15,0) {$\cdot$};
\node at (16,0) {$\cdot$};
\node at (17,0) {$\cdot$};
\node at (18,0) {$\cdot$};
\draw(0,0)--(1,0);
\draw(0,0)--(1,0)--(1,1)--(0,1)--(0,0);
\draw(0,1)--(1,1)--(1,2)--(0,2)--(0,1);
\draw(0,2)--(1,2)--(1,3)--(0,3)--(0,2);
\draw(1,0)--(2,0);
\draw(1,0)--(2,0)--(2,1)--(1,1)--(1,0);
\draw(1,1)--(2,1)--(2,2)--(1,2)--(1,1);
\draw(2,0)--(3,0);
\draw(2,0)--(3,0)--(3,1)--(2,1)--(2,0);
\draw(3,0)--(4,0);
\draw(3,2)--(4,2)--(4,3)--(3,3)--(3,2);
\draw(4,0)--(5,0);
\draw(5,0)--(6,0);
\draw(6,0)--(7,0);
\draw(6,0)--(7,0)--(7,1)--(6,1)--(6,0);
\draw(6,1)--(7,1)--(7,2)--(6,2)--(6,1);
\draw(6,2)--(7,2)--(7,3)--(6,3)--(6,2);
\draw(7,0)--(8,0);
\draw(7,0)--(8,0)--(8,1)--(7,1)--(7,0);
\draw(7,1)--(8,1)--(8,2)--(7,2)--(7,1);
\draw(7,2)--(8,2)--(8,3)--(7,3)--(7,2);
\draw(8,0)--(9,0);
\draw(9,0)--(10,0);
\draw(9,0)--(10,0)--(10,1)--(9,1)--(9,0);
\draw(9,1)--(10,1)--(10,2)--(9,2)--(9,1);
\draw(9,2)--(10,2)--(10,3)--(9,3)--(9,2);
\draw(10,0)--(11,0);
\draw(11,0)--(12,0);
\draw(12,0)--(13,0);
\draw(12,0)--(13,0)--(13,1)--(12,1)--(12,0);
\draw(12,1)--(13,1)--(13,2)--(12,2)--(12,1);
\draw(12,2)--(13,2)--(13,3)--(12,3)--(12,2);
\draw(13,0)--(14,0);
\draw(13,0)--(14,0)--(14,1)--(13,1)--(13,0);
\draw(13,1)--(14,1)--(14,2)--(13,2)--(13,1);
\draw(13,2)--(14,2)--(14,3)--(13,3)--(13,2);
\draw(14,0)--(15,0);
\draw(14,0)--(15,0)--(15,1)--(14,1)--(14,0);
\draw(14,1)--(15,1)--(15,2)--(14,2)--(14,1);
\draw(15,0)--(16,0);
\draw(15,0)--(16,0)--(16,1)--(15,1)--(15,0);
\draw(15,1)--(16,1)--(16,2)--(15,2)--(15,1);
\draw(16,0)--(17,0);
\draw(16,0)--(17,0)--(17,1)--(16,1)--(16,0);
\draw(17,0)--(18,0);
\draw(17,0)--(18,0)--(18,1)--(17,1)--(17,0);
\node at (0.500000000000000,-0.3) {$ $};
\node at (1.50000000000000,-0.3) {$p $};
\node at (2.50000000000000,-0.3) {$ $};
\node at (3.50000000000000,-0.3) {$q $};
\node at (4.50000000000000,-0.3) {$ $};
\node at (5.50000000000000,-0.3) {$ $};
\node at (6.50000000000000,-0.3) {$ $};
\node at (7.50000000000000,-0.3) {$ $};
\node at (8.50000000000000,-0.3) {$ $};
\node at (9.50000000000000,-0.3) {$ $};
\node at (10.5000000000000,-0.3) {$ $};
\node at (11.5000000000000,-0.3) {$ $};
\node at (12.5000000000000,-0.3) {$ $};
\node at (13.5000000000000,-0.3) {$ $};
\node at (14.5000000000000,-0.3) {$r $};
\node at (15.5000000000000,-0.3) {$ $};
\node at (16.5000000000000,-0.3) {$z $};
\node at (17.5000000000000,-0.3) {$ $};
\end{tikzpicture}

\begin{tikzpicture}[scale=0.500000000000000]
\node at (0,0) {$\cdot$};
\node at (1,0) {$\cdot$};
\node at (2,0) {$\cdot$};
\node at (3,0) {$\cdot$};
\node at (4,0) {$\cdot$};
\node at (5,0) {$\cdot$};
\node at (6,0) {$\cdot$};
\node at (7,0) {$\cdot$};
\node at (8,0) {$\cdot$};
\node at (9,0) {$\cdot$};
\node at (10,0) {$\cdot$};
\node at (11,0) {$\cdot$};
\node at (12,0) {$\cdot$};
\node at (13,0) {$\cdot$};
\node at (14,0) {$\cdot$};
\node at (15,0) {$\cdot$};
\node at (16,0) {$\cdot$};
\node at (17,0) {$\cdot$};
\node at (18,0) {$\cdot$};
\draw(0,0)--(1,0);
\draw(0,0)--(1,0)--(1,1)--(0,1)--(0,0);
\draw(0,1)--(1,1)--(1,2)--(0,2)--(0,1);
\draw(0,2)--(1,2)--(1,3)--(0,3)--(0,2);
\draw(1,0)--(2,0);
\draw(1,0)--(2,0)--(2,1)--(1,1)--(1,0);
\draw(1,2)--(2,2)--(2,3)--(1,3)--(1,2);
\draw(2,0)--(3,0);
\draw(2,0)--(3,0)--(3,1)--(2,1)--(2,0);
\draw(3,0)--(4,0);
\draw(3,1)--(4,1)--(4,2)--(3,2)--(3,1);
\draw(4,0)--(5,0);
\draw(5,0)--(6,0);
\draw(6,0)--(7,0);
\draw(6,0)--(7,0)--(7,1)--(6,1)--(6,0);
\draw(6,1)--(7,1)--(7,2)--(6,2)--(6,1);
\draw(6,2)--(7,2)--(7,3)--(6,3)--(6,2);
\draw(7,0)--(8,0);
\draw(7,0)--(8,0)--(8,1)--(7,1)--(7,0);
\draw(7,1)--(8,1)--(8,2)--(7,2)--(7,1);
\draw(7,2)--(8,2)--(8,3)--(7,3)--(7,2);
\draw(8,0)--(9,0);
\draw(9,0)--(10,0);
\draw(9,0)--(10,0)--(10,1)--(9,1)--(9,0);
\draw(9,1)--(10,1)--(10,2)--(9,2)--(9,1);
\draw(9,2)--(10,2)--(10,3)--(9,3)--(9,2);
\draw(10,0)--(11,0);
\draw(11,0)--(12,0);
\draw(12,0)--(13,0);
\draw(12,0)--(13,0)--(13,1)--(12,1)--(12,0);
\draw(12,1)--(13,1)--(13,2)--(12,2)--(12,1);
\draw(12,2)--(13,2)--(13,3)--(12,3)--(12,2);
\draw(13,0)--(14,0);
\draw(13,0)--(14,0)--(14,1)--(13,1)--(13,0);
\draw(13,1)--(14,1)--(14,2)--(13,2)--(13,1);
\draw(13,2)--(14,2)--(14,3)--(13,3)--(13,2);
\draw(14,0)--(15,0);
\draw(14,0)--(15,0)--(15,1)--(14,1)--(14,0);
\draw(14,1)--(15,1)--(15,2)--(14,2)--(14,1);
\draw(15,0)--(16,0);
\draw(15,0)--(16,0)--(16,1)--(15,1)--(15,0);
\draw(15,1)--(16,1)--(16,2)--(15,2)--(15,1);
\draw(16,0)--(17,0);
\draw(16,0)--(17,0)--(17,1)--(16,1)--(16,0);
\draw(17,0)--(18,0);
\draw(17,0)--(18,0)--(18,1)--(17,1)--(17,0);
\node at (0.500000000000000,-0.3) {$ $};
\node at (1.50000000000000,-0.3) {$p $};
\node at (2.50000000000000,-0.3) {$ $};
\node at (3.50000000000000,-0.3) {$q $};
\node at (4.50000000000000,-0.3) {$ $};
\node at (5.50000000000000,-0.3) {$ $};
\node at (6.50000000000000,-0.3) {$ $};
\node at (7.50000000000000,-0.3) {$ $};
\node at (8.50000000000000,-0.3) {$ $};
\node at (9.50000000000000,-0.3) {$ $};
\node at (10.5000000000000,-0.3) {$ $};
\node at (11.5000000000000,-0.3) {$ $};
\node at (12.5000000000000,-0.3) {$ $};
\node at (13.5000000000000,-0.3) {$ $};
\node at (14.5000000000000,-0.3) {$r $};
\node at (15.5000000000000,-0.3) {$ $};
\node at (16.5000000000000,-0.3) {$z $};
\node at (17.5000000000000,-0.3) {$ $};
\end{tikzpicture}
\caption{Kohnert diagrams with weight $x^{\gamma}$}
\label{fig:skyline-1032-positive}
\end{figure}
\end{proof}

\section*{Acknowledgements}
We thank David Brewster, Jiasheng Hu, and Husnain Raza for writing useful computer code (in the
NSF RTG funded ICLUE program). We also thank David Anderson, Mahir Can, Alexander Diaz-Lopez, Christian Gaetz, Megumi Harada, Bogdan Ion, Syu Kato, and Allen Knutson for stimulating conversations during the preparation of this work.
We used SageMath, as well as the Maple packages ACE and Coxeter/Weyl in our investigations.
This work was partially completed during (virtual) residence at ICERM's Spring 2021 semester
``Combinatorial Algebraic Geometry''; we thank the organizers providing an hospitable environment.
AY was partially supported by a Simons Collaboration Grant, and an NSF RTG grant.
RH was partially supported by an AMS-Simons Travel Grant.

\end{document}